\documentclass[10pt,letterpaper,leqno]{amsart}
\usepackage{amsmath}
\usepackage[latin9]{inputenc}
\usepackage{color}
\usepackage{amsthm}
\usepackage{amstext}
\usepackage{amssymb}
\usepackage{latexsym}
\usepackage{amscd}
\usepackage{amsfonts}
\usepackage{enumerate}
\usepackage[mathscr]{euscript}

\makeatletter

\numberwithin{equation}{section}
\numberwithin{figure}{section}

\newlength{\baseunit}

\newcount{\numlines}
\newtheorem{theorem}{Theorem}

\newtheorem{corollary}{Corollary}

\newtheorem{lemma}{Lemma}
\newtheorem{proposition}{Proposition}

\newtheorem{problem}{Problem}

\DeclareMathOperator{\dist}{dist}

\DeclareMathOperator{\norm}{Norm}
\newcommand{\LA}[1]{\refstepcounter{equation}\text{(\theequation)}\label{#1}}

\newcommand{\LAQ}[2]{\begin{itemize}\item[\LA{#1}]{#2} \end{itemize}}
\numberwithin{equation}{subsection}
\makeatother

\begin{document}
\title{Solutions to A System of Equations for $C^m$ Functions}
\date{\today }
\author{Charles Fefferman, Garving K. Luli}
\thanks{The first author is
supported in part by NSF Grant DMS-1608782, AFOSR Grant FA9550-12-1-0425,
and Grant No 2014055 from the United States-Israel Binational Science
Foundation. The second author is supported by NSF Grant DMS-1554733.}
\maketitle

\section{Introduction}

Here and in \cite{cf-luli-generator}, we study systems of linear equations 
\begin{equation}
\sum_{j=1}^{M}A_{ij}\left( x\right) F_{j}\left( x\right) =f_{i}\left(
x\right) \text{ }\left( i=1,\cdots ,N\right) \text{,}  \label{intro1}
\end{equation}%
for unknown functions $F_{1},\cdots ,F_{N}\in C^{m}\left( \mathbb{R}%
^{n}\right) $ for fixed $m$.\footnote{$C^{m}\left( \mathbb{R}^{n}\right) $
denotes the vector space of $m$-times continuously differentiable functions $%
\mathbb{R}^{n}$, with no growth conditions assumed at infinity.\ Similarly, $%
C^{m}\left( \mathbb{R}^{n},\mathbb{R}^{D}\right) $ denotes the space of all
such $\mathbb{R}^{D}$-valued functions on $\mathbb{R}^{n}$. These notations
remain in force during the introduction, but will be changed later.}

Because $m$ is fixed, we're not allowed to lose derivatives.

The most interesting systems (\ref{intro1}) are underdetermined. An example
due to Epstein and Hochster \cite{Hochster} is the single equation 
\begin{equation}
x^{2}F_{1}+y^{2}F_{2}+xyz^{2}F_{3}=f\left( x,y,z\right)  \label{intro2}
\end{equation}%
for unknown continuous functions $F_{1},F_{2},F_{3}$ on $\mathbb{R}^{3}$.

For a system of the form (\ref{intro1}), we pose three problems.

\begin{problem}
\label{problem1}Suppose the $A_{ij}$ and $f_{i}$ are given functions. For
fixed $m$, how can we decide whether \eqref{intro1} admits a solution $%
F=\left( F_{1},\cdots ,F_{M}\right) \in C^{m}\left( \mathbb{R}^{n},\mathbb{R}%
^{M}\right) $?
\end{problem}

\begin{problem}
\label{problem2}Suppose the $A_{ij}$ are given polynomials. For fixed $m$,
the vectors $f=\left( f_{1},\cdots ,f_{N}\right) $ of polynomials $%
f_{1},\cdots ,f_{N}$ for which \eqref{intro1} admits a $C^{m}$ solution $F$
form a module $\mathcal{M}$ over the ring $\mathcal{R}$ of polynomials on $%
\mathbb{R}^{n}$. Exhibit generators for $\mathcal{M}$.
\end{problem}

\begin{problem}
\label{problem3}Suppose the $A_{ij}$ and $f_{i}$ are polynomials and suppose %
\eqref{intro1} admits a $C^{m}$ solution $F$. Can we take our $C^{m}$
solution $F$ to be semialgebraic?
\end{problem}

For $m=0$, these problems were posed by Brenner \cite{Brenner}, and
Epstein-Hochster \cite{Hochster}, and solved by Fefferman-Koll\'ar \cite%
{Feff-Kollar} and Koll\'ar \cite{kollar}.

In particular, for $m=0$, the answer to Problem \ref{problem3} is
affirmative; $F\in C^{0}$ can be taken to be semialgebraic. An example in
Koll\'ar-Nowak \cite{kn-continuous} shows that it isn't always possible to
take $F_{1},\cdots ,F_{N}$ to be rational functions. See Brenner-Steinbuch \cite{bren}, Koll\'ar \cite{kollar}, Koll\'ar-Nowak \cite{kn-continuous}, and  
Kucharz-Kurdyka \cite{MR3713914} for several related questions and
results.

For $m\geq 1$, Problem \ref{problem1} was solved in Fefferman-Luli \cite%
{fl-jets}, with no restriction on the functions $A_{ij}$, $f_{i}$.

In this paper and \cite{cf-luli-generator}, we solve Problem \ref{problem2}
for $m\geq 1.$

We needn't assume that the given matrix elements $A_{ij}$ are polynomials;
we may take them to be (possibly discontinuous) semialgebraic functions.

So far, Problem \ref{problem3} for $m\geq 1$ is still unsolved.

Observe that Problem \ref{problem2} can't be solved using only analysis,
because it concerns generators for a module over a polynomial ring. On the
other hand, it can't be solved using only algebra, because it concerns $%
C^{m} $ functions. To make a clean splitting into an analysis problem and an
algebra problem, we pose the analogue of Problem \ref{problem2} for vectors $%
f=\left( f_{1},\cdots ,f_{N}\right) $ of $C^{\infty }$ functions.

\begin{problem}
\label{problem2a}Fix a nonnegative integer $m$ and a matrix $\left(
A_{ij}\right) $ of semialgebraic functions on $\mathbb{R}^{n}$. Characterize
all the $f=\left( f_{1},\cdots ,f_{N}\right) \in C^{\infty }\left( \mathbb{R}%
^{n},\mathbb{R}^{N}\right) $ for which \eqref{intro1} admits a $C^{m}$%
-solution.
\end{problem}

To illustrate our result on Problem \ref{problem2a}, consider the
Epstein-Hochster equation (\ref{intro2}), for unknown continuous $%
F_{1},F_{2},F_{3}$. For $f\in C^{\infty }\left( \mathbb{R}^{3}\right) $, one
checks that a continuous solution exists if and only if $f$ satisfies 
\begin{equation}
\left[ 
\begin{array}{l}
f\left( x,y,z\right) =\frac{\partial f}{\partial x}\left( x,y,z\right) =%
\frac{\partial f}{\partial y}\left( x,y,z\right) =0 \\ 
\text{and} \\ 
\frac{\partial ^{2}f}{\partial x\partial y}\left( x,y,z\right) =\frac{%
\partial ^{3}f}{\partial x\partial y\partial z}\left( x,y,z\right) =0%
\end{array}%
\right. 
\begin{array}{l}
\text{for }x=y=0,z\in \mathbb{R} \\ 
\\ 
\text{at }x=y=z=0\text{.}%
\end{array}
\label{intro3}
\end{equation}

Note that a third derivative of $f$ enters into (\ref{intro3}), even though
we are merely looking for solutions $F=\left( F_{1},F_{2},F_{3}\right) \in
C^{0}$.

For general systems (\ref{intro1}), our result on Problem \ref{problem2a} is
as follows.

\begin{theorem}
\label{theorem1} Fix $m\geq 0$, and let $\left( A_{ij}\left( x\right)
\right) _{1\leq i\leq N,1\leq j\leq M}$ be a matrix of semialgebraic
functions on $\mathbb{R}^{n}$. Then there exist linear partial differential
operators $L_{1},L_{2},\cdots ,L_{K}$, for which the following hold.

\begin{itemize}
\item Each $L_{\nu }$ acts on vectors $f=\left( f_{1},\cdots ,f_{N}\right)
\in C^{\infty }\left( \mathbb{R}^{n},\mathbb{R}^{N}\right) $, and has the
form 
\begin{equation*}
L_{\nu }f\left( x\right) =\sum_{i=1}^{N}\sum_{\left\vert \alpha \right\vert
\leq \bar{m}}a_{\nu i\alpha }\left( x\right) \partial ^{\alpha }f_{i}\left(
x\right) 
\end{equation*}%
where the coefficients $a_{\nu i\alpha }$ are semialgebraic. (Perhaps $\bar{m%
}>m$.)

\item Let $f=\left( f_{1},\cdots ,f_{N}\right) \in C^{\infty }\left( \mathbb{%
R}^{n},\mathbb{R}^{N}\right) $. Then the system \eqref{intro1} admits a
solution $F=\left( F_{1},\cdots ,F_{M}\right) \in C^{m}\left( \mathbb{R}^{n},%
\mathbb{R}^{M}\right) $ if and only if $L_{\nu }f=0$ on $\mathbb{R}^{n}$ for
each $\nu =1,\cdots ,K$.
\end{itemize}
\end{theorem}

For the Epstein-Hochster equation (\ref{intro2}), with $m=0,$ the operators $%
L_{\nu }$ are 
\begin{equation*}
\mathbb{I}_{x=y=0}\partial ^{0},\mathbb{I}_{x=y=0}\partial _{x},\mathbb{I}%
_{x=y=0}\partial _{y},\mathbb{I}_{x=y=z=0}\partial _{xy}^{2},\mathbb{I}%
_{x=y=z=0}\partial _{xyz}^{3}, 
\end{equation*}%
where $\mathbb{I}$ denotes the indicator function. (Compare with (\ref%
{intro3}).)

Our proof of Theorem \ref{theorem1} is constructive. In principle, we can
compute the operators $L_{\nu }$ from the data $m$, $\left( A_{ij}\left(
x\right) \right) _{1\leq i\leq N,1\leq j\leq M}$.

In \cite{cf-luli-generator}, we apply Theorem \ref{theorem1} to solve
Problem \ref{problem2}. The idea is as follows.

Given linear partial differential operators $L_{1},\cdots ,L_{K}$ with
semialgebraic coefficients (not necessarily given by Theorem \ref{theorem1}%
), we introduce the $\mathcal{R}$-module of all polynomial vectors 
\begin{equation*}
f=\left( f_{1},\cdots ,f_{N}\right) \text{ such that }L_{\nu }\left(
Pf\right) =0\text{ on }\mathbb{R}^{n}\text{ for each }\nu =1,\cdots ,K, 
\end{equation*}%
and for every polynomial $P$. (Recall $\mathcal{R}$ is the ring of
polynomials on $\mathbb{R}^{n}$.) Let us call this $\mathcal{R}$-module $%
\mathcal{M}\left( L_{1},\cdots ,L_{K}\right) $.

If, as in Theorem \ref{theorem1}, the polynomial vectors $f$ annihilated by
the $L_{\nu } $ already form an $\mathcal{R}$-module, then that $\mathcal{R}$%
-module coincides with $\mathcal{M}\left( L_{1},\cdots ,L_{K}\right) $.

In particular, the module $\mathcal{M}$ in Problem \ref{problem2} satisfies $%
\mathcal{M=\mathcal{M}}\left( L_{1},\cdots ,L_{K}\right) $ for the $%
L_{1},\cdots ,L_{K}$ given by Theorem \ref{theorem1}.

Consequently, Problem \ref{problem2} is reduced to the following problem of
computational algebra.

\begin{problem}
\label{problem2b}Given linear partial differential operators $L_{1},\cdots
,L_{K}$ with semialgebraic coefficients, exhibit generators for the $%
\mathcal{R}$-module $\mathcal{M}\left( L_{1},\cdots ,L_{K}\right) $.
\end{problem}

We solve Problem \ref{problem2b} in \cite{cf-luli-generator}, completing the
solution of Problem \ref{problem2}.

Thus, by posing Problem \ref{problem2a}, we have split Problem \ref{problem2}
into an analysis problem and an algebra problem.

Let us sketch the proof of Theorem \ref{theorem1}. We oversimplify to bring
out the main ideas. The correct discussion appears in Sections \ref{prelim}, $\cdots$, \ref{passtononcompact} below.

To prepare the way, we introduce notation. For $x\in \mathbb{R}^{n}$ and $%
F\in C^{m}\left( \mathbb{R}^{n},\mathbb{R}^{D}\right) $, we write $%
J_{x}^{\left( m\right) }F$ (the \textquotedblleft $m$-jet" of $F$ at $x$) to
denote the $m^{th}$ order Taylor polynomial of $F$ at $x$.

Thus, $J_{x}^{\left( m\right) }F$ belongs to $\mathcal{P}^{\left( m\right)
}\left( \mathbb{R}^{n},\mathbb{R}^{D}\right) $, the vector space of all $%
\mathbb{R}^{D}$-valued polynomials of degree at most $m$ on $\mathbb{R}^{n}$.

If $D=1$, we write $\mathcal{P}^{\left( m\right) }\left( \mathbb{R}%
^{n}\right) $ in place of $\mathcal{P}^{\left( m\right) }\left( \mathbb{R}%
^{n},\mathbb{R}\right) $.

For $F,G\in C^{m}\left( \mathbb{R}^{n}\right) $ and $x\in \mathbb{R}^{n}$,
we have $J_{x}^{\left( m\right) }\left( FG\right) =J_{x}^{\left( m\right)
}F\odot _{x}J_{x}^{\left( m\right) }G$, where $P\odot _{x}Q:=J_{x}^{(m)}\left(
PQ\right) $ for $P,Q\in \mathcal{P}^{\left( m\right) }\left( \mathbb{R}%
^{n}\right) $. The multiplication $\odot _{x}$ makes $\mathcal{P}^{\left(
m\right) }\left( \mathbb{R}^{n}\right) $ into a ring $\mathcal{R}_{x}^{(m)}$%
, the \textquotedblleft ring of $m$-jets at $x$". Similarly, the
multiplication $Q\odot _{x}\left( P_{1},\cdots ,P_{M}\right) :=\left( Q\odot
_{x}P_{1},\cdots ,Q\odot _{x}P_{M}\right) $ makes $\mathcal{P}^{\left(
m\right) }\left( \mathbb{R}^{n},\mathbb{R}^{M}\right) $ into an $\mathcal{R}%
_{x}^{\left( m\right) }$-module.

We can now explain our solution \cite{fl-jets} to Problem \ref{problem1};
later, we will apply what we have learned to Problem \ref{problem2a}. Thus,
let us fix $m\geq 0$, and let $A_{ij}$ and $f_{i}$ be given functions. We
investigate whether (\ref{intro1}) has a solution $F\in C^{m}\left( \mathbb{R%
}^{n},\mathbb{R}^{M}\right) $.

The idea is to construct families $\mathscr{H}=\left( H_{x}\right) _{x\in 
\mathbb{R}^{n}}$ of affine subspaces $H_{x}\subset \mathcal{P}^{\left(
m\right) }\left( \mathbb{R}^{n},\mathbb{R}^{M}\right) $, such that any $%
C^{m} $ solution of (\ref{intro1}) necessarily satisfies
\begin{equation}
J_{x}^{\left( m\right) }F\in H_{x}\text{ for all }x\in \mathbb{R}^{n}\text{.}
\label{intro4}
\end{equation}

To start with, we can simply take

\begin{itemize}
\item[\refstepcounter{equation}\text{(\theequation)}\label{intro5}] {$%
\widehat{\mathscr{H}}=\left( \hat{H}_{x}\right) _{x\in \mathbb{R}^{n}},$
where 
\begin{equation*}
\hat{H}_{x}=\left\{ P=\left( P_{1},\cdots ,P_{M}\right) \in \mathcal{P}%
^{\left( m\right) }\left( \mathbb{R}^{n},\mathbb{R}^{M}\right)
:\sum_{j=1}^{M}A_{ij}\left( x\right) P_{j}\left( x\right) =f_{i}\left(
x\right) , \text{ }\left( i=1,\cdots ,N\right) \right\}.
\end{equation*}%
}
\end{itemize}

We allow the empty set as an affine subspace of $\mathcal{P}^{\left(
m\right) }\left( \mathbb{R}^{n},\mathbb{R}^{M}\right) $. This can already
happen for $\widehat{\mathscr{H}}$ given by (\ref{intro5}), if equations (%
\ref{intro1}) are inconsistent for some $x$. Obviously, (\ref{intro4})
cannot hold if some of the $H_{x}$ are empty.

The nonempty $H_{x}$ arising in our families $\mathscr{H}$ will have a
special form; they are translates of $\mathcal{R}_{x}^{\left( m\right) }$%
-submodules of $\mathcal{P}^{\left( m\right) }\left( \mathbb{R}^{n},\mathbb{R%
}^{M}\right) $.

Thus,

\begin{itemize}
\item[\refstepcounter{equation}\text{(\theequation)}\label{intro6}] {$%
\mathscr{H}=\left( H_{x}\right) _{x\in \mathbb{R}^{n}}$, where for each $x$,
either $H_{x}$ is empty, or $H_{x}=P_{x}+I\left( x\right) $, where $P_{x}\in 
\mathcal{P}^{\left( m\right) }\left( \mathbb{R}^{n},\mathbb{R}^{M}\right) $
and $I\left( x\right) \in \mathcal{P}^{\left( m\right) }\left( \mathbb{R}%
^{n},\mathbb{R}^{M}\right) $ is an $\mathcal{R}_{x}^{(m)}$-submodule.}
\end{itemize}

For instance, $\widehat{\mathscr{H}}$ given by (\ref{intro5}) has this form.

We call any $\mathscr{H}$ of the form (\ref{intro6}) a \textquotedblleft
bundle", and we call $F\in C^{m}\left( \mathbb{R}^{n},\mathbb{R}^{M}\right) $
a \textquotedblleft section" of the bundle $\mathscr{H}$ if (\ref{intro4})
holds. Also, if $\mathscr{H}=\left( H_{x}\right) _{x\in \mathbb{R}^{n}}$ and 
$\mathscr{H}^{\prime }=\left( H_{x}^{\prime }\right) _{x\in \mathbb{R}^{n}}$
are bundles, then we say that $\mathscr{H}^{\prime }$ is a \textquotedblleft
subbundle" of $\mathscr{H}$ if $H_{x}^{\prime }\subset H_{x}$ if all $x\in 
\mathbb{R}^{n}$. We write $\mathscr{H}$ $\supset \mathscr{H}^{\prime }$ to
indicate that $\mathscr{H}^{\prime }$ is a subbundle of $\mathscr{H}$.

We call $H_{x_{0}}$ the \textquotedblleft fiber" of $\mathscr{H}=\left(
H_{x}\right) _{x\in \mathbb{R}^{n}}$ at $x_{0}$.

Immediately from the definition (\ref{intro5}), we see that a $C^{m}$
solution of the system (\ref{intro1}) is precisely a section of the bundle $%
\widehat{\mathscr{H}}$. Therefore, Problem \ref{problem1} is a special case
of the following.

\begin{problem}
\label{problem1forbundles}Given a bundle $\mathscr{H}$, decide whether $%
\mathscr{H}$ has a section.
\end{problem}

We solve this problem using the notion of \textquotedblleft Glaeser
refinement". The idea is as follows. Let $\mathscr{H}=\left( H_{x}\right)
_{x\in \mathbb{R}^{n}}$ be a bundle, and let $x_{0}\in \mathbb{R}^{n}$. By
definition, any section $F$ of $\mathscr{H}$ must satisfy $J_{x_{0}}^{\left(
m\right) }F\in H_{x_{0}}$. However, $H_{x_{0}}$ may contain polynomials $%
P_{0}\in \mathcal{P}^{\left( m\right) }\left( \mathbb{R}^{n},\mathbb{R}%
^{M}\right) $ that can never arise as the $m$-jet at $x_{0}$ of any section.

In that case, we may replace $\mathscr{H}$ by a subbundle $\tilde{\mathscr{H}}$ without losing any sections. Let us see how such an $\tilde{\mathscr{H}}$ can be
defined.

Fix $x_{0}\in \mathbb{R}^{n}$ and $P_{0}\in H_{x_{0}}$. Suppose $F$ is a
section of $\mathscr{H}$, with $J_{x_{0}}^{\left( m\right) }F=P_{0}$. Fix a
large integer constant $k$ (determined by $m,n,M),$ and let $x_{1},\cdots
,x_{k}\in \mathbb{R}^{n}$ lie close to $x_{0}$.

Setting $P_{i}=J_{x_{i}}^{\left( m\right) }F$ for $i=1,\cdots ,k$, we have $%
P_{1}\in H_{x_{1}},P_{2}\in H_{x_{2}},\cdots ,P_{k}\in H_{x_{k}}$, and

\begin{itemize}
\item[\refstepcounter{equation}\text{(\theequation)}\label{intro7}] {$%
\sum_{0\leq i< j \leq k }\sum_{\left\vert \alpha \right\vert \leq m}\left( \frac{%
\left\vert \partial ^{\alpha }\left( P_{i}-P_{j}\right) \left( x_{j}\right)
\right\vert }{\left\vert x_{i}-x_{j}\right\vert ^{m-\left\vert \alpha
\right\vert }}\right) ^{2}\rightarrow 0$ as $x_{1},\cdots ,x_{k}\rightarrow
x_{0}$, by Taylor's theorem.}
\end{itemize}

Note that $P_{0},\cdots ,P_{k}$ enter into (\ref{intro7}), but $P_{0}$ plays
a different r\^ole from $P_{1},\cdots ,P_{k}$.

The above remarks lead us to define the Glaeser refinement of the bundle $%
\mathscr{H}=\left( H_{x}\right) _{x\in \mathbb{R}^{n}}$ by setting $%
\mathscr{G}\left( \mathscr{H}\right) =\left( \mathscr{\tilde{H}}_{x}\right)
_{x\in \mathbb{R}^{n}}$, where for each $x_{0}\in \mathbb{R}^{n}$, $\tilde{H}%
_{x_{0}}$ consists of those $P_{0}\in H_{x_{0}}$ such that

\begin{itemize}
\item[\refstepcounter{equation}\text{(\theequation)}\label{intro8}] {$\min
\left\{ \sum_{0\leq i <j \leq k}\sum_{\left\vert \alpha \right\vert \leq m}\left( 
\frac{\left\vert \partial ^{\alpha }\left( P_{i}-P_{j}\right) \left(
x_{0}\right) \right\vert }{\left\vert x_{i}-x_{j}\right\vert ^{m-\left\vert
\alpha \right\vert }}\right) ^{2}:P_{1}\in H_{x_{1}},\cdots ,P_{k}\in
H_{x_{k}}\right\} $ tends to zero \\ as $x_{1},\cdots ,x_{k}\rightarrow x_{0}$.}
\end{itemize}

The Glaeser refinement has three basic properties.

\begin{itemize}
\item $\mathscr{G}\left( \mathscr{H}\right) $ is a subbundle of $\mathscr{H}$%
.

\item $\mathscr{G}\left( \mathscr{H}\right) $ and $\mathscr{H}$ have the
same sections, as we saw above.

\item $\mathscr{G}\left( \mathscr{H}\right) $ can in principle be computed
from $\mathscr{H}$, thanks to the explicit nature of (\ref{intro8}).
\end{itemize}

Note that $\mathscr{G}\left( \mathscr{H}\right) $ may have empty fibers,
even if $\mathscr{H}$ has none. In that case, we know that $\mathscr{H}$ has
no sections.

Starting from a given bundle $\mathscr{H}$, we can now perform
\textquotedblleft iterated Glaeser refinement" to pass to ever smaller
subbundles $\mathscr{H}^{\left( 1\right) },\mathscr{H}^{\left( 2\right) }$,
etc., without losing sections. We set $\mathscr{H}^{\left( 0\right) }=%
\mathscr{H}$, and for $l\geq 0$, we set $\mathscr{H}^{\left( l+1\right) }=%
\mathscr{G}\left( \mathscr{H}^{\left( l\right) }\right) $. Thus, by an
obvious induction on $l$, $\mathscr{H}=\mathscr{H}^{\left( 0\right) }\supset %
\mathscr{H}^{\left( 1\right) }\supset \mathscr{H}^{\left( 2\right) }\supset
\cdots ,$ yet $\mathscr{H}$ and $\mathscr{H}^{\left( l\right) }$ have the
same sections.

In principle, each $\mathscr{H}^{\left( l\right) }$ can be computed from $%
\mathscr{H}$.

The main result of \cite{fl-jets} gives the

\underline{Solution to Problem \ref{problem1forbundles}}: {\emph{For a large
enough integer constant }} $l_{\ast }$\emph{\ determined by }$m,n,M,$\emph{\
the following holds. }

\emph{Let }$\mathscr{H}$\emph{\ be a bundle, and let }$\mathscr{H}^{\left(
0\right) }, \mathscr{H}^{\left( 1\right) }, \mathscr{H}^{\left( 2\right)
},\cdots $\emph{\ be its iterated Glaeser refinements. Then }$\mathscr{H}$%
\emph{\ has a section if and only if }$\mathscr{H}^{\left( l_{\ast }\right) }
$\emph{\ has no empty fibers.}

In particular, this solves Problem \ref{problem1} for systems of equations (%
\ref{intro1}). This concludes our discussion of Problem \ref{problem1}.

We now want to apply the above to Problem \ref{problem2a}. To do so, we have
to understand how the iterated Glaeser refinements arising from the bundle $%
\widehat{\mathscr{H}}$ in (\ref{intro5}) depend on the right-hand side $%
f=\left( f_{1},\cdots ,f_{N}\right) $ in (\ref{intro1}), assuming $f\in
C^{\infty }.$

This gives rise to the study of bundles of the form

\begin{itemize}
\item[\refstepcounter{equation}\text{(\theequation)}\label{intro9}] {$%
\mathscr{H}_{f}=\left( T\left( x\right) J_{x}^{\left( \bar{m}\right)
}f+I\left( x\right) \right) _{x\in \mathbb{R}^{n}}$, where $I\left( x\right)
\subset \mathcal{P}^{\left( m\right) }\left( \mathbb{R}^{n},\mathbb{R}%
^{M}\right) $ is an $\mathcal{R}_{x}^{(m)}$-submodule depending
semialgebraically on $x$, and $T\left( x\right) :\mathcal{P}^{\left( \bar{m}%
\right) }\left( \mathbb{R}^{n},\mathbb{R}^{N}\right) \rightarrow \mathcal{P}%
^{\left( m\right) }\left( \mathbb{R}^{n},\mathbb{R}^{M}\right) $ is a linear
map, also depending semialgebraically on $x$.}
\end{itemize}

We want to understand how the Glaeser refinement of the bundle $\mathscr{H}%
_{f}$ in (\ref{intro9}) depends on $f\in C^{\infty }$. In particular, we
want to know when that Glaeser refinement has no empty fibers. Under
suitable assumptions on $T\left( x\right) $ in (\ref{intro9}), we prove the
following:

\begin{itemize}
\item[\refstepcounter{equation}\text{(\theequation)}\label{intro10}] {The
fibers of $\mathscr{G}\left( \mathscr{H}_{f}\right) $ are all non-empty if
and only if $f$ is annihilated by finitely many linear partial differential
operators $L_{1},\cdots ,L_{K}$ \ with semialgebraic coefficients.}
\end{itemize}

\begin{itemize}
\item[\refstepcounter{equation}\text{(\theequation)}\label{intro11}] {If the
fibers of $\mathscr{G}\left( \mathscr{H}_{f}\right) $ are all non-empty,
then $\mathscr{G}\left( \mathscr{H}_{f}\right) $ again has the form (\ref%
{intro9}), possibly with a smaller $I\left( x\right) $, a larger $\bar{m}$, and a
different $T\left( x\right) $.}
\end{itemize}

This allows us to keep track of the $f$-dependence of the iterated Glaeser
refinements of the bundle $\widehat{\mathscr{H}}$ in (\ref{intro5}), thus
proving Theorem \ref{theorem1}.

Let us say a few words about the proof of (\ref{intro10}) and (\ref{intro11}%
).

Because a quadratic form in (\ref{intro8}) lies at the heart of the matter,
we have to understand quadratic forms acting on the jets of a function $f\in
C^{\infty }\left( \mathbb{R}^{n},\mathbb{R}^{N}\right) $ at points $%
x_{1},\cdots ,x_{k}\in \mathbb{R}^{n}$. More precisely, suppose we are given
a positive semidefinite quadratic form

\begin{itemize}
\item[\refstepcounter{equation}\text{(\theequation)}\label{intro13}] {$%
\left( P_{0},P_{1},\cdots ,P_{k}\right) \mapsto Q\left(
x_{0},P_{0},x_{1},P_{1},\cdots ,x_{k},P_{k}\right) $ depending
semialgebraically on points $x_{1},\cdots ,x_{k}\in \mathbb{R}^{n}$. Here, $%
P_{0}\in \mathcal{P}^{\left( m\right) }\left( \mathbb{R}^{n},\mathbb{R}%
^{M}\right) $, while $P_{1},\cdots ,P_{k}\in \mathcal{P}^{\left( \bar{m}%
\right) }\left( \mathbb{R}^{n},\mathbb{R}^{N}\right) $.}
\end{itemize}

We fix $x_{0},P_{0},$ and let $x_{1},\cdots ,x_{k}$ vary. We have to
characterize the functions $f\in C^{\infty }\left( \mathbb{R}^{n},\mathbb{R}%
^{N}\right) $ such that $Q\left( x_{0},P_{0},x_{1},J_{x_{1}}^{\left( \bar{m}%
\right) }f,\cdots ,x_{k},J_{x_{k}}^{\left( \bar{m}\right) }f\right)
\rightarrow 0$ as $x_{1},\cdots ,x_{k}\rightarrow x_{0}$.

Section \ref{statement-quadform} contains our results on this problem,
namely Propositions \ref{SAQF1} and \ref{SAQF2}. These propositions are
proven by induction on the dimension of a relevant semialgebraic set. To
make the induction work, we have to allow our quadratic form (\ref{intro13})
to depend on additional points $z_{1},\cdots ,z_{L}$. We refer the reader to
Section \ref{proofofprops23} for full details. Section \ref{proofofprops23} contains the  
main work in our proof of Theorem \ref{theorem1}.

We will first establish the following variant of Theorem \ref{theorem1}.

\begin{theorem}
\label{statement-main-theore}Let $E\subset \mathbb{R}^{n}$ be compact,
semialgebraic. Let $\left( A_{ij}\left( x\right) \right) _{1\leq i\leq
N,1\leq j\leq M}$ be a matrix of semialgebraic functions defined on $E$. Let 
$m\geq 0$ be given. Then there exist linear partial differential operators $%
L_{\nu }$ $\left( 1\leq \nu \leq \nu _{\max }\right) $, for which the
following hold.

\begin{itemize}
\item Each $L_{\nu }$ has semialgebraic coefficients and carries functions
in $C^{\infty }\left( \mathbb{R}^{n},\mathbb{R}^{N}\right) $ to
scalar-valued functions on $\mathbb{R}^{n}$.

\item Let $f=\left( f_{1},\cdots ,f_{N}\right) \in C^{\infty }\left( \mathbb{%
R}^{n},\mathbb{R}^{N}\right) $. Then there exist $F_{1},\cdots ,F_{M}\in
C^{m}\left( \mathbb{R}^{n}\right) $ such that 
\begin{equation*}
\sum_{j=1}^{M}A_{ij}\left( x\right) F_{j}\left( x\right) =f_{i}\left(
x\right)
\end{equation*}%
all $x\in E$, $i=1,\cdots ,N$
\end{itemize}

if and only if $L_{\nu }f=0$ on $E$ for all $\nu =1,\cdots ,\nu _{\max }$.
\end{theorem}

In Section \ref{passtononcompact}, we show how to pass from the compact case
to the noncompact case, and thus establish Theorem \ref{theorem1}.

This concludes our explanation of the proof of Theorem \ref{theorem1}. We
again warn the reader that the explanation is oversimplified, and that the
true story is to be found in Sections \ref{prelim}, $\cdots$, \ref%
{passtononcompact} below.

We should also warn the reader that although our results solve Problems \ref%
{problem1} and \ref{problem2} in principle, the calculations involved are
prohibitive in practice, except in the simplest cases.

The results of this paper were announced in \cite{f-luli}.

We are grateful to Matthias Aschenbrenner, Saugata Basu, Edward Bierstone, 
Zeev Dvir, J\'anos Koll\'ar,
Pierre Milman, Wieslaw Paw\l {}ucki, Ary Shaviv and the participants in the 9%
$^{th}$-11$^{th}$ Whitney workshops for valuable discussions, and to the
Technion -- Israel Institute of Technology, College of William and Mary, and
Trinity College Dublin, for hosting the above workshops. 

\section{Preliminaries}\label{prelim}

\subsection{Notation}

$\mathcal{P}^{(m)}\left( \mathbb{R}^{n},\mathbb{R}^{D}\right) $ denotes the
vector space of $\mathbb{R}^{D}$-valued polynomials of degree at most $m$ on 
$\mathbb{R}^{n}$. If $D=1$, we may write $\mathcal{P}^{(m)}\left( \mathbb{R}%
^{n}\right) $ in place of $\mathcal{P}^{(m)}\left( \mathbb{R}^{n},\mathbb{R}%
\right) $.

{\textbf{We depart from the notation used in the Introduction.}} From now on, $%
C^{m}\left( \mathbb{R}^{n},\mathbb{R}^{D}\right) $ denotes the space of all $%
\mathbb{R}^D$-valued functions on $\mathbb{R}^n$ whose derivatives up to
order $m$ are continuous and bounded on $\mathbb{R}^n$; $C^m_{loc}(\mathbb{R}%
^n, \mathbb{R}^D)$ denotes the space of $\mathbb{R}^D$-valued functions on $%
\mathbb{R}^n$ with continuous derivatives up to order $m$; $C^{\infty }_0
\left( \mathbb{R}^{n},\mathbb{R}^{D}\right) $ denotes the space of
infinitely differentiable $\mathbb{R}^D$-valued functions of compact support
on $\mathbb{R}^n$; $C^{\infty } \left( \mathbb{R}^{n},\mathbb{R}^{D}\right) $
denotes the space of infinitely differentiable $\mathbb{R}^D$-valued
functions on $\mathbb{R}^n$. If $D=1$, we write $C^m(\mathbb{R}^n),
C^{\infty}(\mathbb{R}^n), C_0^{\infty}(\mathbb{R}^n)$ in place of $C^m(%
\mathbb{R}^n,\mathbb{R}^D), C^{\infty}(\mathbb{R}^n,\mathbb{R}^D),
C^{\infty}_0 (\mathbb{R}^n,\mathbb{R}^D)$, respectively.

If $F\in C^{m}\left( \mathbb{R}^{n},\mathbb{R}^{D}\right) $ and $x\in 
\mathbb{R}^{n}$, then $J_{x}^{\left( m\right) }F$ (the \textquotedblleft $m$%
-jet" of $F$ at $x$) denotes the $m^{\text{th}}$ order Taylor polynomial of $%
F$ at $x$.

We write $\pi _{x}^{m^{\prime }\rightarrow m}:\mathcal{P}^{(m^{\prime
})}\left( \mathbb{R}^{n},\mathbb{R}^{D}\right) \rightarrow \mathcal{P}%
^{(m)}\left( \mathbb{R}^{n},\mathbb{R}^{D}\right) $ for the natural
projection from $m^{\prime }$-jets at $x$ to $m$-jets at $x$ $\left(
m^{\prime }\geq m\right) $.

\subsection{A simple consequence of Taylor's Theorem}

Let $F\in C^{\bar{\bar{m}}}\left( \mathbb{R}^{n},\mathbb{R}^{D}\right) $,
let $\bar{\bar{m}} \geq \bar{m}$, and let $x,y\in \mathbb{R}^{n}$. Then for $%
\left\vert \alpha \right\vert \leq \bar{m}$, we have 
\begin{eqnarray}
&&\partial ^{\alpha }\left\{ \pi _{x}^{\bar{\bar{m}}\rightarrow \bar{m}%
}J_{y}^{\left( \bar{\bar{m}}\right) }F-J_{x}^{\left( \bar{m}\right)
}F\right\} \left( x\right)  \notag \\
&=&\left. \partial _{z}^{\alpha }\left\{ \sum_{\left\vert \gamma \right\vert
\leq \bar{m}}\frac{1}{^{\gamma !}}\left[ \partial ^{\gamma }\left(
J_{y}^{\left( \bar{\bar{m}}\right) }F\right) \left( x\right) \right] \cdot
\left( z-x\right) ^{\gamma }-\sum_{\left\vert \gamma \right\vert \leq \bar{m}%
}\frac{1}{\gamma !}\partial ^{\gamma }F\left( x\right) \cdot \left(
z-x\right) ^{\gamma }\right\} \right\vert _{z=x}  \notag \\
&&\text{(by definition of }\pi _{x}^{\bar{\bar{m}}\rightarrow \bar{m}}\text{
and }J_{x}^{\left( \bar{m}\right) }F\text{)}  \notag \\
&=&\partial ^{\alpha }\left( J_{y}^{\left( \bar{\bar{m}}\right) }F\right)
\left( x\right) -\partial ^{\alpha }F\left( x\right)  \notag \\
&&\text{ (since }\partial _{z}^{\alpha }\left( z-x\right) ^{\gamma }=\alpha
!\delta _{\gamma \alpha }\text{ at }z=x\text{, }\delta _{\gamma \alpha }%
\text{ denotes the Kronecker delta)}  \notag \\
&=&\partial _{x}^{\alpha }\left\{ \sum_{\left\vert \gamma \right\vert \leq 
\bar{\bar{m}}}\frac{1}{\gamma !}\partial ^{\gamma }F\left( y\right) \cdot
\left( x-y\right) ^{\gamma }-F\left( x\right) \right\} \text{ (by definition
of }J_{y}^{\left( \bar{\bar{m}}\right) }F\text{).}  \label{1}
\end{eqnarray}%
The quantity \eqref{1} has absolute value (i.e. norm in $\mathbb{R}^{D}$) at
most 
\begin{equation*}
C_{0}\left\Vert F\right\Vert _{C^{\bar{\bar{m}}}\left( \mathbb{R}^{n},%
\mathbb{R}^{D}\right) }\cdot \left\vert x-y\right\vert ^{\bar{\bar{m}}%
-\left\vert \alpha \right\vert }\text{,}
\end{equation*}%
by Taylor's theorem, with $C_{0}$ depending only on $\bar{\bar{m}},n,D$.
Therefore, the following holds:

\begin{proposition}
\label{taylor'sthm}%
\begin{eqnarray*}
&&\left\vert \partial ^{\alpha }\left\{ \pi _{x}^{\bar{\bar{m}}\rightarrow 
\bar{m}}J_{y}^{\left( \bar{\bar{m}}\right) }F-J_{x}^{\left( \bar{m}\right)
}F\right\} \left( x\right) \right\vert \\
&\leq &C\left\Vert F\right\Vert _{C^{\bar{\bar{m}}}\left( \mathbb{R}^{n},%
\mathbb{R}^{D}\right) }\cdot \left\vert x-y\right\vert ^{\bar{\bar{m}}%
-\left\vert \alpha \right\vert }\text{ for }\bar{\bar{m}}\geq \bar{m}\geq
\left\vert \alpha \right\vert \text{, }F\in C^{\infty }_0\left( \mathbb{R}%
^{n},\mathbb{R}^{D}\right) ,
\end{eqnarray*}%
where $C$ depends only on $\bar{\bar{m}},n,D$.
\end{proposition}

\subsection{Semialgebraic sets and functions}

Let $A:E\rightarrow \mathbb{R}$, where $E\subset \mathbb{R}^{N}$ is
semialgebraic. Recall that $A$ is a semialgebraic function if its graph $%
\left\{ \left( x,y\right) \in E\times \mathbb{R}: y=A\left( x\right)
\right\} $ is semialgebraic. In particular, semialgebraic functions needn't be continuous.

Note that, by definition, a semialgebraic function is finite everywhere on $%
E $. Thus, the following functions are not semialgebraic on $\mathbb{R}$:

\begin{itemize}
\item $f\left( x\right) =\left\{ 
\begin{array}{c}
\frac{1}{x^{2}} \\ 
\infty%
\end{array}%
\right. 
\begin{array}{c}
\text{if }x\not=0 \\ 
\text{if }x=0%
\end{array}%
$

\item $f\left( x\right) =\left\{ 
\begin{array}{c}
\frac{1}{x^{2}} \\ 
\text{undefined}%
\end{array}%
\right. 
\begin{array}{l}
\text{if }x\not=0 \\ 
\text{if }x=0\text{.}%
\end{array}%
$
\end{itemize}

However, the following function is semialgebraic:%
\begin{equation*}
h\left( x\right) =\left\{ 
\begin{array}{c}
\frac{1}{x^{2}} \\ 
17%
\end{array}%
\right. 
\begin{array}{c}
\text{if }x\not=0 \\ 
\text{if }x=0%
\end{array}%
\text{.}
\end{equation*}

The \underline{dimension} of a semialgebraic set $E\subset \mathbb{R}^{n}$
is the maximum of the dimensions of all the imbedded (not necessarily
compact) submanifolds of $\mathbb{R}^{n}$ that are contained in $E$.

For instance, in $\mathbb{R}^{3}$, the union of the $x$-$y$ plane and the $z$
axis has dimension $2$.

A map $\phi :E\rightarrow \mathbb{R}^{N}$ is semialgebraic if $\left\{
\left( x,y\right) \in E\times \mathbb{R}^{N}:y=\phi \left( x\right) \right\} 
$ is a semialgebraic set.

Again, semialgebraic maps $\phi :E\rightarrow \mathbb{R}^{N}$ are defined
everywhere on $E$.

%
%
%

\subsection{Limits}

Let $E$ be a metric space, let $f:E\rightarrow \mathbb{R}$ be a function,
and let $x\in E$ be given. As every student knows, $\lim_{y\rightarrow
x}f\left( y\right) =L$ means that given $\varepsilon >0$ there exists $%
\delta >0$ such that $\left\vert f\left( y\right) -L\right\vert <\varepsilon 
$ for all $y\in E$ with $\text{dist}\left( y,x\right) <\delta $.

We point out here that if $x$ is an isolated point of $E$, then $%
\lim_{y\rightarrow x}f\left( x\right) =L$ means simply that $f\left(
x\right) =L$.

Note that the function 
\begin{equation*}
f\left( x\right) =\left\{ 
\begin{array}{c}
0 \\ 
17%
\end{array}%
\right. 
\begin{array}{c}
\text{if }x\not=0 \\ 
\text{if }x=0%
\end{array}%
,
\end{equation*}%
defined on $\mathbb{R}$, does not satisfy $\lim_{y\rightarrow 0}f\left(
y\right) =0$.

Similarly, if $G \subset E \times \cdots \times E$, then the condition 
\begin{equation}
\lim_{\substack{ y_{1},\cdots ,y_{N}\rightarrow x  \\ \left(y_{1},\cdots
,y_{N}\right) \in G }} f(y_1,\cdots,y_N) = L  \label{2.4.1}
\end{equation}
is defined in the usual way via $\varepsilon$'s and $\delta$'s.

In particular, \eqref{2.4.1} holds vacuously if $G$ fails to contain points
arbitrarily close to $(x,x, \cdots, x)$.

\subsection{Computations with Semialgebraic Sets}\label{sec2.5}

In this section, we present some known technology for computations
involving semialgebraic sets. See the reference book \cite{basu}.

We begin by describing our model of computation. Our algorithms are to be
run on an idealized computer with standard von Neumann architecture \cite{vonneumann},
able to store and perform basic arithmetic operations on integers and
infinite precision real numbers, without roundoff errors or overflow
conditions. We suppose that our computer can access an ORACLE that solves
polynomial equations in one unknown. More precisely, the ORACLE answers
queries; a query consists of a non-constant polynomial $P$ (in one variable)
with real coefficients, and the ORACLE responds to a query $P$ by producing
a list of all the real roots of $P$.

Let us compare our model of computation with that of \cite{basu}.

All  arithmetic in \cite{basu} is performed within a subring $\Lambda$ of a real  
closed field $K$ (e.g. the integers sitting inside the reals). However,  
some algorithms in \cite{basu} produce as output a finite list of elements of  
$K$ not necessarily belonging to $\Lambda$. A field element $x_0$ arising in  
such an output is  specified by exhibiting a polynomial $P$ (in one  
variable) with coefficients in $\Lambda$ such that $P(x_0)=0$,  together with  
other data to distinguish $x_0$ from the other roots of $P$.

In our model of computation, we take $\Lambda$ and $K$ to consist of all real
numbers, and we query the ORACLE whenever \cite{basu} specifies a real number  
by means of a polynomial $P$ as above.

Next, we describe how we will represent a semialgebraic set $E$. We will
specify a Boolean combination of sets of the form

\LAQ{XXX}{$\left\{ \left(x_1, \cdots, x_n \right) \in \mathbb{R}^n:
P(x_1,\cdots, x_n)>0 \right\}$,}
\LAQ{YYY}{$\left\{ \left(x_1, \cdots, x_n \right) \in \mathbb{R}^n:
P(x_1,\cdots, x_n)<0 \right\}$, or}
\LAQ{ZZZ}{$\left\{ \left(x_1, \cdots, x_n \right) \in \mathbb{R}^n:
P(x_1,\cdots, x_n)=0 \right\}$}
for polynomials $P$ with real coefficients.

Of course it is possible to represent the same set $E$ by many different
Boolean combination of the above form, but that doesn't bother us.

We will specify a semialgebraic function by specifying its graph.

Let us mention a few basic algorithms from \cite{basu}.

\begin{itemize}
\item Given a semialgebraic set $E \subset \mathbb{R}^n$, we compute its
dimension. (See Algorithm 14.31 in \cite{basu}.)

\item Given a zero-dimensional (and consequently finite) semialgebraic set $%
E \subset \mathbb{R}^n$, we compute a list of all the elements of $E$. (See Algorithm 16.20 in \cite{basu}.)

\item Given a semialgebraic set, we produce a list of its connected components, exhibiting each component as a Boolean combination of sets of the form \eqref{XXX}, \eqref{YYY}, and \eqref{ZZZ}. (See Algorithm 16.20 in \cite{basu}.)

\item Given semialgebraic sets $E_1 \subset \mathbb{R}^{n_1}$, $E \subset
E_1 \times \mathbb{R}^{n_2}$, we check whether it is the case that
\begin{equation}\label{compute1}
\text{For
every }x \in E_1,\text{ there exists }y \in \mathbb{R}^{n_2}\text{ such that } (x,y)
\in E.     
\end{equation}
If \eqref{compute1} holds, we compute a (possibly discontinuous) semialgebraic function $F: E_1
\rightarrow \mathbb{R}^{n_2}$ such that $(x, F(x)) \in E$ for all $x \in E_1$. (See Algorithm 11.3 as well as Section 5.1 in \cite{basu}.) 
\end{itemize}

Next, we discuss ``elimination of quantifiers", a powerful tool to show that
certain sets are semialgebraic, and to compute those sets.

The sets in question consist of all $(x_1, \cdots, x_n) \in \mathbb{R}^n$
that satisfy a certain condition $\Phi(x_1, \cdots, x_n)$. Here, $\Phi(x_1,
\cdots, x_n)$ is a statement in a formal language, the ``first order
predicate calculus for the theory of real closed fields".

Rather than giving careful definitions, we illustrate with a few examples,
and refer the reader to \cite{carlos,basu}.

\begin{itemize}
\item If $E \subset \mathbb{R}^{n_1} \times \mathbb{R}^{n_2}$ is a given
semialgebraic set, and if $\pi: \mathbb{R}^{n_1} \times \mathbb{R}^{n_2}
\rightarrow \mathbb{R}^{n_1}$ denotes the natural projection, then we can
compute the semialgebraic set $\pi E$, because $\pi E$ consists of all $%
(x_1, \cdots, x_{n_1}) \in \mathbb{R}^{n_1}$ satisfying the condition
\begin{equation*}
\Phi (x_1, \cdots, x_{n_1}): \left(\exists y_1 \right)\cdots \left(\exists
y_{n_2} \right) \left( (x_1, \cdots, x_{n_1}, y_1, \cdots, y_{n_2}) \in
E\right).
\end{equation*}

\item Suppose $E \subset \mathbb{R}^n$ is semialgebraic. Then we can compute 
$E^{\text{closure}}$, the closure of $E$, because $E^{\text{closure}}$
consists of all $(x_1, \cdots, x_n)$ satisfying the condition
\begin{equation*}
\Phi (x_{1},\cdots ,x_{n}):\left( \forall \varepsilon >0\right) \left(
\exists y_{1}\right) \cdots \left( \exists y_{n}\right) \left( \left[ \left(
y_{1},\cdots ,y_{n}\right) \in E\right] \wedge \left[ \left(
x_{1}-y_{1}\right) ^{2}+\cdots +\left( x_{n}-y_{n}\right) ^{2}<\varepsilon
^{2}\right] \right) \text{.}
\end{equation*}%
In particular, $E^{\text{closure}}$ is semialgebraic.

\item Let $E,\underline{E}\subset \mathbb{R}^{n}$ be given semialgebraic
sets. Then we can compute the semialgebraic set 
\begin{equation*}
\tilde{E}=\left\{ 
\begin{array}{c}
\left( x_{1},\cdots ,x_{n},\underline{x}_{1},\cdots ,\underline{x}%
_{n}\right) \in E\times \underline{E}:\left( \underline{x}_{1},\cdots ,%
\underline{x}_{n}\right) \text{ is at least as close as} \\ 
\text{any point of }\underline{E}\text{ to }\left( x_{1},\cdots
,x_{n}\right) 
\end{array}%
\right\} ,
\end{equation*}%
because $\tilde{E}$ consists of all $\left( x_{1},\cdots ,x_{n},\underline{x}%
_{1},\cdots ,\underline{x}_{n}\right) \in \mathbb{R}^{2n}$ satisfying the
condition 
\begin{multline*}
\Phi \left( x_{1},\cdots ,x_{n},\underline{x}_{1},\cdots ,\underline{x}%
_{n}\right) :\left[ \left( x_{1},\cdots ,x_{n}\right) \in E\right] \wedge %
\left[ \left( \underline{x}_{1},\cdots ,\underline{x}_{n}\right) \in 
\underline{E}\right] \wedge  \\
\left[ \left( \forall y_{1}\right) \cdots \left( \forall y_{n}\right)
\left\{ \left[ \left( y_{1},\cdots ,y_{n}\right) \in \underline{E}\right]
\rightarrow \left( x_{1}-\underline{x}_{1}\right) ^{2}+\cdots +\left( x_{n}-%
\underline{x}_{n}\right) ^{2}\leq \left( x_{1}-y_{1}\right) ^{2}+\cdots
+\left( x_{n}-y_{n}\right) ^{2}\right\} \right] \text{.}
\end{multline*}
\end{itemize}

With a single exception, all the semialgebraic sets and functions arising in
our arguments in the following sections can be computed by obvious
applications of the above standard algorithms, together with elimination of
quantifiers. When that exception arises (in the next section), we explain how to  
deal with it.

\subsection{Growth of Semialgebraic functions}\label{sec2.6}

We will use a special case of a result of \L ojasiewicz and Wachta \cite%
{wachta}.

\begin{theorem}
\label{wachta} Let $S_1, S_2 \subset \mathbb{R}^n \times \mathbb{R}^p$ be
compact semialgebraic sets. For $x \in \mathbb{R}^n$, let $S_i(x) = \{y \in 
\mathbb{R}^p:(x,y) \in S_i \}$ $(i=1,2)$. Then there exists a positive
integer $K$ for which the following holds.

Given $x \in \mathbb{R}^n$ such that $S_1 (x) \cap S_2(x) \not = \emptyset$,
there exists a positive number $C(x)$ such that 
\begin{equation*}
\dist(y, S_2(x)) \geq C(x) \cdot \left[ \dist\left(y, S_1(x) \cap S_2(x)\right)\right]^K 
\end{equation*}
for all $y \in S_1(x)$.
\end{theorem}

Remarks: The result of \cite{wachta} applies to subanalytic sets. We need
only the semialgebraic case. Our notation differs from that of \cite{wachta}.

We will apply Theorem \ref{wachta} to prove the following result.

\begin{lemma}[Growth Lemma]
\label{growthlemma}Let $E\subset \mathbb{R}^{n_{1}}$ and $E^{+}\subset
E\times \mathbb{R}^{n_{2}}$ be compact and semialgebraic, with $\dim E^+
\geq 1$. Let $A$ be a semialgebraic function on $E^{+}$. Then there exist an
integer $K\geq 1$, a semialgebraic function $A_{1}$ on $E$, and a compact
semialgebraic set $\underline{E}^{+}\subset E^{+}$, with the following
properties.

\begin{description}
\item[(GL1)] $\dim \underline{E}^{+}<\dim E^{+}$.

For $x\in E$, set $E^{+}\left( x\right) =\left\{ y\in \mathbb{R}%
^{n_{2}}:\left( x,y\right) \in E^{+}\right\} $ and $\underline{E}^{+}\left(
x\right) =\left\{ y\in \mathbb{R}^{n_{2}}:\left( x,y\right) \in \underline{E}%
^{+}\right\} $. Then, for each $x\in E$, the following hold.

\item[(GL2)] If $\underline{E}^{+}\left( x\right) $ is empty, then 
\begin{equation*}
\left\vert A\left( x,y\right) \right\vert \leq A_{1}\left( x\right) \text{
for all }y\in E^{+}\left( x\right) .
\end{equation*}

\item[(GL3)] If $\underline{E}^{+}\left( x\right) $ is non-empty, then 
\begin{equation*}
\left\vert A\left( x,y\right) \right\vert \leq A_{1}\left( x\right) \cdot 
\left[ \dist\left( y,\underline{E}^{+}\left( x\right) \right) \right] ^{-K}%
\text{ for all }y\in E^{+}\left( x\right) \setminus \underline{E}^{+}\left(
x\right) .
\end{equation*}
\end{description}
\end{lemma}

\begin{proof}
By replacing $A(x,y)$ by $1+|A(x,y)|$, we may suppose that $A(x,y) \geq 1$
for all $(x,y) \in E^+$.

We define semialgebraic sets

\begin{itemize}
\item $S_0=\left\{ (x,y, \frac{1}{A(x,y)}) \in \mathbb{R}^{n_1} \times 
\mathbb{R}^{n_2} \times \mathbb{R}: (x,y) \in E^+ \right\}$.

\item $S_1=$ closure of $S_0$ in $\mathbb{R}^{n_1} \times \mathbb{R}^{n_2}
\times \mathbb{R}$.

\item $S_2=E^+\times\{0\} \subset \mathbb{R}^{n_1} \times\mathbb{R}^{n_2}
\times\mathbb{R}$.
\end{itemize}

In particular, $S_1, S_2$ are compact semialgebraic subsets of $\mathbb{R}%
^{n_1} \times \mathbb{R}^{n_2}\times \mathbb{R}$, so Theorem \ref{wachta} applies.

Observe that $S_1 \cap S_2$ has the form $\underline{E}^+ \times\{0 \}$ for
a compact semialgebraic set $\underline{E}^+ \subset E^+$. Moreover, $S_1
\cap S_2 \subset S_0^{\text{closure}} \setminus S_0$, hence $\dim \underline{%
E}^+ = \dim S_1 \cap S_2 \leq \dim (S_0^{\text{closure}} \setminus S_0) <
\dim S_0 = \dim E^+$.

Thus, $\dim \underline{E}^+ < \dim E^+$.

Now let $x \in E$ be given. We write $\underline{E}^+(x) = \left\{y \in 
\mathbb{R}^{n_2}: (x,y) \in \underline{E}^+ \right\}$.

\underline{Case 1}: Suppose $\underline{E}^+(x) = \emptyset$. Then $S_1(x)
\cap S_2(x)= \emptyset$.

Therefore, $S_0(x)$ avoids a neighborhood of $S_2(x)$ in $\mathbb{R}^{n_1}
\times \mathbb{R}^{n_2} \times \mathbb{R}$, which implies that $\frac{1}{%
A(x,y)}$ avoids a neighborhood of zero as $y$ varies over $E^+(x)$.

Therefore, for some constant $B^x$, we have

\begin{itemize}
\item[\refstepcounter{equation}\text{(\theequation)}\label{wa2}] {$%
|A(x,y)|=A(x,y) \leq B^x$ for all $y \in E^+(x)$.}
\end{itemize}

Estimate \eqref{wa2} holds if $\underline{E}^+(x) =\emptyset$.

\underline{Case 2}: Suppose $\underline{E}^+(x) \not=\emptyset$. Then $%
S_1(x) \cap S_2(x) \not=\emptyset$.

Let $K$ and $C(x)$ be as in Theorem \ref{wachta}. Then, for any $y \in E^+(x)
$, we have $(y,\frac{1}{A(x,y)}) \in S_1(x)$ and $(y,0) \in S_2(x)$,
therefore

\begin{eqnarray*}
\frac{1}{A(x,y)} &=& \dist\left((y,\frac{1}{A(x,y)}),(y,0)\right) \geq \dist%
\left((y,\frac{1}{A(x,y)}),S_2(x)\right) \\
&\geq& C(x) \cdot \left[ \dist\left((y, \frac{1}{A(x,y)}),S_1(x) \cap
S_2(x)\right) \right]^K \\
&=& C(x) \cdot \left[ \dist\left((y, \frac{1}{A(x,y)}), \underline{E}^+(x)
\times\{0\}\right) \right]^K \\
&\geq& C(x) \cdot \left[ \dist\left(y, \underline{E}^+(x) \right) \right]^K 
\text{ with } C(x) >0.
\end{eqnarray*}
Thus, for some constant $B^x$, we have 
\begin{itemize}
\item[\refstepcounter{equation}\text{(\theequation)}\label{wa3}] {$%
|A(x,y)|=A(x,y) \leq \frac{B^x}{\left[\dist\left(y, \underline{E}%
^+(x)\right) \right]^K}$ for all $y \in E^+(x) \setminus \underline{E}^+(x)$.%
}
\end{itemize}

Estimate \eqref{wa3} holds if $\underline{E}^+(x) \not=\emptyset$.

For $x \in E$, we now set 
\begin{equation*}
A_1(x) =\left. 
\begin{cases}
\text{smallest } B^+ \geq 0 \text{ for which } \eqref{wa2} \text{ holds, if }%
\underline{E}^+(x) = \emptyset \\ 
\text{smallest }B^+ \geq 0 \text{ for which }\eqref{wa3} \text{ holds, if } 
\underline{E}^+(x) \not=\emptyset%
\end{cases}
\right\}. 
\end{equation*}

Then $A_1(x)$ is a non-negative semialgebraic function, and we have 
\begin{eqnarray*}
|A(x,y)| &\leq& A_1(x) \text{ if } y \in E^+(x_0) \text{ and } \underline{E}%
^+(x_0) = \emptyset, \\
|A(x,y)| &\leq& \frac{A_1(x)}{\left[\dist\left(y, \underline{E}^+(x)\right) %
\right]^K} \text{ if } y \in E^+(x_0) \setminus \underline{E}^+(x_0) \text{
and } \underline{E}^+(x_0) \not=\emptyset.
\end{eqnarray*}

The proof of Lemma \ref{growthlemma} is complete.
\end{proof}

We thank M. Aschenbrenner and W. Paw\l ucki for pointing out that Lemma \ref%
{growthlemma} follows easily from known results, thus subtracting
approximately 20 pages from this paper. Paw\l ucki supplied the above proof
of Lemma \ref{growthlemma} based on \cite{wachta}.

We indicate how $\underline{E}^+,K, A_1$ in Lemma \ref{growthlemma} can be computed. The delicate  
point is the computation of $K$.

Proceeding as in the proof of Lemma \ref{growthlemma}, we first replace $A$ by  
$1+|A|$, then compute $\underline{E}^+$.

Given any positive integer $K$, we can then decide whether the following  
hold for each $x \in E$.

\LAQ{2.5.3}{$\sup\{A(x,y):y \in E^+(x)\}$ finite if $\underline{E}^+(x)$ is empty.}
\LAQ{2.5.4}{$\sup \{A(x,y)\cdot [\dist(y, \underline{E}^+(x))]^K: y \in E^+(x)\}$ finite if  
$\underline{E}^+(x)$ is nonempty.}

We successively test $K=1,2,3,\cdots$ until we find a $K$ for which \eqref{2.5.3}  
and \eqref{2.5.4} hold. We will eventually find such a $K$, thanks to the  
proof of Lemma \ref{growthlemma}.

Once we have found $\underline{E}^+$ and $K$, we can compute the function $A_1$ defined  
in the proof of Lemma \ref{growthlemma}. Thus, we compute $\underline{E}^+,K, A_1$ as  promised.

\section{Semialgebraic Quadratic Forms and $C^{\infty }$ Functions}\label{sec3}

\subsection{Setup\label{section-setup}}

\label{quadratic-form-set-up-section}

We are given the following objects and assumptions:

\begin{itemize}
\item[\refstepcounter{equation}\text{(\theequation)}\label{PSQF1}] $E\subset 
\mathbb{R}^{n}$, $E^{+}\subset E\times \underset{\text{k copies}}{%
\underbrace{\mathbb{R}^{n}\times \cdots \times \mathbb{R}^{n}}}$ , $%
E^{++}\subset E^{+}\times \underset{L\text{ copies}}{\underbrace{\mathbb{R}%
^{n}\times \cdots \times \mathbb{R}^{n}}}$.

\item[\refstepcounter{equation}\text{(\theequation)}\label{PSQF2}] $%
E,E^{+},E^{++}$ are semialgebraic.

\item[\refstepcounter{equation}\text{(\theequation)}\label{PSQF3}] $E$ and $%
E^{+}$ are compact. We do not assume $E^{++}$ compact.

\item[\refstepcounter{equation}\text{(\theequation)}\label{PSQF4}] $\bar{m}%
\geq m \geq 0 $ integers. $D,I$ are positive integers.

\item[\refstepcounter{equation}\text{(\theequation)}\label{PSQF5}] $Q\left(
x_{0},P_{0},x_{1},P_{1},\cdots ,x_{k,}P_{k},z_{1},\cdots ,z_{L}\right) $, a
semialgebraic function of $\left( x_{0},x_{1},\cdots ,x_{k},z_{1},\cdots
,z_{L}\right) \in E^{++}$, $P_{0}\in \mathcal{P}^{\left( m\right) }\left( 
\mathbb{R}^{n},\mathbb{R}^{D}\right) $, $P_{1},\cdots ,P_{k}\in \mathcal{P}%
^{\left( \bar{m}\right) }\left( \mathbb{R}^{n},\mathbb{R}^{I}\right) $.

\item[\refstepcounter{equation}\text{(\theequation)}\label{PSQF6}] For fixed 
$\left( x_{0},x_{1},\cdots ,x_{k},z_{1},\cdots ,z_{L}\right) \in E^{++}$,
the map 
\begin{equation*}
\left( P_{0},P_{1},\cdots ,P_{k}\right) \mapsto Q\left(
x_{0},P_{0},x_{1},P_{1},\cdots ,x_{k,}P_{k},z_{1},\cdots ,z_{L}\right)
\end{equation*}
is a positive semidefinite quadratic form.

\item[\refstepcounter{equation}\text{(\theequation)}\label{PSQF7}] $A\left(
x_{0},\cdots ,x_{k}\right) $ is a nonnegative semialgebraic function on $%
E^{+}$.

\item[\refstepcounter{equation}\text{(\theequation)}\label{PSQF8}] 
\begin{equation*}
Q\left( x_{0},P_{0},x_{1},P_{1},\cdots ,x_{k,}P_{k},z_{1},\cdots
,z_{L}\right) \leq A\left( x_{0},\cdots ,x_{k}\right) \cdot \left[
\sum_{|\alpha |\leq m}|\partial ^{\alpha
}P_{0}(x_{0})|^{2}+\sum_{i=1}^{k}\sum_{\left\vert \alpha \right\vert \leq 
\bar{m}}\left\vert \partial ^{\alpha }P_{i}\left( x_{i}\right) \right\vert
^{2}\right]
\end{equation*}%
for all $\left( x_{0},\cdots ,x_{k},z_{1},\cdots ,z_{L}\right) \in E^{++}$, $%
P_{0}\in \mathcal{P}^{\left( m\right) }\left( \mathbb{R}^{n},\mathbb{R}%
^{D}\right) $, $P_{1},\cdots ,P_{k}\in \mathcal{P}^{\left( \bar{m}\right)
}\left( \mathbb{R}^{n},\mathbb{R}^{I}\right) $.
\end{itemize}

For $x_0 \in E$, we define

\begin{equation*}
E^+(x_0)=\{(x_1, \cdots, x_k) \in \mathbb{R}^n \times \cdots \times \mathbb{R%
}^n: (x_0, x_1, \cdots, x_k) \in E^+\}
\end{equation*}
and 
\begin{equation*}
E^{++}(x_0)=\{(x_1, \cdots, x_k,z_1, \cdots, z_L) \in \mathbb{R}^n \times
\cdots \times \mathbb{R}^n: (x_0, x_1, \cdots, x_k,z_1,\cdots, z_L) \in
E^{++}\}.
\end{equation*}

In \eqref{PSQF1}, \eqref{PSQF5}, \eqref{PSQF6}, and \eqref{PSQF8}, we allow $%
L=0$. That is, \eqref{PSQF5}, \eqref{PSQF6}, and \eqref{PSQF8} need not
involve $z$'s.

\subsection{Statement of main propositions on semialgebraic quadratic forms}

\label{statement-quadform}

Under the assumptions of Section \ref{section-setup}, we have the following
results:

\begin{proposition}
\label{SAQF1}There exist an integer $\bar{\bar{m}}\geq \bar{m}$ and a family
of vector spaces 
\begin{equation*}
H^{\text{bdd}}\left( x_{0},\cdots ,x_{k}\right) \subset \mathcal{P}^{\left(
m\right) }\left( \mathbb{R}^{n},\mathbb{R}^{D}\right) \oplus \underset{\text{%
k copies}}{\underbrace{\mathcal{P}^{\left( \bar{\bar{m}}\right) }\left( 
\mathbb{R}^{n},\mathbb{R}^{I}\right) \oplus \cdots \oplus \mathcal{P}%
^{\left( \bar{\bar{m}}\right) }\left( \mathbb{R}^{n},\mathbb{R}^{I}\right) }}
\end{equation*}%
depending semialgebraically on $\left( x_{0},x_{1},\cdots ,x_{k}\right) \in
E^{+}$, such that the following holds. \newline
Let $x_{0}\in E$, $P_{0}\in \mathcal{P}^{\left( m\right) }\left( \mathbb{R}%
^{n},\mathbb{R}^{D}\right) $, $F\in C^{\infty }_0\left( \mathbb{R}^{n},%
\mathbb{R}^{I}\right) $. Then 
\begin{equation}
\sup \left\{ 
\begin{array}{c}
Q\left( x_{0},P_{0},y_{1},J_{y_{1}}^{\left( \bar{m}\right) }F,\cdots
,y_{k,}J_{y_{k}}^{\left( \bar{m}\right) }F,z_{1},\cdots ,z_{L}\right) : \\ 
\left( y_{1},\cdots ,y_{k},z_{1},\cdots ,z_{L}\right) \in E^{++}(x_{0})%
\end{array}%
\right\} <\infty  \label{PSQFx1}
\end{equation}%
if and only if 
\begin{equation}
\left( P_{0},J_{y_{1}}^{\left( \bar{\bar{m}}\right) }F,\cdots
,J_{y_{k}}^{\left( \bar{\bar{m}}\right) }F\right) \in H^{\text{bdd}}\left(
x_{0},y_{1},\cdots ,y_{k}\right)  \label{PSQFx2}
\end{equation}%
for each $\left( y_{1},\cdots ,y_{k}\right) \in E^{+}(x_{0})$.

Moreover, the above family of vector spaces $H^{\text{bdd}}$ can be computed from  
the data provided in Section \ref{section-setup}.
\end{proposition}
\begin{proposition}
\label{SAQF2}There exist $m^{+}\geq \bar{m}$ and a family of vector
subspaces 
\begin{equation*}
H^{\lim }\left( x_{0}\right) \subset \mathcal{P}^{\left( m\right) }\left( 
\mathbb{R}^{n},\mathbb{R}^{D}\right) \oplus \mathcal{P}^{\left( m^{+}\right)
}\left( \mathbb{R}^{n},\mathbb{R}^{I}\right) ,
\end{equation*}%
depending semialgebraically on $x_{0}\in E$, such that the following holds. 
\newline
Let $x_{0}\in E$, $P_{0}\in \mathcal{P}^{\left( m\right) }\left( \mathbb{R}%
^{n},\mathbb{R}^{D}\right) $, $F\in C^{\infty }_0 \left( \mathbb{R}^{n},%
\mathbb{R}^{I}\right) $. Assume that \eqref{PSQFx1} holds. Then%
\begin{equation}
\lim_{\substack{ y_{1},\cdots ,y_{k},z_{1},\cdots ,z_{L}\rightarrow x_{0} 
\\ \left( y_{1},\cdots ,y_{k},z_{1},\cdots ,z_{L}\right) \in E^{++}(x_{0})}}%
Q\left( x_{0},P_{0},y_{1},J_{y_{1}}^{\left( \bar{m}\right) }F,\cdots
,y_{k},J_{y_{k}}^{\left( \bar{m}\right) }F,z_{1},\cdots ,z_{L}\right) =0
\label{PSQFxx1}
\end{equation}%
if and only if 
\begin{equation}
\left( P_{0},J_{x_{0}}^{\left( m^{+}\right) }F\right) \in H^{\lim }\left(
x_{0}\right) \text{.}  \label{PSQFxx2}
\end{equation}

Moreover, the above family of vector spaces $H^{\lim}$ can be computed from  
the data provided in Section \ref{section-setup}.
\end{proposition}
Note: In Proposition \ref{SAQF1}, the $\sup $ (\ref{PSQFx1}) is taken to be $%
0$ if $E^{++}\left( x_{0}\right) $ is empty.

\section{Proof of Propositions \protect\ref{SAQF1} and \protect\ref{SAQF2}}

\label{proofofprops23}

\begin{proof}
We prove Propositions \ref{SAQF1} and {\ref{SAQF2}} by induction on $\dim
E^{+}$.

All the semialgebraic sets and functions arising in our proof of  
Propositions \ref{SAQF1} and \ref{SAQF2} will be computable by the methods of Sections \ref{sec2.5}  
and \ref{sec2.6}.

\underline{In the base case,} $\dim E^{+}=0$, i.e., $E^{+}$ is finite.

In that case, hypothesis \eqref{PSQF8} from Section \ref{section-setup}
asserts that for some constant $A$, we have 
\begin{equation*}
Q\left( x_{0},P_{0},x_{1},P_{1},\cdots ,x_{k,}P_{k},z_{1},\cdots
,z_{L}\right) \leq A\cdot \left[\sum_{|\alpha| \leq m } \vert
\partial^\alpha P_0 (x_0) \vert^2+ \sum_{i=1}^{k}\sum_{\left\vert \alpha
\right\vert \leq \bar{m}}\left\vert \partial ^{\alpha }P_{i}\left(
x_{i}\right) \right\vert ^{2}\right]
\end{equation*}%
for all $\left( x_{0},\cdots ,x_{k},z_{1},\cdots ,z_{L}\right) \in E^{++}$
and all $P_0, P_1, \cdots, P_k$.

Consequently \eqref{PSQFx1} in Proposition \ref{SAQF1} is satisfied, for any 
$F\in C_{0}^{\infty }\left( \mathbb{R}^{n},\mathbb{R}^{I}\right) $. We take $%
\bar{\bar{m}}=\bar{m}$ and 
\begin{equation*}
H^{\text{bdd}}\left( x_{0},y_{1},\cdots ,y_{k}\right) =\mathcal{P}^{\left(
m\right) }\left( \mathbb{R}^{n},\mathbb{R}^{D}\right) \oplus \underset{k%
\text{ copies}}{\underbrace{\mathcal{P}^{\left( \bar{m}\right) }\left( 
\mathbb{R}^{n},\mathbb{R}^{I}\right) \oplus \cdots \oplus \mathcal{P}%
^{\left( \bar{m}\right) }\left( \mathbb{R}^{n},\mathbb{R}^{I}\right) }}\text{%
.}
\end{equation*}%
Then \eqref{PSQFx1} and \eqref{PSQFx2} both hold for any $F\in C_{0}^{\infty
}\left( \mathbb{R}^{n},\mathbb{R}^{I}\right) $, proving Proposition \ref%
{SAQF1} in the base case $\dim E^{+}=0$.

To prove Proposition \ref{SAQF2} in the base case, we note that since $E^{+}$
is finite, condition (\ref{PSQFxx1}) in Proposition \ref{SAQF2} is
equivalent to the following:%
\begin{equation*}
\lim_{\substack{ z_{1},\cdots ,z_{L}\rightarrow x_{0}  \\ \left(
x_{0},x_{0},\cdots ,x_{0},z_{1},\cdots ,z_{L}\right) \in E^{++}}}Q\left(
x_{0},P_{0},x_{0},J_{x_{0}}^{\left( \bar{m}\right) }F,\cdots
,x_{0},J_{x_{0}}^{\left( \bar{m}\right) }F,z_{1},\cdots ,z_{L}\right) =0.
\end{equation*}%
We define $H^{\lim }\left( x_{0}\right) $ to consist of all $\left(
P_{0},P\right) \in \mathcal{P}^{\left( m\right) }\left( \mathbb{R}^{n},%
\mathbb{R}^{D}\right) \oplus \mathcal{P}^{\left( \bar{m}\right) }\left( 
\mathbb{R}^{n},\mathbb{R}^{I}\right) $ such that 
\begin{equation*}
\lim_{\substack{ z_{1},\cdots ,z_{L}\rightarrow x_{0}  \\ \left(
x_{0},x_{0},\cdots ,x_{0},z_{1},\cdots ,z_{L}\right) \in E^{++}}}Q\left(
x_{0},P_{0},x_{0},P,\cdots ,x_{0},P,z_{1},\cdots ,z_{L}\right) =0.
\end{equation*}%
Then, by taking $m^{+}=\bar{m}$, we see that (\ref{PSQFxx1}) is equivalent
to (\ref{PSQFxx2}) as in Proposition \ref{SAQF2}. Note that $H^{\lim }\left(
x_{0}\right) $ is a vector subspace of $\mathcal{P}^{\left( m\right) }\left( 
\mathbb{R}^{n},\mathbb{R}^{D}\right) \oplus \mathcal{P}^{\left( \bar{m}%
\right) }\left( \mathbb{R}^{n},\mathbb{R}^{I}\right) $, since 
\begin{equation*}
\left( P_{0},P_{1},\cdots ,P_{k}\right) \mapsto Q\left( x_{0},P_{0},\cdots
,x_{k},P_{k},z_{1},\cdots ,z_{L}\right)
\end{equation*}%
is a semidefinite quadratic form for each $\left( x_{0},x_{1},\cdots
,x_{k},z_{1},\cdots ,z_{L}\right) \in E^{++}$.

The semialgebraic dependence of $H^{\lim }\left( x_{0}\right) $ on $x_{0}\in
E$ is trivial, since $E$ is finite (because $E^{+}$ is finite).

This completes the proof of Propositions \ref{SAQF1} and \ref{SAQF2} in the
base case ($\dim E^{+}=0$).

\underline{For the induction step}, we fix a positive integer $\Delta $, and
assume that Propositions \ref{SAQF1} and \ref{SAQF2} hold whenever $\dim
E^{+}<\Delta $. We will prove those propositions in the case $\dim
E^{+}=\Delta $.

Let us assume that all is as in Section \ref{section-setup}, and that $\dim
E^{+}=\Delta $.

We apply Lemma \ref{growthlemma} to the semialgebraic function $A\left(
x_{0},x_{1},\cdots ,x_{k}\right) $ in (\ref{PSQF7}), (\ref{PSQF8}) of
Section \ref{section-setup}.

Thus, there exist an integer $K\geq 1$, a compact semialgebraic subset $%
\underline{E}^{+}\subset E^{+}$, and a semialgebraic function $A_{1}$
defined on $E$, having the following properties:

\begin{itemize}
\item[\refstepcounter{equation}\text{(\theequation)}\label{P1}] $\dim 
\underline{E}^{+}<\Delta $.
\end{itemize}

Let $x_{0}\in E$. Recall that

\begin{itemize}
\item[\refstepcounter{equation}\text{(\theequation)}\label{P2}] $E^{+}\left(
x_{0}\right) =\left\{ \left( x_{1},\cdots ,x_{k}\right) \in \mathbb{R}%
^{n}\times \cdots \times \mathbb{R}^{n}:\left( x_{0},x_{1},\cdots
,x_{k}\right) \in E^{+}\right\}$.
\end{itemize}

Define

\begin{itemize}
\item[\refstepcounter{equation}\text{(\theequation)}\label{P3}] $\underline{E%
}^{+}\left( x_{0}\right) =\left\{ \left( x_{1},\cdots ,x_{k}\right) \in 
\mathbb{R}^{n}\times \cdots \times \mathbb{R}^{n}:\left( x_{0},x_{1},\cdots
,x_{k}\right) \in \underline{E}^{+}\right\} $.
\end{itemize}

Then, then for each $x_{0}\in E$, the following hold.

\begin{itemize}
\item[\refstepcounter{equation}\text{(\theequation)}\label{P4}] If $%
\underline{E}^{+}\left( x_{0}\right) $ is empty, then 
\begin{eqnarray*}
&&\left\vert A\left( x_{0},x_{1},\cdots ,x_{k}\right) \right\vert \\
&\leq &A_{1}\left( x_{0}\right) \text{ for all }\left( x_{1},\cdots
,x_{k}\right) \in E^{+}\left( x_{0}\right) \text{.}
\end{eqnarray*}

\item[\refstepcounter{equation}\text{(\theequation)}\label{P5}] If $%
\underline{E}^{+}\left( x_{0}\right) $ is nonempty, then%
\begin{eqnarray*}
&&\left\vert A\left( x_{0},x_{1},\cdots ,x_{k}\right) \right\vert \\
&\leq &A_{1}\left( x_{0}\right) \cdot \left[ \dist\left( \left( x_{1},\cdots
,x_{k}\right) ,\underline{E}^{+}\left( x_{0}\right) \right) \right] ^{-K}%
\text{ for all }\left( x_{1},\cdots ,x_{k}\right) \in E^{+}\left(
x_{0}\right) \setminus \underline{E}^{+}\left( x_{0}\right) \text{.}
\end{eqnarray*}
\end{itemize}

Let

\begin{itemize}
\item[\refstepcounter{equation}\text{(\theequation)}\label{P6}] $\underline{E%
}=\left\{ x_{0}\in E:\underline{E}^{+}\left( x_{0}\right) \not=\emptyset
\right\} $.
\end{itemize}

Note that $\underline{E}$ and $\underline{E}^{+}$ are compact semialgebraic
sets, and that

\begin{itemize}
\item[\refstepcounter{equation}\text{(\theequation)}\label{P7}] $\underline{E%
}^{+}\subset \underline{E}\times \underset{k\text{ copies}}{\underbrace{%
\mathbb{R}^{n}\times \cdots \times \mathbb{R}^{n}}}$.

\item[\refstepcounter{equation}\text{(\theequation)}\label{P8}] Let $%
\underline{E}^{++}=\left\{ \left( x_{0},\cdots ,x_{k},z_{1},\cdots
,z_{L}\right) \in E^{++}:\left( x_{0},\cdots ,x_{k}\right) \in \underline{E}%
^{+}\right\} $.
\end{itemize}

For $x_0 \in \underline{E}$, define

\begin{itemize}
\item[\refstepcounter{equation}\text{(\theequation)}\label{P8.a}] $%
\underline{E}^{++}(x_0)=\{(y_1, \cdots, y_k,z_1,\cdots,z_L) \in \mathbb{R}^n
\times\cdots \times \mathbb{R}^n:(x_0,y_1,\cdots,y_k,z_1,\cdots, z_L) \in 
\underline{E}^{++} \}. $
\end{itemize}

Thus, $\underline{E}^{++}$ is a semialgebraic subset of $\underline{E}%
^{+}\times \underset{L\text{ copies}}{\underbrace{\mathbb{R}^{n}\times
\cdots \times \mathbb{R}^{n}}}$.

Thanks to $\eqref{P1}$, our induction hypothesis applies to $%
Q(x_0,P_0,x_1,P_1,\cdots,x_k,P_k,z_1,\cdots,z_k)$ restricted to $%
(x_0,\cdots,x_k,z_1,\cdots,z_L) \in \underline{E}^{++}$. Applying
Proposition \ref{SAQF1}, we therefore learn the following.

There exist $\underline{m}^{\prime }\geq \bar{m}$, and a computable family of vector spaces
\begin{equation}
\underline{H}^{\text{bdd}}\left( x_{0},\cdots ,x_{k}\right) \subset \mathcal{%
P}^{\left( m\right) }\left( \mathbb{R}^{n},\mathbb{R}^{D}\right) \oplus 
\underset{k\text{ copies}}{\underbrace{\mathcal{P}^{\left( \underline{m}%
^{\prime }\right) }\left( \mathbb{R}^{n},\mathbb{R}^{I}\right) \oplus \cdots
\oplus \mathcal{P}^{\left( \underline{m}^{\prime }\right) }\left( \mathbb{R}%
^{n},\mathbb{R}^{I}\right) }},  \label{P9}
\end{equation}%
depending semialgebraically on $\left( x_{0},\cdots ,x_{k}\right) \in 
\underline{E}^{+}$, such that the following holds.

Let $x_{0}\in \underline{E}$, $P_{0}\in \mathcal{P}^{\left( m\right) }\left( 
\mathbb{R}^{n},\mathbb{R}^{D}\right) $, $F\in C_{0}^{\infty }\left( \mathcal{%
\mathbb{R}}^{n},\mathbb{R}^{I}\right) $.

Then 
\begin{equation}
\sup \left\{ Q\left( x_{0},P_{0},y_{1},J_{y_{1}}^{\left( \bar{m}\right)
}F,\cdots ,y_{k},J_{y_{k}}^{\left( \bar{m}\right) }F,z_{1},\cdots
,z_{L}\right) :\left(y_{1},\cdots ,y_{k},z_{1},\cdots ,z_{L}\right) \in 
\underline{E}^{++}(x_0)\right\} <\infty  \label{P10a}
\end{equation}%
if and only if 
\begin{equation}
\left( P_{0},J_{y_{1}}^{\left( \underline{m}^{\prime }\right) }F,\cdots
,J_{y_{k}}^{\left( \underline{m}^{\prime }\right) }F\right) \in \underline{H}%
^{\text{bdd}}\left( x_{0},y_{1},\cdots ,y_{k}\right)  \label{P10b}
\end{equation}%
for each $\left( y_{1},\cdots ,y_{k}\right) \in \underline{E}^{+}(x_0)$.

Later on, we will apply Proposition \ref{SAQF2} in the same setting; see %
\eqref{P70} below.

Let us sketch the arguments that follow, concentrating on the proof of
Proposition \ref{SAQF1}.

Given $x_{0},P_{0},F,$ we must decide whether the quantity 
\begin{equation}
Q\left( x_{0},P_{0},x_{1},J_{x_1}^{\left( \bar{m}\right) }F_{1},\cdots
,x_{k},J_{x_k}^{\left( \bar{m}\right) }F,z_{1},\cdots ,z_{L}\right)
\label{add1}
\end{equation}

stays bounded as $\left( x_{1},\cdots ,x_{k},z_{1},\cdots ,z_{L}\right) $
varies over $E^{++}\left( x_{0}\right) $.

If we restrict attention to the set of such $\left( x_{1},\cdots
,x_{k},z_{1},\cdots ,z_{L}\right) $ where $\left( x_{1},\cdots ,x_{k}\right)
\in \underline{E}^{+}\left( x_{0}\right) $, then already the equivalence of %
\eqref{P10a} to \eqref{P10b} settles the issue. Hence, we may restrict
attention to the set where $\left( x_{1},\cdots ,x_{k}\right) \not\in 
\underline{E}^{+}\left( x_{0}\right) $. Also, \eqref{PSQF8} and \eqref{P4}
easily imply that the quantity (\ref{add1}) stays bounded whenever $%
\underline{E}^{+}\left( x_{0}\right) $ is empty. Hence, we may assume that $%
\underline{E}^{+}\left( x_{0}\right) $ is non-empty.

Thus, our problem is to decide whether the quantity (\ref{add1}) stays
bounded as $\left( x_{1},\cdots ,x_{k},z_{1},\cdots ,z_{L}\right) $ varies
over the set where

\begin{itemize}
\item[\refstepcounter{equation}\text{(\theequation)}\label{add2}] {$\left(
x_{1},\cdots ,x_{k},z_{1},\cdots ,z_{L}\right) \in E^{++}\left( x_{0}\right) 
$, $\left( x_{1},\cdots ,x_{k}\right) \not\in \underline{E}^{+}\left(
x_{0}\right) ,$ assuming $\underline{E}^{+}\left( x_{0}\right)
\not=\emptyset $.}
\end{itemize}

In deciding this question, we are fighting against the factor $\left[ \dist%
\left( \left( x_{1},\cdots ,x_{k}\right) ,\underline{E}^{+}\left(
x_{0}\right) \right) \right] ^{-K}$ in estimate \eqref{P5}.

Our strategy is to pick $\bar{\bar{m}}$ much larger than $\bar{m}$%
, and approximate $J_{x_i}^{\left( \bar{m}\right) }F$ by $\pi _{x_{i}}^{%
\bar{\bar{m}}\rightarrow \bar{m}}J_{\underline{x}_{i}}^{\left( 
\bar{\bar{m}}\right) }F$ in (\ref{add1}), where $\left( \underline{%
x}_{1},\cdots ,\underline{x}_{k}\right) $ is a point of $\underline{E}%
^{+}\left( x_{0}\right) $ lying as close as possible to $\left( x_{1},\cdots
,x_{k}\right) $. According to Proposition \ref{taylor'sthm}, the derivatives
of the error $J_{x_{i}}^{\left( \bar{m}\right) }F-\pi _{x_{i}}^{\overline{%
\overline{m}}\rightarrow \bar{m}}J_{\underline{x}_{i}}^{\left( \overline{%
\overline{m}}\right) }F$ at $x_{i}$ are 
\begin{eqnarray*}
&&O\left( \left[ \dist\left( \left( x_{1},\cdots ,x_{k}\right) ,\left( 
\underline{x}_{1},\cdots ,\underline{x}_{k}\right) \right) \right] ^{%
\bar{\bar{m}}-\bar{m}}\right) \\
&=&O\left( \left[ \dist \left( ( x_{1},\cdots ,x_{k}),\underline{E}%
^{+}\left( x_{0}\right) \right) \right] ^{\bar{\bar{m}}-\overline{m%
}}\right) \text{.}
\end{eqnarray*}%
If we pick $\bar{\bar{m}}$ large enough, then the small factor $%
\left[ \dist\left( \cdots \right) \right] ^{\bar{\bar{m}}-\bar{m}}$
overcomes the large factor $\left[ \dist\left( \cdots \right) \right] ^{-K}$
in \eqref{P5}.

Therefore, the induction step in the proof of Proposition \ref{SAQF1} comes
down to deciding whether the quantity 
\begin{equation}
Q\left( x_{0},P_{0},x_{1},\pi _{x_{1}}^{\bar{\bar{m}}\rightarrow 
\bar{m}}J_{\underline{x}_{1}}^{\left( \bar{\bar{m}}\right)
}F,\cdots ,x_{k},\pi _{x_{k}}^{\bar{\bar{m}}\rightarrow \overline{m%
}}J_{\underline{x}_{k}}^{\left( \bar{\bar{m}}\right)
}F,z_{1},\cdots ,z_{L}\right)  \label{add3}
\end{equation}%
remains bounded as $\left( \underline{x}_{1},\cdots ,\underline{x}%
_{k},z_{1},\cdots ,z_{L},x_{1},\cdots ,x_{k}\right) $ varies over the set
where%
\begin{equation}
\left[ 
\begin{array}{c}
\left( x_{1},\cdots ,x_{k},z_{1},\cdots ,z_{L}\right) \in E^{++}\left(
x_{0}\right) ,\left( x_{1},\cdots ,x_{k}\right) \not\in \underline{E}%
^{+}\left( x_{0}\right) ,\left( \underline{x}_{1},\cdots ,\underline{x}%
_{k}\right) \in \underline{E}^{+}\left( x_{0}\right) , \\ 
\dist\left( \left( x_{1},\cdots ,x_{k}\right) ,\underline{E}%
^{+}(x_{0})\right) =\dist\left( \left( x_{1},\cdots ,x_{k}\right) ,\left( 
\underline{x}_{1},\cdots ,\underline{x}_{k}\right) \right)%
\end{array}%
\right] \text{.}  \label{add4}
\end{equation}

We are therefore tempted to define

\begin{eqnarray}
&&\tilde{Q}\left( x_{0},P_{0},\underline{x}_{1},P_{1},\cdots ,\underline{x}%
_{k},P_{k},z_{1},\cdots ,z_{L},x_{1},\cdots ,x_{k}\right)  \label{add5} \\
&=&Q\left( x_{0},P_{0},x_{1},\pi _{x_{1}}^{\bar{\bar{m}}%
\rightarrow \bar{m}}P_{1},\cdots ,{x}_{k},\pi _{x_{k}}^{\overline{%
\overline{m}}\rightarrow \bar{m}}P_{k},z_{1},\cdots ,z_{L}\right)  \notag
\end{eqnarray}%
for $\left( \underline{x}_{1},\cdots ,\underline{x}_{k},z_{1},\cdots
,z_{L},x_{1},\cdots ,x_{k}\right) $ as in (\ref{add4}), and for $P_{0}\in 
\mathcal{P}^{\left( m\right) }\left( \mathbb{R}^{n},\mathbb{R}^{D}\right) $, 
$P_{1},\cdots ,P_{k}\in \mathcal{P}^{\left( \bar{\bar{m}}\right)
}\left( \mathbb{R}^{n},\mathbb{R}^{I}\right) $.

Our problem is then to decide whether the quantity 
\begin{equation}
\tilde{Q}\left( x_{0},P_{0},\underline{x}_{1},J_{\underline{x}_{1}}^{\left( 
\bar{\bar{m}}\right) }F,\cdots ,\underline{x}_{k},J_{\underline{x}%
_{k}}^{\left( \bar{\bar{m}}\right) }F,z_{1},\cdots
,z_{L},x_{1},\cdots ,x_{k}\right)  \label{add6}
\end{equation}%
remains bounded as $\left( \underline{x}_{1},\cdots ,\underline{x}%
_{k},z_{1},\cdots ,z_{L},x_{1},\cdots ,x_{k}\right) $ varies over the set $%
\hat{E}^{++}\left( x_{0}\right) $ defined by (\ref{add4}). We hope to decide
this question by applying Proposition \ref{SAQF1} to the quadratic form $%
\tilde{Q}$ and the sets 
\begin{equation}
\left[ 
\begin{array}{l}
\hat{E}^{++}=\left\{ \left( x_{0},\underline{x}_{1},\cdots ,\underline{x}%
_{k},z_{1},\cdots ,z_{L},x_{1},\cdots ,x_{k}\right) :\text{ Conditions (\ref%
{add4}) hold}\right\} \\ 
\hat{E}^{+}=\underline{E}^{+} \\ 
\hat{E}=\underline{E}\text{, see \eqref{P6}}%
\end{array}%
\right] \text{,}  \label{add7}
\end{equation}%
in place of $Q,E^{++},E^{+},E$.

Here, our present $\left( \underline{x}_{1},\cdots ,\underline{x}_{k}\right) 
$ plays the r\^ole of $\left( x_{1},\cdots ,x_{k}\right) $ in the statement of
Proposition \ref{SAQF1}, while our present $\left( z_{1},\cdots
,z_{L},x_{1},\cdots ,x_{k}\right) $ plays the r\^ole of $\left( z_{1},\cdots
,z_{L}\right) $. The key point is that $\dim \hat{E}^{+}=\dim \underline{E}^+%
<\Delta $, by \eqref{P1}, so we can hope to apply our induction hypothesis
(Proposition \ref{SAQF1} holds when $\dim E^{+}<\Delta $). If we could apply
Proposition \ref{SAQF1} to $\tilde{Q},\hat{E}^{++},\hat{E}^{+},\hat{E}$,
then we could decide whether the quantity \eqref{add6} satisfies the
required boundedness condition, and thus complete our inductive proof of
Proposition \ref{SAQF1}.

Unfortunately, the above doesn't quite work, because the quadratic form $%
\tilde{Q}$ needn't satisfy the analogue of hypothesis \eqref{PSQF8} of
Proposition \ref{SAQF1}.

We rescue the argument by modifying $\tilde{Q}$. Given $\left( x_{0},%
\underline{x}_{1},\cdots ,\underline{x}_{k}\right) $ we define a projection $%
\Pi _{\left( x_{0},\underline{x}_{1},\cdots ,\underline{x}_{k}\right) }$
from 
\begin{equation*}
\mathcal{P}^{\left( m\right) }\left( \mathbb{R}^{n},\mathbb{R}^{D}\right)
\oplus \underset{k\text{ copies}}{\underbrace{\mathcal{P}^{\left( \overline{%
\overline{m}}\right) }\left( \mathbb{R}^{n},\mathbb{R}^{I}\right) \oplus
\cdots \oplus \mathcal{P}^{\left( \bar{\bar{m}}\right) }\left( 
\mathbb{R}^{n},\mathbb{R}^{I}\right) }}
\end{equation*}%
to a subspace on which $\tilde{Q}$ behaves well. We then define 
\begin{eqnarray*}
&&\hat{Q}\left( x_{0},P_{0},\underline{x}_{1},P_{1},\cdots ,\underline{x}%
_{k},P_{k},z_{1},\cdots ,z_{L},x_{1},\cdots ,x_{k}\right)  \\
&=&Q\left( x_{0},\hat{P}_{0},x_{1},\pi _{x_{1}}^{\bar{\bar{m}}%
\rightarrow \bar{m}}\hat{P}_{1},\cdots ,x_{k},\pi _{x_{k}}^{\overline{%
\overline{m}}\rightarrow \bar{m}}\hat{P}_{k},z_{1},\cdots ,z_{L}\right) 
\text{,}
\end{eqnarray*}%
where $\left( \hat{P}_{0},\hat{P}_{1},\cdots ,\hat{P}_{k}\right)
=\prod_{\left( x_{0},\underline{x}_{1},\cdots ,\underline{x}_{k}\right)
}\left( P_{0},P_{1},\cdots ,P_{k}\right) $. (Compare with (\ref{add5}).)

Our induction hypothesis -- Proposition \ref{SAQF1} in lower dimension --
applies to $\hat{Q},\hat{E}^{++},\hat{E}^{+},\hat{E},$ allowing us to argue
as we hoped to do for $\tilde{Q},\hat{E}^{++},\hat{E}^{+},\hat{E}$. Thus, we
can complete our induction on $\dim E^{+}$. This completes our preview of
the proof of Proposition \ref{SAQF1}.

Proposition \ref{SAQF2} is proven similarly. In some cases, we approximate $%
J_{x_{i}}^{\left( \bar{m}\right) }F$ in \eqref{PSQFx1} by $\pi _{x_{i}}^{%
\bar{\bar{m}}\rightarrow \bar{m}}J_{x_{0}}^{\left( \overline{%
\overline{m}}\right) }F$, rather than by $\pi _{x_{i}}^{\bar{\bar{m%
}}\rightarrow \bar{m}}J_{\underline{x}_{i}}^{\left( \bar{\bar{m}}%
\right) }F$.

Having done with previews, let us return to the proof of Proposition \ref%
{SAQF1}.

We will now estimate $Q(x_0, P_0, x_1, J_{x_1}^{(\bar{m})}F, \cdots, x_k,
J_{x_k}^{(\bar{m})}F,z_1,\cdots,z_L)$ under several different assumptions on 
$(x_0,x_1,\cdots, x_k,z_1, \cdots, z_L) \in E^{++}$.

\underline{Case 1:}

Suppose $\left( x_{0},x_{1},\cdots ,x_{k}\right) \in E^{+}\setminus 
\underline{E}^{+}$, and suppose also that $\underline{E}^{+}\left(
x_{0}\right) $ is nonempty. Let $F \in C^{\infty}_0(\mathbb{R}^n,\mathbb{R}%
^I)$.

Let $\left( \underline{x}_{1},\cdots ,\underline{x}_{k}\right) \in 
\underline{E}^{+}\left( x_{0}\right) $ be as close as possible to $\left(
x_{1},\cdots ,x_{k}\right) $. (Recall, $\underline{E}^{+}\left( x_{0}\right) 
$ is compact.)

Then, for $\bar{\bar{m}}>\bar{m}$ to be picked in a moment, Taylor's theorem
(see Proposition \ref{taylor'sthm}) gives%
\begin{eqnarray}
&&\left\vert \partial ^{\alpha }\left\{ \pi _{x_{j}}^{\bar{\bar{m}}%
\rightarrow \bar{m}}J_{\underline{x}_{j}}^{\left( \bar{\bar{m}}\right)
}F-J_{x_{j}}^{\left( \bar{m}\right) }F\right\} \left( x_{j}\right)
\right\vert  \notag \\
&\leq &C\left\Vert F\right\Vert _{C^{\bar{\bar{m}}}\left( \mathbb{R}^{n},%
\mathbb{R}^{I}\right) }\cdot \left[ \dist\left( \left( x_{1},\cdots
,x_{k}\right) ,\left( \underline{x}_{1},\cdots ,\underline{x}_{k}\right)
\right) \right] ^{\bar{\bar{m}}-\bar{m}}  \label{P11}
\end{eqnarray}%
for $\left\vert \alpha \right\vert \leq \bar{m}$, $j=1,\cdots k$, $F\in
C_{0}^{\infty }\left( \mathbb{R}^{n},\mathbb{R}^{I}\right) $.

Throughout this section, we write $c,C,C^{\prime },$ etc. to denote
constants determined by $\bar{\bar{m}}$, $\bar{m},$ $n,k,L,I,K,D,$ and an
upper bound for diameter of $E^{+}$. These symbols may denote different
constants in different occurrences.

Since 
\begin{equation*}
\left( P_{0},P_{1},\cdots ,P_{k}\right) \mapsto Q\left( x_{0},P_{0},\cdots
,x_{k},P_{k},z_{1},\cdots ,z_{L}\right)
\end{equation*}%
is a positive semidefinite quadratic form, we have the estimates%
\begin{eqnarray}
&&Q\left( x_{0},P_{0},x_{1},J_{x_{1}}^{\left( \bar{m}\right) }F,\cdots
,x_{k},J_{x_{k}}^{\left( \bar{m}\right) }F,z_{1},\cdots ,z_{L}\right)  \notag
\\
&\leq &2Q\left( x_{0},P_{0},x_{1},\pi _{x_{1}}^{\bar{\bar{m}}\rightarrow 
\bar{m}}J_{\underline{x}_{1}}^{\left( \bar{\bar{m}}\right) }F,\cdots
,x_{k},\pi _{x_{k}}^{\bar{\bar{m}}\rightarrow \bar{m}}J_{\underline{x}%
_{k}}^{\left( \bar{\bar{m}}\right) }F,z_{1},\cdots ,z_{L}\right)  \label{P12}
\\
&&+2Q\left( x_{0},0,x_{1},\left\{ \pi _{x_{1}}^{\bar{\bar{m}}\rightarrow 
\bar{m}}J_{\underline{x}_{1}}^{\left( \bar{\bar{m}}\right)
}F-J_{x_{1}}^{\left( \bar{m}\right) }F\right\} ,\cdots ,x_{k},\left\{ \pi
_{x_{k}}^{\bar{\bar{m}}\rightarrow \bar{m}}J_{\underline{x}_{k}}^{\left( 
\bar{\bar{m}}\right) }F-J_{x_{k}}^{\left( \bar{m}\right) }F\right\}
,z_{1},\cdots ,z_{L}\right)  \notag
\end{eqnarray}%
and similarly 
\begin{eqnarray}
&&Q\left( x_{0},P_{0},x_{1},\pi _{x_{1}}^{\bar{\bar{m}}\rightarrow \bar{m}%
}J_{\underline{x}_{1}}^{\left( \bar{\bar{m}}\right) }F,\cdots , x_k, \pi
_{x_{k}}^{\bar{\bar{m}}\rightarrow \bar{m}}J_{\underline{x}_{k}}^{\left( 
\bar{\bar{m}}\right) }F,z_{1},\cdots ,z_{L}\right)  \notag \\
&\leq &2Q\left( x_{0},P_{0},x_{1},J_{x_{1}}^{\left( \bar{{m}}\right)
}F,\cdots ,x_{k},J_{x_{k}}^{\left( \bar{{m}}\right) }F,z_{1},\cdots
,z_{L}\right)  \label{P13} \\
&&+2Q\left( x_{0},0,x_{1},\left\{ \pi _{x_{1}}^{\bar{\bar{m}}\rightarrow 
\bar{m}}J_{\underline{x}_{1}}^{\left( \bar{\bar{m}}\right)
}F-J_{x_{1}}^{\left( \bar{m}\right) }F\right\} ,\cdots ,x_{k},\left\{ \pi
_{x_{k}}^{\bar{\bar{m}}\rightarrow \bar{m}}J_{\underline{x}_{k}}^{\left( 
\bar{\bar{m}}\right) }F-J_{{x}_{k}}^{(\bar{m})}F\right\} ,z_{1},\cdots
,z_{L}\right) .  \notag
\end{eqnarray}%
From (\ref{P5}), (\ref{P11}), and hypothesis (\ref{PSQF8}), we see that 
\begin{eqnarray}
0 &\leq &Q\left( x_{0},0,x_{1},\left\{ \pi _{x_{1}}^{\bar{\bar{m}}%
\rightarrow \bar{m}}J_{\underline{x}_{1}}^{\left( \bar{\bar{m}}\right)
}F-J_{x_{1}}^{\left( \bar{m}\right) }F\right\} ,\cdots ,x_{k},\left\{ \pi
_{x_{k}}^{\bar{\bar{m}}\rightarrow \bar{m}}J_{\underline{x}_{k}}^{\left( 
\bar{\bar{m}}\right) }F-J_{x_{k}}^{(\bar{m})}F\right\} ,z_{1},\cdots
,z_{L}\right)  \notag \\
&\leq &A\left( x_{0},x_{1},\cdots ,x_{k}\right) \cdot C^{2}\left\Vert
F\right\Vert _{C^{\bar{\bar{m}}}\left( \mathbb{R}^{n},\mathbb{R}^{I}\right)
}^{2}\cdot \left[ \dist\left( \left( x_{1},\cdots ,x_{k}\right) ,\left( 
\underline{x}_{1},\cdots ,\underline{x}_{k}\right) \right) \right] ^{2\left( 
\bar{\bar{m}}-\bar{m}\right) }  \notag \\
&\leq &C^{2}\left\Vert F\right\Vert _{C^{\bar{\bar{m}}}\left( \mathbb{R}^{n},%
\mathbb{R}^{I}\right) }^{2}\cdot A_{1}\left( x_{0}\right) \cdot \left[ \dist%
\left( \left( x_{1},\cdots ,x_{k}\right) ,\left( \underline{x}_{1},\cdots ,%
\underline{x}_{k}\right) \right) \right] ^{2\left( \bar{\bar{m}}-\bar{m}%
\right) -K},  \label{P14}
\end{eqnarray}
where the last inequality follows from the fact that $\left( \underline{x}%
_{1},\cdots ,\underline{x}_{k}\right) $ is a closest point to $\left(
x_{1},\cdots ,x_{k}\right) $ in $\underline{E}^{+}\left( x_{0}\right) $.

We now pick $\bar{\bar{m}}$ so large that $2\left( \bar{\bar{m}}-\bar{m}%
\right) -K\geq 200$.

From (\ref{P12}), (\ref{P13}), (\ref{P14}), we then learn that 
\begin{eqnarray}
&&Q\left( x_{0},P_{0},x_{1},J_{x_{1}}^{\left( \bar{m}\right) }F,\cdots
,x_{k},J_{x_{k}}^{\left( \bar{m}\right) }F,z_{1},\cdots ,z_{L}\right)  \notag
\\
&\leq &2Q\left( x_{0},P_{0},x_{1},\pi _{x_{1}}^{\bar{\bar{m}}\rightarrow 
\bar{m}}J_{\underline{x}_{1}}^{\left( \bar{\bar{m}}\right) }F,\cdots
,x_{k},\pi _{x_{k}}^{\bar{\bar{m}}\rightarrow \bar{m}}J_{\underline{x}%
_{k}}^{\left( \bar{\bar{m}}\right) }F,z_{1},\cdots ,z_{L}\right)  \notag \\
&&+C^{\prime }\left\Vert F\right\Vert _{C^{\bar{\bar{m}}}\left( \mathbb{R}%
^{n},\mathbb{R}^{I}\right) }^{2}\cdot A_{1}\left( x_{0}\right) \cdot
\left\vert \left( x_{1},\cdots ,x_{k}\right) -\left( \underline{x}%
_{1},\cdots ,\underline{x}_{k}\right) \right\vert ^{2},  \label{P15}
\end{eqnarray}%
and 
\begin{eqnarray}
&&Q\left( x_{0},P_{0},x_{1},\pi _{x_{1}}^{\bar{\bar{m}}\rightarrow \bar{m}%
}J_{\underline{x}_{1}}^{\left( \bar{\bar{m}}\right) }F,\cdots ,x_{k},\pi
_{x_{k}}^{\bar{\bar{m}}\rightarrow \bar{m}}J_{\underline{x}_{k}}^{\left( 
\bar{\bar{m}}\right) }F,z_{1},\cdots ,z_{L}\right)  \notag \\
&\leq &2Q\left( x_{0},P_{0},x_{1},J_{x_{1}}^{\left( \bar{m}\right) }F,\cdots
,x_{k},J_{x_{k}}^{\left( \bar{m}\right) }F,z_{1},\cdots ,z_{L}\right)
\label{P16} \\
&&+C^{\prime }\left\Vert F\right\Vert _{C^{\bar{\bar{m}}}\left( \mathbb{R}%
^{n},\mathbb{R}^{I}\right) }^{2}\cdot A_{1}\left( x_{0}\right) \cdot
\left\vert \left( x_{1},\cdots ,x_{k}\right) -\left( \underline{x}%
_{1},\cdots ,\underline{x}_{k}\right) \right\vert ^{2}.  \notag
\end{eqnarray}

\begin{itemize}
\item[\refstepcounter{equation}\text{(\theequation)}\label{P17}] Estimates (%
\ref{P15}) and (\ref{P16}) hold in Case 1, i.e., when $(x_0,x_1,%
\cdots,x_k,z_1,\cdots,z_L) \in E^{++}$,  \\ $\left( x_{0},x_{1},\cdots
,x_{k}\right) \in E^{+}\setminus \underline{E}^{+}$, $\underline{E}%
^{+}\left( x_{0}\right) $ is nonempty, $\left( \underline{x}_{1},\cdots ,%
\underline{x}_{k}\right) \in \underline{E}^{+}\left( x_{0}\right) $ is as
close as possible to $\left( x_{1},\cdots ,x_{k}\right)$, and $F \in
C_0^{\infty}(\mathbb{R}^n,\mathbb{R}^I)$.
\end{itemize}

This completes the discussion of Case 1.

\underline{Case 2}: Now suppose $\left( x_{0},x_{1},\cdots ,x_{k}\right) \in
E^{+}$ and that $\underline{E}^{+}\left( x_{0}\right) $ is empty.

Then (\ref{P4}) and hypothesis (\ref{PSQF8}) show that 
\begin{eqnarray}
&&Q\left( x_{0},P_{0},x_{1},J_{x_{1}}^{\left( \bar{m}\right) }F,\cdots
,x_{k},J_{x_{k}}^{\left( \bar{m}\right) }F,z_{1},\cdots ,z_{L}\right)  \notag
\\
&\leq &A\left( x_{0},x_{1},\cdots ,x_{k}\right) \left[ \sum_{i=1}^{k}\sum_{%
\left\vert \alpha \right\vert \leq \bar{m}}\left\vert \partial ^{\alpha
}F\left( x_{i}\right) \right\vert ^{2}+\sum_{\left\vert \alpha \right\vert
\leq {m}}\left\vert \partial ^{\alpha }P_{0}\left( x_{0}\right) \right\vert
^{2}\right]  \notag \\
&\leq &CA_{1}\left( x_{0}\right) \left[ \sum_{\left\vert \alpha \right\vert
\leq {m}}\left\vert \partial ^{\alpha }P_{0}\left( x_{0}\right) \right\vert
^{2}+\left\Vert F\right\Vert _{C^{\bar{\bar{m}}}\left( \mathbb{R}^{n},%
\mathbb{R}^{I}\right) }^{2}\right] ,  \label{P18}
\end{eqnarray}%
whenever $F\in C_{0}^{\infty }\left( \mathbb{R}^{n},\mathbb{R}^{I}\right) $
and $\left( x_{0},x_{1},\cdots ,x_{k},z_{1},\cdots ,z_{L}\right) \in E^{++}$.

Moreover, because 
\begin{equation*}
\left( P_{0},P_{1},\cdots ,P_{k}\right) \mapsto Q\left( x_{0},P_{0},\cdots
,x_{k},P_{k},z_{1},\cdots ,z_{L}\right)
\end{equation*}%
is a nonnegative quadratic form, we have 
\begin{eqnarray}
&&Q\left( x_{0},P_{0},x_{1},J_{x_{1}}^{\left( \bar{m}\right) }F,\cdots
,x_{k},J_{x_{k}}^{\left( \bar{m}\right) }F,z_{1},\cdots ,z_{L}\right)  \notag
\\
&\leq &2Q\left( x_{0},P_{0},x_{1},\pi _{x_{1}}^{\bar{\bar{m}}\rightarrow 
\bar{m}}J_{x_{0}}^{\left( \bar{\bar{m}}\right) }F,\cdots ,x_{k},\pi
_{x_{k}}^{\bar{\bar{m}}\rightarrow \bar{m}}J_{x_{0}}^{\left( \bar{\bar{m}}%
\right) }F,z_{1},\cdots ,z_{L}\right)  \label{P19} \\
&&+2Q\left( x_{0},0,x_{1},\left\{ \pi _{x_{1}}^{\bar{\bar{m}}\rightarrow 
\bar{m}}J_{x_{0}}^{\left( \bar{\bar{m}}\right) }F-J_{x_{1}}^{\left( \bar{m}%
\right) }F\right\} ,\cdots ,x_{k},\left\{ \pi _{x_{k}}^{\bar{\bar{m}}%
\rightarrow \bar{m}}J_{x_{0}}^{\left( \bar{\bar{m}}\right)
}F-J_{x_{k}}^{\left( \bar{m}\right) }F\right\} ,z_{1},\cdots ,z_{L}\right) 
\notag
\end{eqnarray}%
and 
\begin{eqnarray}
&&Q\left( x_{0},P_{0},x_{1},\pi _{x_{1}}^{\bar{\bar{m}}\rightarrow \bar{m}%
}J_{x_{0}}^{\left( \bar{\bar{m}}\right) }F,\cdots ,x_{k},\pi _{x_{k}}^{\bar{%
\bar{m}}\rightarrow \bar{m}}J_{x_{0}}^{\left( \bar{\bar{m}}\right)
}F,z_{1},\cdots ,z_{L}\right)  \notag \\
&\leq &2Q\left( x_{0},P_{0},x_{1},J_{x_{1}}^{\left( \bar{m}\right) }F,\cdots
,x_{k},J_{x_{k}}^{\left( \bar{m}\right) }F,z_{1},\cdots ,z_{L}\right)
\label{P20} \\
&&+2Q\left( x_{0},0,x_{1},\left\{ \pi _{x_{1}}^{\bar{\bar{m}}\rightarrow 
\bar{m}}J_{x_{0}}^{\left( \bar{\bar{m}}\right) }F-J_{x_{1}}^{\left( \bar{m}%
\right) }F\right\} ,\cdots ,x_{k},\left\{ \pi _{x_{k}}^{\bar{\bar{m}}%
\rightarrow \bar{m}}J_{x_{0}}^{\left( \bar{\bar{m}}\right)
}F-J_{x_{k}}^{\left( \bar{m}\right) }F\right\} ,z_{1},\cdots ,z_{L}\right) .
\notag
\end{eqnarray}%
By (\ref{P4}), hypothesis (\ref{PSQF8}), and Taylor's theorem (see
Proposition \ref{taylor'sthm}), we have 
\begin{eqnarray}
&&Q\left( x_{0},0,x_{1},\left\{ \pi _{x_{1}}^{\bar{\bar{m}}\rightarrow \bar{m%
}}J_{x_{0}}^{\left( \bar{\bar{m}}\right) }F-J_{x_{1}}^{\left( \bar{m}\right)
}F\right\} ,\cdots ,x_{k},\left\{ \pi _{x_{k}}^{\bar{\bar{m}}\rightarrow 
\bar{m}}J_{x_{0}}^{\left( \bar{\bar{m}}\right) }F-J_{x_{k}}^{\left( \bar{m}%
\right) }F\right\} ,z_{1},\cdots ,z_{L}\right)  \notag \\
&\leq &A_{1}\left( x_{0}\right) \sum_{i=1}^{k}\sum_{\left\vert \alpha
\right\vert \leq \bar{m}}\left\vert \partial ^{\alpha }\left\{ \pi _{x_{i}}^{%
\bar{\bar{m}}\rightarrow \bar{m}}J_{x_{0}}^{\left( \bar{\bar{m}}\right)
}F-J_{x_{i}}^{\left( \bar{m}\right) }F\right\} \left( x_{i}\right)
\right\vert ^{2}  \notag \\
&\leq &CA_{1}\left( x_{0}\right) \left\Vert F\right\Vert _{C^{\bar{\bar{m}}%
}\left( \mathbb{R}^{n},\mathbb{R}^{I}\right) }^{2}\sum_{i=1}^{k}\left\vert
x_{i}-x_{0}\right\vert ^{2}\text{,}  \label{P21}
\end{eqnarray}%
since we picked $\bar{\bar{m}}$ so that (in particular) $\bar{\bar{m}}-\bar{m%
}\geq 100$.

Putting (\ref{P21}) into (\ref{P19}), (\ref{P20}), we see that 
\begin{eqnarray}
&&Q\left( x_{0},P_{0},x_{1},J_{x_{1}}^{\left( \bar{m}\right) }F,\cdots
,x_{k},J_{x_{k}}^{\left( \bar{m}\right) }F,z_{1},\cdots ,z_{L}\right)  \notag
\\
&\leq &2Q\left( x_{0},P_{0},x_{1},\pi _{x_{1}}^{\bar{\bar{m}}\rightarrow 
\bar{m}}J_{x_{0}}^{\left( \bar{\bar{m}}\right) }F,\cdots ,x_{k},\pi
_{x_{k}}^{\bar{\bar{m}}\rightarrow \bar{m}}J_{x_{0}}^{\left( \bar{\bar{m}}%
\right) }F,z_{1},\cdots ,z_{L}\right)  \label{P22} \\
&&+2CA_{1}\left( x_{0}\right) \left\Vert F\right\Vert _{C^{\bar{\bar{m}}%
}\left( \mathbb{R}^{n},\mathbb{R}^{I}\right) }^{2}\max_{i=1,\cdots
,k}\left\vert x_{i}-x_{0}\right\vert ^{2}  \notag
\end{eqnarray}%
and 
\begin{eqnarray}
&&Q\left( x_{0},P_{0},x_{1},\pi _{x_{1}}^{\bar{\bar{m}}\rightarrow \bar{m}%
}J_{x_{0}}^{\left( \bar{\bar{m}}\right) }F,\cdots ,x_{k},\pi _{x_{k}}^{\bar{%
\bar{m}}\rightarrow \bar{m}}J_{x_{0}}^{\left( \bar{\bar{m}}\right)
}F,z_{1},\cdots ,z_{L}\right)  \notag \\
&\leq &2Q\left( x_{0},P_{0},x_{1},J_{x_{1}}^{\left( \bar{m}\right) }F,\cdots
,x_{k},J_{x_{k}}^{\left( \bar{m}\right) }F,z_{1},\cdots ,z_{L}\right)
\label{P23} \\
&&+2CA_{1}\left( x_{0}\right) \left\Vert F\right\Vert _{C^{\bar{\bar{m}}%
}\left( \mathbb{R}^{n},\mathbb{R}^{I}\right) }^{2}\max_{i=1,\cdots
,k}\left\vert x_{i}-x_{0}\right\vert ^{2}.  \notag
\end{eqnarray}

Estimates (\ref{P18}), (\ref{P22}), (\ref{P23}) hold in Case 2, i.e., when

\begin{itemize}
\item[\refstepcounter{equation}\text{(\theequation)}\label{P24}] $\left(
x_{0},x_{1},\cdots ,x_{k}\right) \in E^{+}$, $\underline{E}^{+}\left(
x_{0}\right) =\emptyset $, $\left( x_{0},x_{1},\cdots ,x_{k},z_{1},\cdots
,z_{L}\right) \in E^{++}$, and $F\in C_{0}^{\infty }\left( \mathbb{R}^{n},%
\mathbb{R}^{I}\right) $.
\end{itemize}

This completes our discussion of Case 2.

One of the two cases Case 1 and Case 2 above must occur whenever $\left(
x_{0},\cdots ,x_{k},z_{1},\cdots ,z_{L}\right) \in E^{++}$ and $\left(
x_{0},\cdots ,x_{k}\right) \in E^{+}\setminus \underline{E}^{+}$.

However, it is also useful to study the following subcase of Case 1.

\underline{Case 1'}: Suppose $(x_0,x_1,\cdots,x_k,z_1,\cdots,z_L) \in E^{++}$%
, $\left( x_{0},x_{1},\cdots ,x_{k}\right) \in E^{+}\setminus \underline{E}%
^{+}$, $\underline{E}^{+}\left( x_{0}\right) \not=\emptyset $, and $\left(
x_{0},\cdots ,x_{0}\right) \not\in \underline{E}^{+}\left( x_{0}\right) $.
Let $F \in C_0^{\infty}(\mathbb{R}^n,\mathbb{R}^I)$.

Because $\underline{E}^{+}$, $\underline{E}^{+}\left( x_{0}\right) $ are
compact, we have 
\begin{equation*}
\dist\left( \left( x_{0},\cdots ,x_{0}\right) ,\underline{E}^{+}\left(
x_{0}\right) \right) >0.
\end{equation*}%
Suppose 
\begin{equation*}
\left\vert \left( x_{1},\cdots ,x_{k}\right) -\left( x_{0},\cdots
,x_{0}\right) \right\vert <\frac{1}{2}\dist\left( \left( x_{0},\cdots
,x_{0}\right) ,\underline{E}^{+}\left( x_{0}\right) \right) ,
\end{equation*}%
hence, by (\ref{P5}), 
\begin{equation*}
\left\vert A\left( x_{0},x_{1},\cdots ,x_{k}\right) \right\vert \leq
CA_{1}\left( x_{0}\right) \left[ \dist\left( \left( x_{0},\cdots
,x_{0}\right) ,\underline{E}^{+}\left( x_{0}\right) \right) \right] ^{-K}.
\end{equation*}%
Hypothesis (\ref{PSQF8}) therefore yields 
\begin{eqnarray*}
&&Q\left( x_{0},P_{0},x_{1},P_{1},\cdots ,x_{k},P_{k},z_{1},\cdots
,z_{L}\right) \\
&\leq &CA_{1}\left( x_{0}\right) \left[ \dist\left( \left( x_{0},\cdots
,x_{0}\right) ,\underline{E}^{+}\left( x_{0}\right) \right) \right] ^{-K}%
\left[ \sum_{|\alpha |\leq m}|\partial ^{\alpha
}P_{0}(x_{0})|^{2}+\sum_{i=1}^{k}\sum_{\left\vert \alpha \right\vert \leq 
\bar{m}}\left\vert \partial ^{\alpha }P_{i}\left( x_{i}\right) \right\vert
^{2}\right]
\end{eqnarray*}%
for any $P_{0},P_{1},\cdots ,P_{k}$; compare with (\ref{P18}). Hence, we may
proceed as in the proof of (\ref{P22}), (\ref{P23}), to establish the
following result.

\begin{itemize}
\item[\refstepcounter{equation}\text{(\theequation)}\label{P24Beta}] In Case
1', if $\left\vert \left( x_{1},\cdots ,x_{k}\right) -\left( x_{0},\cdots
,x_{0}\right) \right\vert <\frac{1}{2}\dist\left( \left( x_{0},\cdots
,x_{0}\right) ,\underline{E}^{+}\left( x_{0}\right) \right) $, then we have 
\begin{eqnarray*}
&&Q\left( x_{0},P_{0},x_{1},J_{x_{1}}^{\left( \bar{m}\right) }F,\cdots
,x_{k},J_{x_{k}}^{\left( \bar{m}\right) }F,z_{1},\cdots ,z_{L}\right) \\
&\leq &2Q\left( x_{0},P_{0},x_{1},\pi _{x_{1}}^{\bar{\bar{m}}\rightarrow 
\bar{m}}J_{x_{0}}^{\left( \bar{\bar{m}}\right) }F,\cdots ,x_{k},\pi
_{x_{k}}^{\bar{\bar{m}}\rightarrow \bar{m}}J_{x_{0}}^{\left( \bar{\bar{m}}%
\right) }F,z_{1},\cdots ,z_{L}\right) \\
&&+\frac{C\left\Vert F\right\Vert _{C^{\bar{\bar{m}}}\left( \mathbb{R}^{n},%
\mathbb{R}^{I}\right) }^{2}A_{1}\left( x_{0}\right) }{\left[ \dist\left(
\left( x_{0},\cdots ,x_{0}\right) ,\underline{E}^{+}\left( x_{0}\right)
\right) \right] ^{K}}\cdot \max_{i=1,\cdots ,k}\left\vert
x_{i}-x_{0}\right\vert ^{2}
\end{eqnarray*}%
and 
\begin{eqnarray*}
&&Q\left( x_{0},P_{0},x_{1},\pi _{x_{1}}^{\bar{\bar{m}}\rightarrow \bar{m}%
}J_{x_{0}}^{\left( \bar{\bar{m}}\right) }F,\cdots ,x_{k},\pi _{x_{k}}^{\bar{%
\bar{m}}\rightarrow \bar{m}}J_{x_{0}}^{\left( \bar{\bar{m}}\right)
}F,z_{1},\cdots ,z_{L}\right) \\
&\leq &2Q\left( x_{0},P_{0},x_{1},J_{x_{1}}^{\left( \bar{m}\right) }F,\cdots
,x_{k},J_{x_{k}}^{\left( \bar{m}\right) }F,z_{1},\cdots ,z_{L}\right) \\
&&+\frac{C\left\Vert F\right\Vert _{C^{\bar{\bar{m}}}\left( \mathbb{R}^{n},%
\mathbb{R}^{I}\right) }^{2}A_{1}\left( x_{0}\right) }{\left[ \dist\left(
\left( x_{0},\cdots ,x_{0}\right) ,\underline{E}^{+}\left( x_{0}\right)
\right) \right] ^{K}}\cdot \max_{i=1,\cdots ,k}\left\vert
x_{i}-x_{0}\right\vert ^{2}\text{.}
\end{eqnarray*}
\end{itemize}

This completes our analysis of Case 1'.

For fixed $\left( x_{0},\underline{x}_{1},\cdots ,\underline{x}_{k}\right)
\in \underline{E}^{+}$ and 
\begin{equation*}
\left( P_{0},P_{1},\cdots ,P_{k}\right) \in \mathcal{P}^{\left( m\right)
}\left( \mathbb{R}^{n},\mathbb{R}^{D}\right) \oplus \underset{k\text{ copies}%
}{\underbrace{\mathcal{P}^{\left( \bar{\bar{m}}\right) }\left( \mathbb{R}%
^{n},\mathbb{R}^{I}\right) \oplus \cdots \oplus \mathcal{P}^{\left( \bar{%
\bar{m}}\right) }\left( \mathbb{R}^{n},\mathbb{R}^{I}\right) }},
\end{equation*}%
we ask whether 
\begin{equation}
\sup \left\{ 
\begin{array}{c}
Q\left( x_{0},P_{0},x_{1},\pi _{x_{1}}^{\bar{\bar{m}}\rightarrow \bar{m}%
}P_{1},\cdots ,x_{k},\pi _{x_{k}}^{\bar{\bar{m}}\rightarrow \bar{m}%
}P_{k},z_{1},\cdots ,z_{L}\right) : \\ 
\left( x_{1},\cdots ,x_{k},z_{1},\cdots ,z_{L}\right) \in E^{++}(x_{0}), \\ 
\left( x_{1},\cdots ,x_{k}\right) \in E^{+}(x_{0})\setminus \underline{E}%
^{+}(x_{0}), \\ 
\dist\left( \left( x_{1},\cdots ,x_{k}\right) ,\underline{E}^{+}\left(
x_{0}\right) \right) =\left\vert \left( x_{1},\cdots ,x_{k}\right) -\left( 
\underline{x}_{1},\cdots ,\underline{x}_{k}\right) \right\vert%
\end{array}%
\right\} <\infty .  \label{P25}
\end{equation}

Note: In (\ref{P25}), $\left( x_{0},\underline{x}_{1},\cdots ,\underline{x}%
_{k}\right) $ are held fixed, i.e., the sup is over $z_{1},\cdots
,z_{L},x_{1},\cdots ,x_{k}$ satisfying the constraints. The sup is taken to
be $0$ if the the set of points satisfying the constraints is empty.

Observe that the set of all $(x_{0},\underline{x}_{1},\cdots ,\underline{x}%
_{k},P_{0},P_{1},\cdots ,P_{k})$ satisfying (\ref{P25}) is semialgebraic,
since $Q$, $E^{++}$, $E^{+}$, $\underline{E}^{+}$ are semialgebraic.
Moreover, since 
\begin{equation*}
\left( P_{0},P_{1},\cdots ,P_{k}\right) \mapsto Q\left(
x_{0},P_{0},x_{1},\pi _{x_{1}}^{\bar{\bar{m}}\rightarrow \bar{m}%
}P_{1},\cdots ,x_{k},\pi _{x_{k}}^{\bar{\bar{m}}\rightarrow \bar{m}%
}P_{k},z_{1},\cdots ,z_{L}\right)
\end{equation*}%
is a nonnegative quadratic form for fixed $x_{0},x_{1},\cdots ,x_{k}$, $%
z_{1},\cdots ,z_{L}$, it follows that the set of all $\left(
P_{0},P_{1},\cdots ,P_{k}\right) $ satisfying (\ref{P25}) is a vector
subspace of 
\begin{equation*}
\mathcal{P}^{\left( m\right) }\left( \mathbb{R}^{n},\mathbb{R}^{D}\right)
\oplus \underset{k\text{ copies}}{\underbrace{\mathcal{P}^{\left( \bar{\bar{m%
}}\right) }\left( \mathbb{R}^{n},\mathbb{R}^{I}\right) \oplus \cdots \oplus 
\mathcal{P}^{\left( \bar{\bar{m}}\right) }\left( \mathbb{R}^{n},\mathbb{R}%
^{I}\right) }},
\end{equation*}%
for fixed $\left( x_{0},\underline{x}_{1},\cdots ,\underline{x}_{k}\right)
\in \underline{E}^{+}$. We denote this vector subspace by 
\begin{equation*}
\hat{H}^{\text{bdd}}\left( x_{0},\underline{x}_{1},\cdots ,\underline{x}%
_{k}\right) \text{.}
\end{equation*}%
Thus,

\begin{itemize}
\item[\refstepcounter{equation}\text{(\theequation)}\label{P26}] $\hat{H}^{%
\text{bdd}}\left( x_{0},\underline{x}_{1},\cdots ,\underline{x}_{k}\right)
\subset \mathcal{P}^{\left( m\right) }\left( \mathbb{R}^{n},\mathbb{R}%
^{D}\right) \oplus \underset{k\text{ copies}}{\underbrace{\mathcal{P}%
^{\left( \bar{\bar{m}}\right) }\left( \mathbb{R}^{n},\mathbb{R}^{I}\right)
\oplus \cdots \oplus \mathcal{P}^{\left( \bar{\bar{m}}\right) }\left( 
\mathbb{R}^{n},\mathbb{R}^{I}\right) }}$ is a vector subspace depending
semialgebraically on $\left( x_{0},\underline{x}_{1},\cdots ,\underline{x}%
_{k}\right) \in \underline{E}^{+}$,
\end{itemize}

and

\begin{itemize}
\item[\refstepcounter{equation}\text{(\theequation)}\label{P27}] for every
given $\left( x_{0},\underline{x}_{1},\cdots ,\underline{x}_{k}\right) \in 
\underline{E}^{+}$ and 
\begin{equation*}
\left( P_{0},P_{1},\cdots ,P_{k}\right) \in \mathcal{P}^{\left( m\right)
}\left( \mathbb{R}^{n},\mathbb{R}^{D}\right) \oplus \underset{k\text{ copies}%
}{\underbrace{\mathcal{P}^{\left( \bar{\bar{m}}\right) }\left( \mathbb{R}%
^{n},\mathbb{R}^{I}\right) \oplus \cdots \oplus \mathcal{P}^{\left( \bar{%
\bar{m}}\right) }\left( \mathbb{R}^{n},\mathbb{R}^{I}\right) }},
\end{equation*}%
condition (\ref{P25}) holds if and only if 
\begin{equation*}
\left( P_{0},P_{1},\cdots ,P_{k}\right) \in \hat{H}^{\text{bdd}}\left( x_{0},%
\underline{x}_{1},\cdots ,\underline{x}_{k}\right) \text{.}
\end{equation*}
\end{itemize}

Given $\left( x_{0},\underline{x}_{1},\cdots ,\underline{x}_{k}\right) \in 
\underline{E}^{+}$ and $\left( P_{0},P_{1},\cdots ,P_{k}\right) \in \hat{H}^{%
\text{bdd}}\left( x_{0},\underline{x}_{1},\cdots ,\underline{x}_{k}\right) $%
, we define 
\begin{equation*}
\norm(P_{0},\cdots ,P_{k};x_{0},\underline{x}_{1},\cdots ,\underline{x}_{k})
\end{equation*}%
to be the square root of the sup in \eqref{P25}.

Note that 
\begin{eqnarray*}
&&\norm\left( P_{0}+P_{0}^{\prime },P_{1}+P_{1}^{\prime },\cdots
,P_{k}+P_{k}^{\prime };x_{0},\underline{x}_{1},\cdots ,\underline{x}%
_{k}\right) \\
&\leq &\norm\left( P_{0},P_{1},\cdots ,P_{k};x_{0},\underline{x}_{1},\cdots ,%
\underline{x}_{k}\right) \\
&&+\norm\left( P_{0}^{\prime },P_{1}^{\prime },\cdots ,P_{k}^{\prime };x_{0},%
\underline{x}_{1},\cdots ,\underline{x}_{k}\right)
\end{eqnarray*}%
for any $\left( P_{0},P_{1},\cdots ,P_{k}\right) $, $\left( P_{0}^{\prime
},\cdots ,P_{k}^{\prime }\right) \in \hat{H}^{\text{bdd}}\left( x_{0},%
\underline{x}_{1},\cdots ,\underline{x}_{k}\right) $ and any $\left( x_{0},%
\underline{x}_{1},\cdots ,\underline{x}_{k}\right) \in \underline{E}^{+}$,
because 
\begin{equation*}
Q\left( x_{0},P_{0},x_{1},P_{1},\cdots ,x_{k},P_{k},z_1,\cdots, z_L\right)
\end{equation*}
is a positive semidefinite quadratic form in $\left( P_{0},P_{1},\cdots
,P_{k}\right) $ for fixed $\left( x_{0},x_{1},\cdots ,x_{k},z_1, \cdots, z_L
\right) $.

Also, for $\lambda \in \mathbb{R}$ we have 
\begin{eqnarray*}
&&\norm\left( \lambda P_{0},\cdots ,\lambda P_{k};x_{0},\underline{x}%
_{1},\cdots ,\underline{x}_{k}\right) \\
&=&\left\vert \lambda \right\vert \norm\left( P_{0},\cdots ,P_{k}; x_0, 
\underline{x}_1,\cdots, \underline{x}_k\right) \text{.}
\end{eqnarray*}

Thus, for fixed $(x_0, \underline{x}_1,\cdots, \underline{x}_k) \in 
\underline{E}^{+}$, the function $(P_0, P_1, \cdots, P_k) \mapsto \norm(P_0,
P_1, \cdots, P_k; x_0, \underline{x}_1,\cdots, \underline{x}_k)$ is a
seminorm on the finite-dimensional vector space $\hat{H}^{\text{bdd}}(x_0,%
\underline{x}_1,\cdots, \underline{x}_k)$.

It follows that, for each $\left( x_{0},\underline{x}_{1},\cdots ,\underline{%
x}_{k}\right) \in \underline{E}^{+}$, we have 
\begin{eqnarray}
&&\norm\left( P_{0},P_{1},\cdots ,P_{k};x_{0},\underline{x}_{1},\cdots ,%
\underline{x}_{k}\right)  \notag \\
&\leq &M\left( x_{0},\underline{x}_{1},\cdots ,\underline{x}_{k}\right)
\cdot \left(\sum_{|\alpha| \leq m}\vert \partial^{\alpha} P_0(x_0)\vert^2+
\sum_{i=1}^{k}\sum_{\left\vert \alpha \right\vert \leq \bar{\bar{m}}%
}\left\vert \partial ^{\alpha }P_{i}\left( \underline{x}_{i}\right)
\right\vert ^{2}\right) ^{1/2}  \label{P28}
\end{eqnarray}%
for all $\left( P_{0},P_{1},\cdots ,P_{k}\right) \in \hat{H}^{\text{bdd}%
}\left( x_{0},\underline{x}_{1},\cdots ,\underline{x}_{k}\right) $; here, $%
M\left( x_{0},\underline{x}_{1},\cdots ,\underline{x}_{k}\right) $ is some
finite constant.

Since $\hat{H}^{\text{bdd}}\left( x_{0},\underline{x}_{1},\cdots ,\underline{%
x}_{k}\right) $ depends semialgebraically on $\left( x_{0},\underline{x}%
_{1},\cdots ,\underline{x}_{k}\right) $, and since $\norm\left( P_{0},\cdots
,P_{k};x_{0},\underline{x}_{1},\cdots ,\underline{x}_{k}\right) $ is
semialgebraic in $\left( P_{0},\cdots ,P_{k};x_{0},\underline{x}_{1},\cdots ,%
\underline{x}_{k}\right) $, it follows that the least possible nonnegative $%
M\left( x_{0},\underline{x}_{1},\cdots ,\underline{x}_{k}\right) $ in (\ref%
{P28}) is a semialgebraic function of $\left( x_{0},\underline{x}_{1},\cdots
,\underline{x}_{k}\right) \in \underline{E}^{+}$.

From now on, we define $M\left( x_{0},\underline{x}_{1},\cdots ,\underline{x}%
_{k}\right) $ to be the least possible nonnegative number for which (\ref%
{P28}) holds for all $\left( P_{0},P_{1},\cdots ,P_{k}\right) \in \hat{H}^{%
\text{bdd}}\left( x_{0},\underline{x}_{1},\cdots ,\underline{x}_{k}\right) $.

Thus, $M\left( x_{0},\underline{x}_{1},\cdots ,\underline{x}_{k}\right) $ is
a semialgebraic function; and from (\ref{P28}), we obtain the estimate 
\begin{eqnarray}
&&Q\left( x_{0},P_{0},x_{1},\pi _{x_{1}}^{\bar{\bar{m}}\rightarrow \bar{m}%
}P_{1},\cdots ,x_{k},\pi _{x_{k}}^{\bar{\bar{m}}\rightarrow \bar{m}%
}P_{k};z_{1},\cdots ,z_{L}\right)  \notag \\
&\leq &M^{2}\left( x_{0},\underline{x}_{1},\cdots ,\underline{x}_{k}\right)
\cdot \left( \sum_{|\alpha |\leq m}|\partial ^{\alpha
}P_{0}(x_{0})|^{2}+\sum_{i=1}^{k}\sum_{\left\vert \alpha \right\vert \leq 
\bar{\bar{m}}}\left\vert \partial ^{\alpha }P_{i}\left( \underline{x}%
_{i}\right) \right\vert ^{2}\right) ,  \label{P29}
\end{eqnarray}%
whenever 
\begin{subequations}
\begin{align}
& (x_{0},\underline{x}_{1},\cdots ,\underline{x}_{k})\in \underline{E}^{+},
\label{P30a} \\
& (x_{0},x_{1},\cdots ,x_{k},z_{1},\cdots ,z_{L})\in E^{++},  \label{P30b} \\
& \dist\left( \left( x_{1},\cdots ,x_{k}\right) ,\underline{E}^{+}\left(
x_{0}\right) \right) =\left\vert \left( x_{1},\cdots ,x_{k}\right) -\left( 
\underline{x}_{1},\cdots ,\underline{x}_{k}\right) \right\vert >0,
\label{P30c} \\
& (P_{0},P_{1},\cdots ,P_{k})\in \hat{H}^{\text{bdd}}\left( x_{0},\underline{%
x}_{1},\cdots ,\underline{x}_{k}\right) \text{.}  \label{P30d}
\end{align}%
Now, for $\left( x_{0},\underline{x}_{1},\cdots ,\underline{x}_{k}\right)
\in \underline{E}^{+}$, we let 
\end{subequations}
\begin{equation*}
\Pi _{\left( x_{0},\underline{x}_{1},\cdots ,\underline{x}_{k}\right) }
\end{equation*}%
denote the orthogonal projection from 
\begin{equation*}
\mathcal{P}^{\left( m\right) }\left( \mathbb{R}^{n},\mathbb{R}^{D}\right)
\oplus \underset{k\text{ copies}}{\underbrace{\mathcal{P}^{\left( \bar{\bar{m%
}}\right) }\left( \mathbb{R}^{n},\mathbb{R}^{I}\right) \oplus \cdots \oplus 
\mathcal{P}^{\left( \bar{\bar{m}}\right) }\left( \mathbb{R}^{n},\mathbb{R}%
^{I}\right) }}
\end{equation*}%
onto 
\begin{equation*}
\hat{H}^{\text{bdd}}\left( x_{0},\underline{x}_{1},\cdots ,\underline{x}%
_{k}\right) ,
\end{equation*}%
with respect to the quadratic form $\sum_{|\alpha |\leq m}|\partial ^{\alpha
}P_{0}(x_{0})|^{2}+\sum_{i=1}^{k}\sum_{\left\vert \alpha \right\vert \leq 
\bar{\bar{m}}}\left\vert \partial ^{\alpha }P_{i}\left( \underline{x}%
_{i}\right) \right\vert ^{2}$.

If 
\begin{equation*}
\left( \hat{P}_{0},\hat{P}_{1},\cdots ,\hat{P}_{k}\right) =\Pi _{\left(
x_{0},\underline{x}_{1},\cdots ,\underline{x}_{k}\right) }\left(
P_{0},P_{1},\cdots ,P_{k}\right) ,
\end{equation*}%
then the following hold. 
\begin{eqnarray}
\sum_{|\alpha| \leq m} \vert \partial^{\alpha} \hat{P}_0 (x_0)
\vert^2+\sum_{i=1}^{k}\sum_{\left\vert \alpha \right\vert \leq \bar{\bar{m}}%
}\left\vert \partial ^{\alpha }\hat{P}_{i}\left( \underline{x}_{i}\right)
\right\vert^2 \leq \sum_{|\alpha| \leq m} \vert \partial^{\alpha} P_0 (x_0)
\vert^2+\sum_{i=1}^{k}\sum_{\left\vert \alpha \right\vert \leq \bar{\bar{m}}%
}\left\vert \partial ^{\alpha }P_{i}\left( \underline{x}_{i}\right)
\right\vert^2,  \label{P31}
\end{eqnarray}
\begin{eqnarray}
\left( \hat{P}_{0},\hat{P}_{1},\cdots ,\hat{P}_{k}\right) &\in &\hat{H}^{%
\text{bdd}}\left( x_{0},\underline{x}_{1},\cdots ,\underline{x}_{k}\right) ,
\label{P32}
\end{eqnarray}%
\begin{equation}
\left( \hat{P}_{0},\hat{P}_{1},\cdots ,\hat{P}_{k}\right) =\left(
P_{0},P_{1},\cdots ,P_{k}\right) \text{ if }\left( P_{0},P_{1},\cdots
,P_{k}\right) \in \hat{H}^{\text{bdd}}\left( x_{0},\underline{x}_{1},\cdots ,%
\underline{x}_{k}\right).  \label{P33}
\end{equation}

Also, $\Pi _{\left( x_{0},\underline{x}_{1},\cdots ,\underline{x}_{k}\right)
}\left( P_{0},P_{1},\cdots ,P_{k}\right) $ is given by a semialgebraic map%
\begin{equation*}
\left( x_{0},\underline{x}_{1},\cdots ,\underline{x}_{k},P_{0},P_{1},\cdots
,P_{k}\right) \mapsto \left( \hat{P}_{0},\hat{P}_{1},\cdots ,\hat{P}%
_{k}\right) \text{.}
\end{equation*}

We prepare to invoke the induction hypothesis (Propositions \ref{SAQF1} and %
\ref{SAQF2} hold for $\dim E^{+}<\Delta $).

Let 
\begin{equation}
\hat{E}^{++}=\left\{ 
\begin{array}{c}
\left( x_{0},\underline{x}_{1},\cdots ,\underline{x}_{k},z_{1},\cdots
,z_{L},x_{1},\cdots ,x_{k}\right) \in \mathbb{R}^{n}\times \cdots \times 
\mathbb{R}^{n}: \\ 
\left( x_{0},x_{1},\cdots ,x_{k},z_{1},\cdots ,z_{L}\right) \in E^{++}, \\ 
\left( x_{0},\underline{x}_{1},\cdots ,\underline{x}_{k}\right) \in 
\underline{E}^{+}, \\ 
\dist\left( \left( x_{1},\cdots ,x_{k}\right) ,\underline{E}^{+}\left(
x_{0}\right) \right) =\left\vert \left( x_{1},\cdots ,x_{k}\right) -\left( 
\underline{x}_{1},\cdots ,\underline{x}_{k}\right) \right\vert \not=0%
\end{array}%
\right\} ,  \label{P34}
\end{equation}%
\begin{equation}
\hat{E}^{+}=\underline{E}^{+},  \label{P35}
\end{equation}%
\begin{equation}
\hat{E}=\left\{ x_{0}:\left( x_{0},\underline{x}_{1},\cdots ,\underline{x}%
_{k}\right) \in \underline{E}^{+}\text{ for some }\underline{x}_{1},\cdots ,%
\underline{x}_{k}\right\} \text{.}  \label{P36}
\end{equation}

\begin{itemize}
\item[\refstepcounter{equation}\text{(\theequation)}\label{P37}] For $\left(
x_{0},\underline{x}_{1},\cdots ,\underline{x}_{k},z_{1},\cdots
,z_{L},x_{1},\cdots ,x_{k}\right) \in \hat{E}^{++}$ and 
\begin{equation*}
\left( P_{0},P_{1},\cdots ,P_{k}\right) \in \mathcal{P}^{\left( m\right)
}\left( \mathbb{R}^{n},\mathbb{R}^{D}\right) \oplus \underset{k\text{ copies}%
}{\underbrace{\mathcal{P}^{\left( \bar{\bar{m}}\right) }\left( \mathbb{R}%
^{n},\mathbb{R}^{I}\right) \oplus \cdots \oplus \mathcal{P}^{\left( \bar{%
\bar{m}}\right) }\left( \mathbb{R}^{n},\mathbb{R}^{I}\right) }},
\end{equation*}%
define 
\begin{equation*}
\left( \hat{P}_{0},\hat{P}_{1},\cdots ,\hat{P}_{k}\right) =\Pi _{\left(
x_{0},\underline{x}_{1},\cdots ,\underline{x}_{k}\right) }\left(
P_{0},P_{1},\cdots ,P_{k}\right)
\end{equation*}%
and then set 
\begin{eqnarray*}
&&\hat{Q}\left( x_{0},P_{0},\underline{x}_{1},P_{1},\cdots ,\underline{x}%
_{k},P_{k},z_{1},\cdots ,z_{L},x_{1},\cdots ,x_{k}\right) \\
&=&Q\left( x_{0},\hat{P}_{0},x_{1},\pi _{x_{1}}^{\bar{\bar{m}}\rightarrow 
\bar{m}}\hat{P}_{1},\cdots ,x_{k},\pi _{x_{k}}^{\bar{\bar{m}}\rightarrow 
\bar{m}}\hat{P}_{k},z_{1},\cdots ,z_{L}\right) .
\end{eqnarray*}
\end{itemize}

Here, $\left( \underline{x}_{1},\cdots ,\underline{x}_{k}\right) $ plays the
role of $\left( x_{1},\cdots ,x_{k}\right) $ in Section \ref{section-setup},
and $\left( z_{1},\cdots ,z_{L},x_{1},\cdots ,x_{k}\right) $ plays the role
of $\left( z_{1},\cdots ,z_{L}\right) $ in Section \ref{section-setup}.

For $\left( x_{0},\underline{x}_{1},\cdots ,\underline{x}_{k}\right) \in 
\hat{E}^{+}=\underline{E}^{+}$, we define 
\begin{equation}
\hat{A}\left( x_{0},\underline{x}_{1},\cdots ,\underline{x}_{k}\right)
=\left( M\left( x_{0},\underline{x}_{1},\cdots ,\underline{x}_{k}\right)
\right) ^{2}  \label{P38}
\end{equation}%
with $M\left( x_{0},\underline{x}_{1},\cdots ,\underline{x}_{k}\right) $ as
in (\ref{P29}).

We check that $\hat{E}$, $\hat{E}^{+}$, $\hat{E}^{++}$, $\hat{Q}$, $\hat{A}$
satisfy hypotheses (\ref{PSQF1}), $\cdots $,(\ref{PSQF8}) in Section \ref%
{section-setup}, and that $\dim \hat{E}^{+}<\Delta $. This will allow us to
apply Propositions \ref{SAQF1} and \ref{SAQF2} to $\hat{E}$, $\hat{E}^{+}$, $%
\hat{E}^{++}$, $\hat{Q}$, $\hat{A}$.

Hypothesis (\ref{PSQF1}) for $\hat{E}$, $\hat{E}^{+}$, $\hat{E}^{++}$, $\hat{%
Q}$, $\hat{A}$ simply asserts that 
\begin{equation*}
\hat{E}\subset \mathbb{R}^{n}\text{, }\hat{E}^{+}\subset \hat{E}\times 
\underset{k\text{ copies}}{\underbrace{\mathbb{R}^{n}\times \cdots \times 
\mathbb{R}^{n}}}\text{ , }\hat{E}^{++}\subset \hat{E}^{+}\times \mathbb{R}%
^{n}\times \cdots \times \mathbb{R}^{n}\text{.}
\end{equation*}%
From (\ref{P35}) and (\ref{P36}), we see that $\hat{E}\subset \mathbb{R}^{n}$
and $\hat{E}^{+}\subset \hat{E}\times \mathbb{R}^{n}\times \cdots \times 
\mathbb{R}^{n} $. Also (\ref{P34}) and (\ref{P35}) show that $\hat{E}%
^{++}\subset \hat{E}^{+}\times \mathbb{R}^{n}\times \cdots \times \mathbb{R}%
^{n}$. Thus, Hypothesis (\ref{PSQF1}) holds for $\hat{E}$, $\hat{E}^{+}$, $%
\hat{E}^{++}$, $\hat{Q}$, $\hat{A}$.

Hypothesis (\ref{PSQF2}) for $\hat{E}$, $\hat{E}^{+}$, $\hat{E}^{++}$, $\hat{%
Q}$, $\hat{A}$ asserts that $\hat{E}$, $\hat{E}^{+}$, $\hat{E}^{++}$ are
semialgebraic, which follows from (\ref{P3}), (\ref{P34}), (\ref{P35}), (\ref%
{P36}), since $E^{++}$, $\underline{E}^{++}$ are semialgebraic.

Hypothesis (\ref{PSQF3}) for $\hat{E}$, $\hat{E}^{+}$, $\hat{E}^{++}$, $\hat{%
Q}$, $\hat{A}$ asserts that $\hat{E}$ and $\hat{E}^{+}$ are compact, which
follows from (\ref{P35}), (\ref{P36}) and the compactness of $\underline{E}%
^{+}$.

Hypothesis (\ref{PSQF4}) for $\hat{E}$, $\hat{E}^{+}$, $\hat{E}^{++}$, $\hat{%
Q}$, $\hat{A}$ asserts that $\bar{\bar{m}}\geq m \geq 0 $ and $D, I\geq 1$,
which we know from our selection of $\bar{\bar{m}}$, and from Hypothesis (%
\ref{PSQF4}) for $E$, $E^{+}$, $E^{++}$, $Q$, $A$.

Hypothesis (\ref{PSQF5}) for $\hat{E}$, $\hat{E}^{+}$, $\hat{E}^{++}$, $\hat{%
Q}$, $\hat{A}$ asserts that 
\begin{equation*}
\hat{Q}\left( x_{0},P_{0},\underline{x}_{1},P_{1},\cdots ,\underline{x}%
_{k},P_{k},z_{1},\cdots ,z_{L},x_{1},\cdots ,x_{k}\right)
\end{equation*}%
is a semialgebraic function of $x_{0},\underline{x}_{1},\cdots ,\underline{x}%
_{k},z_{1},\cdots ,z_{L},x_{1},\cdots ,x_{k}$, $P_{0},P_{1},\cdots ,P_{k}$.
This follows from (\ref{P37}), since 
\begin{equation*}
\Pi _{\left( x_{0},\underline{x}_{1},\cdots ,\underline{x}_{k}\right)
}\left( P_{0},P_{1},\cdots ,P_{k}\right)
\end{equation*}%
depends semialgebraically on $x_0, \underline{x}_1, \cdots, \underline{x}%
_k,P_0, P_1,\cdots, P_k$, the projection $\pi _{x_{i}}^{\bar{\bar{m}}%
\rightarrow \bar{m}}$ depends semialgebraically on $x_{i}$ and 
\begin{equation*}
Q\left( x_{0},\hat{P}_{0},x_{1},\hat{P}_{1},\cdots ,x_{k},\hat{P}%
_{k},z_{1},\cdots ,z_{L}\right)
\end{equation*}%
is semialgebraic in $\left( x_{0},x_{1},\cdots ,x_{k},z_{1},\cdots ,z_{L},%
\hat{P}_{0},\hat{P}_{1},\cdots ,\hat{P}_{k}\right) $.

Thus, Hypothesis (\ref{PSQF5}) holds for $\hat{E}$, $\hat{E}^{+}$, $\hat{E}%
^{++}$, $\hat{Q}$, $\hat{A}$.

Hypothesis (\ref{PSQF6}) for $\hat{E}$, $\hat{E}^{+}$, $\hat{E}^{++}$, $\hat{%
Q}$, $\hat{A}$ asserts that for fixed $x_{0},\underline{x}_{1},\cdots ,%
\underline{x}_{k},z_{1},\cdots ,z_{L},x_{1},\cdots ,x_{k}$, the map%
\begin{equation*}
\left( P_{0},P_{1},\cdots ,P_{k}\right) \mapsto \hat{Q}\left( x_{0},P_{0},%
\underline{x}_{1},P_{1},\cdots ,\underline{x}_{k},P_{k},z_{1},\cdots
,z_{L},x_{1},\cdots ,x_{k}\right)
\end{equation*}%
is a positive semidefinite quadratic form. This follows from (\ref{P37}) and
Hypothesis (\ref{PSQF6}) for $E$, $E^{+}$, $E^{++}$, $Q$, $A$, since the maps%
\begin{equation*}
\Pi _{\left( x_{0},\underline{x}_{1},\cdots ,\underline{x}_{k}\right) }\text{
and }\pi _{x_{i}}^{\bar{\bar{m}}\rightarrow \bar{m}}\text{ (for }i=1,\cdots
,k\text{) are linear}
\end{equation*}%
for fixed $x_{0},\underline{x}_{1},\cdots ,\underline{x}_{k},z_{1},\cdots
,z_{L},x_{1},\cdots ,x_{k}$.

Thus, Hypothesis (\ref{PSQF6}) holds for $\hat{E}$, $\hat{E}^{+}$, $\hat{E}%
^{++}$, $\hat{Q}$, $\hat{A}$.

Hypothesis (\ref{PSQF7}) for $\hat{E}$, $\hat{E}^{+}$, $\hat{E}^{++}$, $\hat{%
Q}$, $\hat{A}$ asserts that 
\begin{equation*}
\hat{A}\left( x_{0},\underline{x}_{1},\cdots ,\underline{x}_{k}\right)
\end{equation*}%
is a nonnegative semialgebraic function.

This is immediate from (\ref{P38}), since $M\left( x_{0},\underline{x}%
_{1},\cdots ,\underline{x}_{k}\right) $ is a semialgebraic function.

Hypothesis (\ref{PSQF8}) for $\hat{E}$, $\hat{E}^{+}$, $\hat{E}^{++}$, $\hat{%
Q}$, $\hat{A}$ asserts that 
\begin{eqnarray*}
&&\hat{Q}\left( x_{0},P_{0},\underline{x}_{1},P_{1},\cdots ,\underline{x}%
_{k},P_{k},z_{1},\cdots ,z_{L},x_{1},\cdots ,x_{k}\right) \\
&\leq &\hat{A}\left( x_{0},\underline{x}_{1},\cdots ,\underline{x}%
_{k}\right) \cdot \left(\sum_{|\alpha| \leq m} \vert \partial^{\alpha} P_0
(x_0) \vert^2+\sum_{i=1}^{k}\sum_{\left\vert \alpha \right\vert \leq \bar{%
\bar{m}}}\left\vert \partial ^{\alpha }P_{i}\left( \underline{x}_{i}\right)
\right\vert^2 \right)
\end{eqnarray*}%
for $\left( x_{0},\underline{x}_{1},\cdots ,\underline{x}_{k},z_{1},\cdots
,z_{L},x_{1},\cdots ,x_{k}\right) \in \hat{E}^{++}$.

By definitions (\ref{P37}), (\ref{P38}), this means that 
\begin{eqnarray}
&&Q\left( x_{0},\hat{P}_{0},x_{1},\pi _{x_{1}}^{\bar{\bar{m}}\rightarrow 
\bar{m}}\hat{P}_{1},\cdots ,x_{k},\pi _{x_{k}}^{\bar{\bar{m}}\rightarrow 
\bar{m}}\hat{P}_{k},z_{1},\cdots ,z_{L}\right)  \notag \\
&\leq &\left( M\left( x_{0},\underline{x}_{1},\cdots ,\underline{x}%
_{k}\right) \right) ^{2}\cdot \left(\sum_{|\alpha| \leq m} \vert
\partial^{\alpha} P_0 (x_0) \vert^2+\sum_{i=1}^{k}\sum_{\left\vert \alpha
\right\vert \leq \bar{\bar{m}}}\left\vert \partial ^{\alpha }P_{i}\left( 
\underline{x}_{i}\right) \right\vert^2 \right)  \label{P39}
\end{eqnarray}%
for $\left( x_{0},\underline{x}_{1},\cdots ,\underline{x}_{k},z_{1},\cdots
,z_{L},x_{1},\cdots ,x_{k}\right) \in \hat{E}^{++}$, where 
\begin{equation*}
\left( \hat{P}_{0},\hat{P}_{1},\cdots ,\hat{P}_{k}\right) =\Pi _{\left(
x_{0},\underline{x}_{1},\cdots ,\underline{x}_{k}\right) }\left( {P}_{0},{P}%
_{1},\cdots ,{P}_{k}\right) \text{.}
\end{equation*}%
Let us verify that (\ref{P39}) holds.

Let $\left( x_{0},\underline{x}_{1},\cdots ,\underline{x}_{k},z_{1},\cdots
,z_{L},x_{1},\cdots ,x_{k}\right) \in \hat{E}^{++}$, and let 
\begin{equation*}
P_{0}\in \mathcal{P}^{\left( m\right) }\left( \mathbb{R}^{n},\mathbb{R}%
^{D}\right) \text{, }P_{1},\cdots ,P_{k}\in \mathcal{P}^{\left( \bar{\bar{m}}%
\right) }\left( \mathbb{R}^{n},\mathbb{R}^{I}\right) \text{.}
\end{equation*}

By definition (\ref{P34}), and by Hypothesis (\ref{PSQF1}) for $E$, $E^{+}$%
, $E^{++}$, $Q$, $A$, we see that (\ref{P30a})$,\cdots $, (\ref{P30c}) hold.
Moreover, 
\begin{equation*}
\left( \hat{P}_{0},\hat{P}_{1},\cdots ,\hat{P}_{k}\right) =\Pi _{\left(
x_{0},\underline{x}_{1},\cdots ,\underline{x}_{k}\right) }\left(
P_{0},P_{1},\cdots ,P_{k}\right)
\end{equation*}%
belongs to $\hat{H}^{\text{bdd}}\left( x_{0},\underline{x}_{1},\cdots ,%
\underline{x}_{k}\right) $, by (\ref{P32}).

Hence, (\ref{P30a})$,\cdots $, (\ref{P30d}) all hold with $\left( \hat{P}%
_{0},\hat{P}_{1},\cdots ,\hat{P}_{k}\right) $ in place of $\left(
P_{0},\cdots ,P_{k}\right) ;$ from which we obtain, via (\ref{P29}), that 
\begin{eqnarray*}
&&Q\left( x_{0},\hat{P}_{0},x_{1},\pi _{x_{1}}^{\bar{\bar{m}}\rightarrow 
\bar{m}}\hat{P}_{1},\cdots ,x_{k},\pi _{x_{k}}^{\bar{\bar{m}}\rightarrow 
\bar{m}}\hat{P}_{k},z_{1},\cdots ,z_{L}\right) \\
&\leq &\left( M\left( x_{0},\underline{x}_{1},\cdots ,\underline{x}%
_{k}\right) \right) ^{2}\cdot \left(\sum_{|\alpha| \leq m} \vert
\partial^{\alpha} \hat{P}_0 (x_0) \vert^2+\sum_{i=1}^{k}\sum_{\left\vert
\alpha \right\vert \leq \bar{\bar{m}}}\left\vert \partial ^{\alpha }\hat{P}%
_{i}\left( \underline{x}_{i}\right) \right\vert^2 \right)\text{.}
\end{eqnarray*}%
Together with (\ref{P31}), this implies (\ref{P39}), completing the proof of
Hypothesis (\ref{PSQF8}) for $\hat{E}$, $\hat{E}^{+}$, $\hat{E}^{++}$, $\hat{%
Q}$, $\hat{A}$.

Thus, Hypotheses (\ref{PSQF1}), $\cdots $, (\ref{PSQF8}) all hold for $\hat{E%
}$, $\hat{E}^{+}$, $\hat{E}^{++}$, $\hat{Q}$, $\hat{A}$. Moreover, $\dim 
\hat{E}^{+}=\dim \underline{E}^{+}<\Delta $; see (\ref{P35}) and (\ref{P1}).

So we may apply Propositions \ref{SAQF1} and \ref{SAQF2} to $\hat{E}$, $\hat{%
E}^{+}$, $\hat{E}^{++}$, $\hat{Q}$, $\hat{A}$.

Thus, we learn the following.

Proposition \ref{SAQF1} for $\hat{E}$, $\hat{E}^{+}$, $\hat{E}^{++}$, $\hat{Q%
}$, $\hat{A}:$ There exist $\hat{m}\geq \bar{\bar{m}}$, and a computable family of vector spaces

\begin{equation*}
\check{H}^{\text{bdd}}\left( x_{0},\underline{x}_{1},\cdots ,\underline{x}%
_{k}\right) \subset \mathcal{P}^{\left( m\right) }\left( \mathbb{R}^{n},%
\mathbb{R}^{D}\right) \oplus \underset{k\text{ copies}}{\underbrace{\mathcal{%
P}^{\left( \hat{m}\right) }\left( \mathbb{R}^{n},\mathbb{R}^{I}\right)
\oplus \cdots \oplus \mathcal{P}^{\left( \hat{m}\right) }\left( \mathbb{R}%
^{n},\mathbb{R}^{I}\right) }},
\end{equation*}%
depending semialgebraically on $\left( x_{0},\underline{x}_{1},\cdots ,%
\underline{x}_{k}\right) \in \hat{E}^{+}=\underline{E}^{+}$, such that the
following holds.

For $x_0 \in \hat{E}$, define 
\begin{equation}
\hat{E}^+(x_0) = \{(\underline{x}_1, \cdots, \underline{x}_k) \in \mathbb{R}%
^n \times \cdots \times \mathbb{R}^n: (x_0,\underline{x}_1, \cdots, 
\underline{x}_k) \in \hat{E}^+\}  \label{p.16.a}
\end{equation}
and 
\begin{equation}
\hat{E}^{++}(x_0) = \{(\underline{x}_1, \cdots, \underline{x}_k,z_1,\cdots,
z_L, x_1,\cdots, x_k) \in \mathbb{R}^n \times \cdots \times \mathbb{R}^n:
(x_0,\underline{x}_1, \cdots, \underline{x}_k,z_1,\cdots, z_L, x_1,\cdots,
x_k) \in \hat{E}^{++}\}.  \label{p.16.b}
\end{equation}

Let $x_{0}\in \hat{E}$, $P_{0}\in \mathcal{P}^{(m)}\left( \mathbb{R}^{n},%
\mathbb{R}^{D}\right) $, $F\in C_{0}^{\infty }\left( \mathbb{R}^{n},\mathbb{R%
}^{I}\right) $ be given. Then 
\begin{subequations}
\label{P40}
\begin{align}
& \sup \left\{ 
\begin{array}{c}
\hat{Q}\left( x_{0},P_{0},\underline{x}_{1},J_{\underline{x}_{1}}^{\left( 
\bar{\bar{m}}\right) }F,\cdots ,\underline{x}_{k},J_{\underline{x}%
_{k}}^{\left( \bar{\bar{m}}\right) }F,z_{1},\cdots ,z_{L},x_{1},\cdots
,x_{k}\right) : \\ 
\left( \underline{x}_{1},\cdots ,\underline{x}_{k},z_{1},\cdots
,z_{L},x_{1},\cdots ,x_{k}\right) \in \hat{E}^{++}(x_{0})%
\end{array}%
\right\} <\infty  \label{P40a} \\
& \text{if and only if }  \notag \\
& \left( P_{0},J_{\underline{x}_{1}}^{\left( \hat{m}\right) }F,\cdots ,J_{%
\underline{x}_{k}}^{\left( \hat{m}\right) }F\right) \in \check{H}^{\text{bdd}%
}\left( x_{0},\underline{x}_{1},\cdots ,\underline{x}_{k}\right)
\label{P40b}
\end{align}%
for all $\left( \underline{x}_{1},\cdots ,\underline{x}_{k}\right) \in \hat{E%
}^{+}(x_{0})=\underline{E}^{+}(x_{0})$.

Proposition \ref{SAQF2} for $\hat{E}$, $\hat{E}^{+}$, $\hat{E}^{++}$, $\hat{Q%
}$, $\hat{A}:$ There exist $\hat{m}^{+}\geq \bar{\bar{m}}$, and a computable family of vector spaces 
\end{subequations}
\begin{equation*}
\check{H}^{\lim }\left( x_{0}\right) \subset \mathcal{P}^{\left( m\right)
}\left( \mathbb{R}^{n},\mathbb{R}^{D}\right) \oplus \mathcal{P}^{\left( \hat{%
m}^{+}\right) }\left( \mathbb{R}^{n},\mathbb{R}^{I}\right) ,
\end{equation*}%
depending semialgebraically on $x_{0}\in \hat{E}$, such that the following
holds.

Let $x_{0}\in \hat{E}$, $P_{0}\in \mathcal{P}^{\left( m\right) }\left( 
\mathbb{R}^{n},\mathbb{R}^{D}\right) $, $F\in C_{0}^{\infty }\left( \mathbb{R%
}^{n},\mathbb{R}^{I}\right) $ be given. Assume that condition (\ref{P40a})
holds. Then

\begin{subequations}
\label{P41}
\begin{align}
& \underset{\left(\underline{x}_{1},\cdots ,\underline{x}_{k},z_{1},\cdots
,z_{L},x_{1},\cdots ,x_{k}\right) \in \hat{E}^{++}(x_0)}{\lim_{\underline{x}%
_{1},\cdots ,\underline{x}_{k},z_{1},\cdots ,z_{L},x_{1},\cdots
,x_{k}\rightarrow x_{0}}}\hat{Q}(x_{0},P_{0},\underline{x}_{1},J_{\underline{%
x}_{1}}^{\left( \bar{\bar{m}}\right) }F,\cdots ,\underline{x}_{k},J_{%
\underline{x}_{k}}^{\left( \bar{\bar{m}}\right) }F,z_{1},\cdots
,z_{L},x_{1},\cdots ,x_{k})=0  \label{P41a} \\
& \text{if and only if }  \notag \\
& \left( P_{0},J_{\underline{x}_{0}}^{\left( \hat{m}^{+}\right) }F\right)
\in \check{H}^{\lim }\left( x_{0}\right) \text{.}  \label{P41b}
\end{align}

Now let $x_{0}\in E$, $P_{0}\in \mathcal{P}^{\left( m\right) }\left( \mathbb{%
R}^{n},\mathbb{R}^{D}\right) $, $F\in C_{0}^{\infty }\left( \mathbb{R}^{n},%
\mathbb{R}^{I}\right) $ be given and suppose that 
\end{subequations}
\begin{equation}
\sup \left\{ 
\begin{array}{c}
Q\left( x_{0},P_{0},x_{1},J_{x_{1}}^{\left( \bar{m}\right) }F,\cdots
,x_{k},J_{x_{k}}^{\left( \bar{m}\right) }F,z_{1},\cdots ,z_{L}\right) : \\ 
\left( x_{1},\cdots ,x_{k},z_{1},\cdots ,z_{L}\right) \in E^{++}(x_{0})%
\end{array}%
\right\} <\infty .  \label{P42}
\end{equation}

(In the above $\sup $, $x_{0}$ has been fixed, while $x_{1},\cdots
,x_{k},z_{1},\cdots ,z_{L}$ vary.)

Recall (\ref{P1}), $\cdots $, (\ref{P9}) and (\ref{P10a}), (\ref{P10b}). We
have the following, from the equivalence of (\ref{P10a}) and (\ref{P10b}):

\begin{itemize}
\item[\refstepcounter{equation}\text{(\theequation)}\label{P43}] If $%
x_{0}\in \underline{E}$, then $\left( P_{0},J_{x_{1}}^{\left( \underline{m}%
^{\prime }\right) }F,\cdots ,J_{x_{k}}^{\left( \underline{m}^{\prime
}\right) }F\right) \in \underline{H}^{\text{bdd}}\left( x_{0},x_{1},\cdots
,x_{k}\right) $ for each $\left( x_{1},\cdots ,x_{k}\right) \in \underline{E}%
^{+}(x_0)$.
\end{itemize}

If $x_{0}\in \hat{E}$, then $\underline{E}^{+}\left( x_{0}\right)
\not=\emptyset $ by (\ref{P36}). Hence, by (\ref{P16}) and (\ref{P42}), we
find that (\ref{P25}) holds for any $\left(\underline{x}_{1},\cdots ,%
\underline{x}_{k}\right) \in \underline{E}^{+}(x_0)$, with $P_{i}:=J_{%
\underline{x}_{i}}^{\left( \bar{\bar{m}}\right) }F$ for $i=1,\cdots ,k$.

(To see this, recall that (\ref{P16}) holds whenever (\ref{P17}) holds, and
note that (\ref{P17}) holds whenever $x_{0},x_{1},\cdots ,x_{k},z_{1},\cdots
,z_{L},\underline{x}_{1},\cdots ,\underline{x}_{k}$ are as in (\ref{P25}).)
Thanks to (\ref{P27}), we therefore learn that

\begin{itemize}
\item[\refstepcounter{equation}\text{(\theequation)}\label{P44}] $\left(
P_{0},J_{\underline{x}_{1}}^{\left( \bar{\bar{m}}\right) }F,\cdots ,J_{%
\underline{x}_{k}}^{\left( \bar{\bar{m}}\right) }F\right) \in \hat{H}^{\text{%
bdd}}\left( x_{0},\underline{x}_{1},\cdots ,\underline{x}_{k}\right) $
whenever $\left( \underline{x}_{1},\cdots ,\underline{x}_{k}\right) \in 
\underline{E}^{+}\left( x_{0}\right) $.
\end{itemize}

Hence, by (\ref{P33}), we have

\begin{itemize}
\item[\refstepcounter{equation}\text{(\theequation)}\label{P45}] $\left(
P_{0},J_{\underline{x}_{1}}^{\left( \bar{\bar{m}}\right) }F,\cdots ,J_{%
\underline{x}_{k}}^{\left( \bar{\bar{m}}\right) }F\right) =\Pi _{\left(
x_{0},\underline{x}_{1},\cdots ,\underline{x}_{k}\right) }\left( P_{0},J_{%
\underline{x}_{1}}^{\left( \bar{\bar{m}}\right) }F,\cdots ,J_{\underline{x}%
_{k}}^{\left( \bar{\bar{m}}\right) }F\right) $ whenever $\left( \underline{x}%
_{1},\cdots ,\underline{x}_{k}\right) \in \underline{E}^{+}\left(
x_{0}\right) $.
\end{itemize}

From (\ref{P37}) and (\ref{P45}), we obtain the following. 
\begin{eqnarray}
&&\hat{Q}\left( x_{0},P_{0},\underline{x}_{1},J_{\underline{x}_{1}}^{\left( 
\bar{\bar{m}}\right) }F,\cdots ,\underline{x}_{k},J_{\underline{x}%
_{k}}^{\left( \bar{\bar{m}}\right) }F,z_{1},\cdots ,z_{L},x_{1},\cdots
,x_{k}\right)  \notag \\
&=&Q\left( x_{0},P_{0},x_{1},\pi _{x_{1}}^{\bar{\bar{m}}\rightarrow \bar{m}%
}J_{\underline{x}_{1}}^{\left( \bar{\bar{m}}\right) }F,\cdots ,x_{k},\pi
_{x_{k}}^{\bar{\bar{m}}\rightarrow \bar{m}}J_{\underline{x}_{k}}^{\left( 
\bar{\bar{m}}\right) }F,z_{1},\cdots ,z_{L}\right) ,  \label{P46}
\end{eqnarray}%
whenever $\left(\underline{x}_{1},\cdots ,\underline{x}_{k},z_{1},\cdots
,z_{L},x_{1},\cdots ,x_{k}\right) \in \hat{E}^{++}(x_0)$.

Now, (\ref{P17}) holds whenever 
\begin{equation*}
\left(\underline{x}_{1},\cdots ,\underline{x}_{k},z_{1},\cdots
,z_{L},x_{1},\cdots ,x_{k}\right) \in \hat{E}^{++}(x_0).
\end{equation*}%
Hence, for such points, (\ref{P16}) holds. From (\ref{P16}) and (\ref{P42}),
we conclude that 
\begin{equation}
\sup \left\{ 
\begin{array}{c}
Q\left( x_{0},P_{0},x_{1},\pi _{x_{1}}^{\bar{\bar{m}}\rightarrow \bar{m}}J_{%
\underline{x}_{1}}^{\left( \bar{\bar{m}}\right) }F,\cdots ,x_{k},\pi
_{x_{k}}^{\bar{\bar{m}}\rightarrow \bar{m}}J_{\underline{x}_{k}}^{\left( 
\bar{\bar{m}}\right) }F,z_{1},\cdots ,z_{L}\right) : \\ 
\left(\underline{x}_{1},\cdots ,\underline{x}_{k},z_{1},\cdots
,z_{L},x_{1},\cdots ,x_{k}\right) \in \hat{E}^{++}(x_0)%
\end{array}%
\right\} <\infty \text{.}  \label{P47}
\end{equation}

In (\ref{P47}), $x_{0}$ is fixed and $\underline{x}_{1},\cdots ,\underline{x}%
_{k},z_{1},\cdots ,z_{L},x_{1},\cdots ,x_{k}$ vary.

From (\ref{P46}) and (\ref{P47}), we see that 
\begin{equation}
\sup \left\{ 
\begin{array}{c}
\hat{Q}\left( x_{0},P_{0},\underline{x}_{1},J_{\underline{x}_{1}}^{\left( 
\bar{\bar{m}}\right) }F,\cdots ,\underline{x}_{k},J_{\underline{x}%
_{k}}^{\left( \bar{\bar{m}}\right) }F,z_{1},\cdots ,z_{L},x_{1},\cdots
,x_{k}\right) : \\ 
\left(\underline{x}_{1},\cdots ,\underline{x}_{k},z_{1},\cdots
,z_{L},x_{1},\cdots ,x_{k}\right) \in \hat{E}^{++}(x_0)%
\end{array}%
\right\} <\infty \text{.}  \label{P48}
\end{equation}

Thus, we have verified (\ref{P40a}) for $x_{0}$, $P_{0}$, $F$. Recalling the
equivalence of (\ref{P40a}), (\ref{P40b}), we now know that (\ref{P40b})
holds.

We have proven this under the assumption that $\underline{E}^{+}\left(
x_{0}\right) \not=\emptyset $, and also assuming (\ref{P42}).

Thus, we have proven the following.

\begin{itemize}
\item[\refstepcounter{equation}\text{(\theequation)}\label{P49}] Let $x_{0}\in E$, $P_{0}\in \mathcal{P}^{\left( m\right) }\left( \mathbb{R}%
^{n},\mathbb{R}^{D}\right) $, $F\in C_{0}^{\infty }\left( \mathbb{R}^{n},%
\mathbb{R}^{I}\right) $.
Assume (\ref%
{P42}), and suppose that $\underline{E}^{+}\left( x_{0}\right)
\not=\emptyset $. Then 
\begin{eqnarray*}
\left( P_{0},J_{\underline{x}_{1}}^{\left( \bar{\bar{m}}\right) }F,\cdots
,J_{\underline{x}_{k}}^{\left( \bar{\bar{m}}\right) }F\right) &\in &\hat{H}^{%
\text{bdd}}\left( x_{0},\underline{x}_{1},\cdots ,\underline{x}_{k}\right)
\end{eqnarray*}%
and 
\begin{equation*}
\left( P_{0},J_{\underline{x}_{1}}^{\left( \hat{m}\right) }F,\cdots ,J_{%
\underline{x}_{k}}^{\left( \hat{m}\right) }F\right) \in \check{H}^{\text{bdd}%
}\left( x_{0},\underline{x}_{1},\cdots ,\underline{x}_{k}\right) \text{.}
\end{equation*}
for all $\left( \underline{x}_{1},\cdots ,\underline{x}_{k}\right) \in 
\underline{E}^{+}\left( x_{0}\right) $. (See (\ref{P44}).)
\end{itemize}

From (\ref{P6}), (\ref{P43}) and (\ref{P49}), we draw the following
conclusion.

\begin{itemize}
\item[\refstepcounter{equation}\text{(\theequation)}\label{P50}] Let $%
x_{0}\in E$, $P_{0}\in \mathcal{P}^{\left( m\right) }\left( \mathbb{R}^{n},%
\mathbb{R}^{D}\right) $, $F\in C_{0}^{\infty }\left( \mathbb{R}^{n},\mathbb{R%
}^{I}\right) $. Assume that 
\begin{equation*}
\sup \left\{ 
\begin{array}{c}
Q\left( x_{0},P_{0},x_{1},J_{x_{1}}^{\left( \bar{m}\right) }F,\cdots
,x_{k},J_{x_{k}}^{\left( \bar{m}\right) }F,z_{1},\cdots ,z_{L}\right) : \\ 
\left( x_{1},\cdots ,x_{k},z_{1},\cdots ,z_{L}\right) \in E^{++}(x_{0})%
\end{array}%
\right\} <\infty \text{.}
\end{equation*}%
Then for all $\left( \underline{x}_{1},\cdots ,\underline{x}_{k}\right) \in 
\underline{E}^{+}\left( x_{0}\right) $, we have 
\begin{eqnarray*}
\left( P_{0},J_{\underline{x}_{1}}^{\left( \underline{m}^{\prime }\right)
}F,\cdots ,J_{\underline{x}_{k}}^{\left( \underline{m}^{\prime }\right)
}F\right) &\in &\underline{H}^{\text{bdd}}\left( x_{0},\underline{x}%
_{1},\cdots ,\underline{x}_{k}\right) \\
\left( P_{0},J_{\underline{x}_{1}}^{\left( \bar{\bar{m}}\right) }F,\cdots
,J_{\underline{x}_{k}}^{\left( \bar{\bar{m}}\right) }F\right) &\in &\hat{H}^{%
\text{bdd}}\left( x_{0},\underline{x}_{1},\cdots ,\underline{x}_{k}\right)
\end{eqnarray*}%
and 
\begin{equation*}
\left( P_{0},J_{\underline{x}_{1}}^{\left( \hat{m}\right) }F,\cdots ,J_{%
\underline{x}_{k}}^{\left( \hat{m}\right) }F\right) \in \check{H}^{\text{bdd}%
}\left( x_{0},\underline{x}_{1},\cdots ,\underline{x}_{k}\right) \text{.}
\end{equation*}
\end{itemize}

Conversely, let $x_{0}\in E$, $P_{0}\in \mathcal{P}^{\left( m\right) }\left( 
\mathbb{R}^{n},\mathbb{R}^{D}\right) $, $F\in C_{0}^{\infty }\left( \mathbb{R%
}^{n},\mathbb{R}^{I}\right) $. Assume that for all $\left( \underline{x}%
_{1},\cdots ,\underline{x}_{k}\right) \in \underline{E}^{+}\left(
x_{0}\right) $, we have 
\begin{equation}
\left( P_{0},J_{\underline{x}_{1}}^{\left( \underline{m}^{\prime }\right)
}F,\cdots ,J_{\underline{x}_{k}}^{\left( \underline{m}^{\prime }\right)
}F\right) \in \underline{H}^{\text{bdd}}\left( x_{0},\underline{x}%
_{1},\cdots ,\underline{x}_{k}\right)  \label{P51}
\end{equation}%
\begin{equation}
\left( P_{0},J_{\underline{x}_{1}}^{\left( \bar{\bar{m}}\right) }F,\cdots
,J_{\underline{x}_{k}}^{\left( \bar{\bar{m}}\right) }F\right) \in \hat{H}^{%
\text{bdd}}\left( x_{0},\underline{x}_{1},\cdots ,\underline{x}_{k}\right)
\label{P52}
\end{equation}%
and%
\begin{equation}
\left( P_{0},J_{\underline{x}_{1}}^{\left( \hat{m}\right) }F,\cdots ,J_{%
\underline{x}_{k}}^{\left( \hat{m}\right) }F\right) \in \check{H}^{\text{bdd}%
}\left( x_{0},\underline{x}_{1},\cdots ,\underline{x}_{k}\right) \text{.}
\label{P52a}
\end{equation}

We will prove (\ref{P42}) under the above assumptions.

To see this, first suppose $\underline{E}^{+}\left( x_{0}\right) =\emptyset
. $ Then since (\ref{P24}) implies (\ref{P18}), we have 
\begin{equation*}
\sup \left\{ 
\begin{array}{c}
Q\left( x_{0},P_{0},x_{1},J_{x_{1}}^{\left( \bar{m}\right) }F,\cdots
,x_{k},J_{x_{k}}^{\left( \bar{m}\right) }F,z_{1},\cdots ,z_{L}\right) : \\ 
\left(x_{1},\cdots ,x_{k},z_{1},\cdots ,z_{L}\right) \in E^{++}(x_0)%
\end{array}%
\right\} <\infty \text{,}
\end{equation*}%
where $x_{0}$ remains fixed and $z_{1},\cdots ,z_{L},x_{1},\cdots ,x_{k}$
vary.

Thus, (\ref{P42}) holds if $\underline{E}^{+}\left( x_{0}\right) =\emptyset $%
. We now suppose that $\underline{E}^{+}\left( x_{0}\right) \not=\emptyset $%
. Note that $x_{0}\in \hat{E}$. (See (\ref{P36}).)

From (\ref{P52}), (\ref{P52a}), and the equivalence of (\ref{P40a}), (\ref%
{P40b}), we see that 
\begin{equation}
\sup \left\{ 
\begin{array}{c}
\hat{Q}\left( x_{0},P_{0},\underline{x}_{1},J_{\underline{x}_{1}}^{\left( 
\bar{\bar{m}}\right) }F,\cdots ,\underline{x}_{k},J_{\underline{x}%
_{k}}^{\left( \bar{\bar{m}}\right) }F,z_{1},\cdots ,z_{L},x_{1},\cdots
,x_{k}\right) : \\ 
\left( \underline{x}_{1},\cdots ,\underline{x}_{k},z_{1},\cdots
,z_{L},x_{1},\cdots ,x_{k}\right) \in \hat{E}^{++}(x_0)%
\end{array}%
\right\} <\infty \text{.}  \label{P53}
\end{equation}

From (\ref{P52}), (\ref{P52a}), and the definition of $\Pi _{\left( x_{0},%
\underline{x}_{1},\cdots ,\underline{x}_{k}\right) }$, we learn that 
\begin{equation*}
\left( P_{0},J_{\underline{x}_{1}}^{\left( \bar{\bar{m}}\right) }F,\cdots
,J_{\underline{x}_{k}}^{\left( \bar{\bar{m}}\right) }F\right) =\Pi _{\left(
x_{0},\underline{x}_{1},\cdots ,\underline{x}_{k}\right) }\left( P_{0},J_{%
\underline{x}_{1}}^{\left( \bar{\bar{m}}\right) }F,\cdots ,J_{\underline{x}%
_{k}}^{\left( \bar{\bar{m}}\right) }F\right)
\end{equation*}%
for $\left( \underline{x}_{1},\cdots ,\underline{x}_{k}\right) \in 
\underline{E}^{+}\left( x_{0}\right) $.

Therefore, (\ref{P37}) and (\ref{P53}) yield the following: 
\begin{equation}
\sup \left\{ 
\begin{array}{c}
Q\left( x_{0},P_{0},x_{1},\pi _{x_{1}}^{\bar{\bar{m}}\rightarrow \bar{m}}J_{%
\underline{x}_{1}}^{\left( \bar{\bar{m}}\right) }F,\cdots ,x_{k},\pi
_{x_{k}}^{\bar{\bar{m}}\rightarrow \bar{m}}J_{\underline{x}_{k}}^{\left( 
\bar{\bar{m}}\right) }F,z_{1},\cdots ,z_{L}\right) : \\ 
\left(\underline{x}_{1},\cdots ,\underline{x}_{k},z_{1},\cdots
,z_{L},x_{1},\cdots ,x_{k}\right) \in \hat{E}^{++}(x_0)%
\end{array}%
\right\} <\infty .  \label{P54}
\end{equation}

In (\ref{P53}) and (\ref{P54}), $x_{0}$ stays fixed, while the sup is taken
over $\underline{x}_{1},\cdots ,\underline{x}_{k},z_{1},\cdots
,z_{L},x_{1},\cdots ,x_{k}$.

When $\left(\underline{x}_{1},\cdots ,\underline{x}_{k},z_{1},\cdots
,z_{L},x_{1},\cdots ,x_{k}\right) \in \hat{E}^{++}(x_0)$, we have (\ref{P17}%
); hence also (\ref{P15}) holds in this case.

From (\ref{P15}) and (\ref{P54}), we see that 
\begin{equation}
\sup \left\{ 
\begin{array}{c}
Q\left( x_{0},P_{0},x_{1},J_{x_{1}}^{\left( \bar{m}\right) }F,\cdots
,x_{k},J_{x_{k}}^{\left( \bar{m}\right) }F,z_{1},\cdots ,z_{L}\right) : \\ 
\left(\underline{x}_{1},\cdots ,\underline{x}_{k},z_{1},\cdots
,z_{L},x_{1},\cdots ,x_{k}\right) \in \hat{E}^{++}(x_0)%
\end{array}%
\right\} <\infty \text{.}  \label{P55}
\end{equation}

Whenever $\left(x_{1},\cdots ,x_{k},z_{1},\cdots ,z_{L}\right) \in
E^{++}(x_0)$ and $\underline{E}^{+}\left( x_{0}\right) \not=\emptyset $, we
can pick $\left( \underline{x}_{1},\cdots ,\underline{x}_{k}\right) \in 
\underline{E}^{+}\left( x_{0}\right) $ as close as possible to $\left(
x_{1},\cdots ,x_{k}\right) $, and we will have 
\begin{equation*}
\left(\underline{x}_{1},\cdots ,\underline{x}_{k},z_{1},\cdots
,z_{L},x_{1},\cdots ,x_{k}\right) \in \hat{E}^{++}(x_0)
\end{equation*}%
unless 
\begin{equation*}
\left( x_{1},\cdots ,x_{k}\right) \in \underline{E}^{+}\left( x_{0}\right) .%
\text{ (See (\ref{P34}).)}
\end{equation*}

Therefore, (\ref{P55}) implies the following. 
\begin{equation}
\sup \left\{ 
\begin{array}{c}
Q\left( x_{0},P_{0},x_{1},J_{x_{1}}^{\left( \bar{m}\right) }F,\cdots
,x_{k},J_{x_{k}}^{\left( \bar{m}\right) }F,z_{1},\cdots ,z_{L}\right) : \\ 
\left(x_{1},\cdots ,x_{k},z_{1},\cdots ,z_{L}\right) \in E^{++}(x_0), \\ 
\left( x_{1},\cdots ,x_{k}\right) \not\in \underline{E}^{+}\left(
x_{0}\right)%
\end{array}%
\right\} <\infty .  \label{P56}
\end{equation}

In (\ref{P55}) and (\ref{P56}), $x_{0}$ stays fixed; the $\sup $ is taken
over the other variables.

Also, from (\ref{P8}), the equivalence of (\ref{P10a}) to (\ref{P10b}), and (%
\ref{P51}), we see that 
\begin{equation}
\sup \left\{ 
\begin{array}{c}
Q\left( x_{0},P_{0},x_{1},J_{x_{1}}^{\left( \bar{m}\right) }F,\cdots
,x_{k},J_{x_{k}}^{\left( \bar{m}\right) }F,z_{1},\cdots ,z_{L}\right) : \\ 
\left(x_{1},\cdots ,x_{k},z_{1},\cdots ,z_{L}\right) \in E^{++}(x_0), \\ 
\left( x_{1},\cdots ,x_{k}\right) \in \underline{E}^{+}\left( x_{0}\right)%
\end{array}%
\right\} <\infty \text{,}  \label{P57}
\end{equation}%
where again $x_{0}$ stays fixed. From (\ref{P56}) and (\ref{P57}), we learn
that 
\begin{equation*}
\sup \left\{ 
\begin{array}{c}
Q\left( x_{0},P_{0},x_{1},J_{x_{1}}^{\left( \bar{m}\right) }F,\cdots
,x_{k},J_{x_{k}}^{\left( \bar{m}\right) }F,z_{1},\cdots ,z_{L}\right) : \\ 
\left(x_{1},\cdots ,x_{k},z_{1},\cdots ,z_{L}\right) \in E^{++}(x_0)%
\end{array}%
\right\} <\infty \text{,}
\end{equation*}%
with $x_{0}$ fixed as usual. Thus, as promised, we have proven (\ref{P42}).

Recalling (\ref{P50}), we now see that we have established the following.

\begin{itemize}
\item[\refstepcounter{equation}\text{(\theequation)}\label{P58}] Let $%
x_{0}\in E$, $P_{0}\in \mathcal{P}^{\left( m\right) }\left( \mathbb{R}^{n},%
\mathbb{R}^{D}\right) $, $F\in C_{0}^{\infty }\left( \mathbb{R}^{n},\mathbb{R%
}^{I}\right) $. Then 
\begin{equation*}
Q\left( x_{0},P_{0},x_{1},J_{x_{1}}^{\left( \bar{m}\right) }F,\cdots
,x_{k},J_{x_{k}}^{\left( \bar{m}\right) }F,z_{1},\cdots ,z_{L}\right)
\end{equation*}%
remains bounded over all $\left( x_{1},\cdots ,x_{k},z_{1},\cdots
,z_{L}\right) \in E^{++}(x_{0})$ if and only if for all $\left( \underline{x}%
_{1},\cdots ,\underline{x}_{k}\right) \in \underline{E}^{+}\left(
x_{0}\right) $, we have%
\begin{eqnarray*}
\left( P_{0},J_{\underline{x}_{1}}^{\left( \underline{m}^{\prime }\right)
}F,\cdots ,J_{\underline{x}_{k}}^{\left( \underline{m}^{\prime }\right)
}F\right) &\in &\underline{H}^{\text{bdd}}\left( x_{0},\underline{x}%
_{1},\cdots ,\underline{x}_{k}\right) \\
\left( P_{0},J_{\underline{x}_{1}}^{\left( \bar{\bar{m}}\right) }F,\cdots
,J_{\underline{x}_{k}}^{\left( \bar{\bar{m}}\right) }F\right) &\in &\hat{H}^{%
\text{bdd}}\left( x_{0},\underline{x}_{1},\cdots ,\underline{x}_{k}\right)
\end{eqnarray*}%
and 
\begin{equation*}
\left( P_{0},J_{\underline{x}_{1}}^{\left( \hat{m}\right) }F,\cdots ,J_{%
\underline{x}_{k}}^{\left( \hat{m}\right) }F\right) \in \check{H}^{\text{bdd}%
}\left( x_{0},\underline{x}_{1},\cdots ,\underline{x}_{k}\right) \text{.}
\end{equation*}
\end{itemize}

From (\ref{P58}), we see easily that the conclusion of Proposition \ref%
{SAQF1} holds for $E$, $E^{+}$, $E^{++}$, $Q$, $A$. This completes the
induction step in our proof of Proposition \ref{SAQF1}.

We turn our attention to the induction step in the proof of Proposition \ref%
{SAQF2}.

Let $x_{0}\in E$, $P_{0}\in \mathcal{P}^{\left( m\right) }\left( \mathbb{R}%
^{n},\mathbb{R}^{D}\right) $, $F\in C_{0}^{\infty }\left( \mathbb{R}^{n},%
\mathbb{R}^{I}\right) $ be given. Assume that 
\begin{equation}
\sup \left\{ 
\begin{array}{c}
Q\left( x_{0},P_{0},x_{1},J_{x_{1}}^{\left( \bar{m}\right) }F,\cdots
,x_{k},J_{x_{k}}^{\left( \bar{m}\right) }F,z_{1},\cdots ,z_{L}\right) : \\ 
\left( x_{1},\cdots ,x_{k},z_{1},\cdots ,z_{L}\right) \in E^{++}\left(
x_{0}\right)%
\end{array}%
\right\} <\infty .  \label{P59}
\end{equation}

In our discussion of the induction step for Proposition \ref{SAQF1}, we
already saw that (\ref{P59}) implies the following (see (\ref{P46}) and (\ref%
{P48})).%
\begin{eqnarray}
&&\hat{Q}\left( x_{0},P_{0},\underline{x}_{1},J_{\underline{x}_{1}}^{\left( 
\bar{\bar{m}}\right) }F,\cdots ,\underline{x}_{k},J_{\underline{x}%
_{k}}^{\left( \bar{\bar{m}}\right) }F,z_{1},\cdots ,z_{L},x_{1},\cdots
,x_{k}\right)  \notag \\
&=&Q\left( x_{0},P_{0},x_{1},\pi _{x_{1}}^{\bar{\bar{m}}\rightarrow \bar{m}%
}J_{\underline{x}_{1}}^{\left( \bar{\bar{m}}\right) }F,\cdots ,x_{k},\pi
_{x_{k}}^{\bar{\bar{m}}\rightarrow \bar{m}}J_{\underline{x}_{k}}^{\left( 
\bar{\bar{m}}\right) }F,z_{1},\cdots ,z_{L}\right)  \label{P60}
\end{eqnarray}%
for 
\begin{equation*}
\left( x_{0},\underline{x}_{1},\cdots ,\underline{x}_{k},z_{1},\cdots
,z_{L},x_{1},\cdots ,x_{k}\right) \in \hat{E}^{++}\text{.}
\end{equation*}

\begin{equation}
\sup \left\{ 
\begin{array}{c}
\hat{Q}\left( x_{0},P_{0},\underline{x}_{1},J_{\underline{x}_{1}}^{\left( 
\bar{\bar{m}}\right) }F,\cdots ,\underline{x}_{k},J_{\underline{x}%
_{k}}^{\left( \bar{\bar{m}}\right) }F,z_{1},\cdots ,z_{L},x_{1},\cdots
,x_{k}\right) : \\ 
\left(\underline{x}_{1},\cdots ,\underline{x}_{k},z_{1},\cdots
,z_{L},x_{1},\cdots ,x_{k}\right) \in \hat{E}^{++}(x_0)%
\end{array}%
\right\} <\infty \text{,}  \label{P61}
\end{equation}%
with $x_{0}$ held fixed in (\ref{P60}) and (\ref{P61}).

Now suppose $\underline{E}^{+}\left( x_{0}\right) \not=\emptyset $, i.e., $%
x_{0}\in \hat{E}$ (see (\ref{P36})). Then (\ref{P60}), (\ref{P61}), and the
equivalence of (\ref{P41a}), (\ref{P41b}), together show that

\begin{itemize}
\item[\refstepcounter{equation}\text{(\theequation)}\label{P62}] $\left(
P_{0},J_{x_{0}}^{\left( \hat{m}^{+}\right) }F\right) \in \check{H}^{\lim
}\left( x_{0}\right) $ if and only if%
\begin{equation*}
\underset{\left( \underline{x}_{1},\cdots ,\underline{x}_{k},z_{1},\cdots
,z_{L},x_{1},\cdots ,x_{k}\right) \in \hat{E}^{++}(x_0)}{\lim_{\underline{x}%
_{1},\cdots ,\underline{x}_{k},z_{1},\cdots ,z_{L},x_{1},\cdots
,x_{k}\rightarrow x_{0}}}Q\left( x_{0},P_{0},x_{1},\pi _{x_{1}}^{\bar{\bar{m}%
}\rightarrow \bar{m}}J_{\underline{x}_{1}}^{\left( \bar{\bar{m}}\right)
}F,\cdots ,x_{k},\pi _{x_{k}}^{\bar{\bar{m}}\rightarrow \bar{m}}J_{%
\underline{x}_{k}}^{\left( \bar{\bar{m}}\right) }F,z_{1},\cdots
,z_{L}\right) =0\text{.}
\end{equation*}
\end{itemize}

For $\left(x_0,\underline{x}_{1},\cdots ,\underline{x}_{k},z_{1},\cdots
,z_{L},x_{1},\cdots ,x_{k}\right) \in \hat{E}^{++}$ we have (\ref{P17}) and
therefore (\ref{P15}), (\ref{P16}). (See (\ref{P34}).)

Consequently, (\ref{P62}) implies the following.

\begin{itemize}
\item[\refstepcounter{equation}\text{(\theequation)}\label{P63}] $\left(
P_{0},J_{x_{0}}^{\left( \hat{m}^{+}\right) }F\right) \in \check{H}^{\lim
}\left( x_{0}\right) $ if and only if 
\begin{equation*}
\underset{\left( \underline{x}_{1},\cdots ,\underline{x}_{k},z_{1},\cdots
,z_{L},x_{1},\cdots ,x_{k}\right) \in \hat{E}^{++}(x_0)}{\lim_{\underline{x}%
_{1},\cdots ,\underline{x}_{k},z_{1},\cdots ,z_{L},x_{1},\cdots
,x_{k}\rightarrow x_{0}}}Q\left( x_{0},P_{0},x_{1},J_{x_{1}}^{\left( {\bar{m}%
}\right) }F,\cdots ,x_{k},J_{x_{k}}^{\left( {\bar{m}}\right) }F,z_{1},\cdots
,z_{L}\right) =0\text{.}
\end{equation*}
\end{itemize}

Let $\delta >0$. If $\left( x_{0},\underline{x}_{1},\cdots ,\underline{x}%
_{k},z_{1},\cdots ,z_{L},x_{1},\cdots ,x_{k}\right) \in \hat{E}^{++}$ and 
\begin{equation*}
\left\vert \underline{x}_{i}-x_{0}\right\vert ,\left\vert
x_{i}-x_{0}\right\vert <\delta \text{ (all }i\text{),}
\end{equation*}%
and also 
\begin{equation*}
\left\vert z_{i}-x_{0}\right\vert <\delta \text{ (all }i\text{),}
\end{equation*}%
it follows that

\begin{itemize}
\item[\refstepcounter{equation}\text{(\theequation)}\label{P64}] $\left(
x_{1},\cdots ,x_{k},z_{1},\cdots ,z_{L}\right) \in E^{++}(x_0)$, $%
\left(x_{1},\cdots ,x_{k}\right) \not\in \underline{E}^{+}(x_0)$, and $%
\left\vert x_{i}-x_{0}\right\vert <\delta $ (all $i$), $\left\vert
z_{i}-x_{0}\right\vert <\delta $ (all $i$), see (\ref{P34}).
\end{itemize}

Conversely, suppose (\ref{P64}) holds and assume $(x_0, \cdots, x_0) \in 
\underline{E}^{+}(x_0)$. Let $\left( \underline{x}_{1},\cdots ,\underline{x}%
_{k}\right) \in \underline{E}^{+}\left( x_{0}\right) $ be as close as
possible to $\left( x_{1},\cdots ,x_{k}\right) $. Then by (\ref{P34}), $%
\left(\underline{x}_{1},\cdots ,\underline{x}_{k},z_{1},\cdots
,z_{L},x_{1},\cdots ,x_{k}\right) $ belongs to $\hat{E}^{++}(x_0);$ and we
have 
\begin{eqnarray*}
\left\vert \underline{x}_{i}-x_{0}\right\vert &\leq &\dist\left( \left( 
\underline{x}_{1},\cdots ,\underline{x}_{k}\right) ,\left( x_{1},\cdots
,x_{k}\right) \right) +\left\vert x_{i}-x_{0}\right\vert \\
&\leq &\dist\left( \left( x_{0},\cdots ,x_{0}\right) ,\left( x_{1},\cdots
,x_{k}\right) \right) +\left\vert x_{i}-x_{0}\right\vert \\
&<&C\delta \text{ for each }i\text{.}
\end{eqnarray*}

In view of the above remarks, (\ref{P63}) implies the following.

\begin{itemize}
\item[\refstepcounter{equation}\text{(\theequation)}\label{P65}] Suppose (%
\ref{P59}) holds, and suppose $\left( x_{0},\cdots ,x_{0}\right) \in 
\underline{E}^{+}\left( x_{0}\right) $. Then $\left( P_{0},J_{x_{0}}^{\left( 
\hat{m}^{+}\right) }F\right) \in \check{H}^{\lim }\left( x_{0}\right) $ if
and only if 
\begin{equation*}
\underset{\left( x_{1},\cdots ,x_{k}\right) \not\in \underline{E}^{+}\left(
x_{0}\right) }{\underset{\left(x_{1},\cdots ,x_{k},z_{1},\cdots
,z_{L}\right) \in {E}^{++}(x_0)}{\lim_{x_{1},\cdots ,x_{k},z_{1},\cdots
,z_{L}\rightarrow x_{0}}}}Q\left( x_{0},P_{0},x_{1},J_{x_{1}}^{\left( \bar{m}%
\right) }F,\cdots ,x_{k},J_{x_{k}}^{\left( \bar{m}\right) }F,z_{1},\cdots
,z_{L}\right) =0\text{.}
\end{equation*}
\end{itemize}

On the other hand, suppose (\ref{P59}) holds, $\underline{E}^{+}\left(
x_{0}\right) \not=\emptyset $, and assume that $\left( x_{0},\cdots
,x_{0}\right) \not\in \underline{E}^{+}\left( x_{0}\right) $.

Then we are in Case 1'; see (\ref{P24}). Thus, (\ref{P24Beta}) holds.

Consequently,

\begin{itemize}
\item[\refstepcounter{equation}\text{(\theequation)}\label{P66}] $\underset{%
\left( x_{1},\cdots ,x_{k},z_{1},\cdots ,z_{L}\right) \in E^{++}(x_0)}{%
\lim_{x_{1},\cdots ,x_{k},z_{1},\cdots ,z_{L}\rightarrow x_{0}}}Q\left(
x_{0},P_{0},x_{1},J_{x_{1}}^{\left( \bar{m}\right) }F,\cdots
,x_{k},J_{x_{k}}^{\left( \bar{m}\right) }F,z_{1},\cdots ,z_{L}\right) =0$ if
and only if 
\begin{equation*}
\underset{\left(x_{1},\cdots ,x_{k},z_{1},\cdots ,z_{L}\right) \in
E^{++}(x_0)}{\lim_{x_{1},\cdots ,x_{k},z_{1},\cdots ,z_{L}\rightarrow x_{0}}}%
Q\left( x_{0},P_{0},x_{1},\pi _{x_{1}}^{\bar{\bar{m}}\rightarrow \bar{m}%
}J_{x_0}^{\left( \bar{\bar{m}}\right) }F,\cdots ,x_{k},\pi _{x_{k}}^{\bar{%
\bar{m}}\rightarrow \bar{m}}J_{x_0}^{\left( \bar{\bar{m}}\right)
}F,z_{1},\cdots ,z_{L}\right) =0\text{.}
\end{equation*}

\item[\refstepcounter{equation}\text{(\theequation)}\label{P67}] We define 
\begin{equation*}
H^{\prime }\left( x_{0}\right) \subset \mathcal{P}^{\left( m\right) }\left( 
\mathbb{R}^{n},\mathbb{R}^{D}\right) \oplus \mathcal{P}^{\left( \bar{\bar{m}}%
\right) }\left( \mathbb{R}^{n},\mathbb{R}^{I}\right)
\end{equation*}%
to consist of all $\left( P_{0},P\right) $ such that 
\begin{equation*}
\underset{\left( x_{1},\cdots ,x_{k},z_{1},\cdots ,z_{L}\right) \in
E^{++}(x_{0})}{\lim_{x_{1},\cdots ,x_{k},z_{1},\cdots ,z_{L}\rightarrow
x_{0}}}Q\left( x_{0},P_{0},x_{1},\pi _{x_{1}}^{\bar{\bar{m}}\rightarrow \bar{%
m}}P,\cdots ,x_{k},\pi _{x_{k}}^{\bar{\bar{m}}\rightarrow \bar{m}%
}P,z_{1},\cdots ,z_{L}\right) =0\text{.}
\end{equation*}
\end{itemize}

Then, $H^{\prime }\left( x_{0}\right) \subset \mathcal{P}^{\left( m\right)
}\left( \mathbb{R}^{n},\mathbb{R}^{D}\right) \oplus \mathcal{P}^{\left( \bar{%
\bar{m}}\right) }\left( \mathbb{R}^{n},\mathbb{R}^{I}\right) $ is a vector
space depending semialgebraically on $x_{0}$; see \eqref{PSQF6}. Moreover, (%
\ref{P66}) now tells us the following.

\begin{itemize}
\item[\refstepcounter{equation}\text{(\theequation)}\label{P68}] Suppose (%
\ref{P59}) holds, $\underline{E}\left( x_{0}\right) \not=\emptyset $, but $%
\left( x_{0},\cdots ,x_{0}\right) \not\in \underline{E}\left( x_{0}\right) $%
. Then 
\begin{equation*}
\underset{\left(x_{1},\cdots ,x_{k},z_{1},\cdots ,z_{L}\right) \in \hat{E}%
^{++}(x_0)}{\lim_{x_{1},\cdots ,x_{k},z_{1},\cdots ,z_{L}\rightarrow x_{0}}}%
Q\left( x_{0},P_{0},x_{1},J_{{x}_{1}}^{\left( \bar{m}\right) }F,\cdots
,x_{k},J_{{x}_{k}}^{\left( \bar{m}\right) }F,z_{1},\cdots ,z_{L}\right) =0
\end{equation*}%
if and only if 
\begin{equation*}
\left( P_{0},J_{x_{0}}^{\left( \bar{\bar{m}}\right) }F\right) \in H^{\prime
}\left( x_{0}\right) \text{.}
\end{equation*}
\end{itemize}

Next, suppose (\ref{P59}) holds and $\underline{E}\left( x_{0}\right)
=\emptyset $. Then we are in Case 2 above; see (\ref{P18})$,\cdots $,(\ref%
{P24}).

In particular, (\ref{P24}) holds for $\left(x_{1},\cdots ,x_{k},z_{1},\cdots
,z_{L}\right) \in E^{++}(x_0)$; hence (\ref{P22}) and (\ref{P23}) hold.
Therefore, as in our discussion of \eqref{P66}$\cdots$\eqref{P68}, we learn
from (\ref{P67}) that 
\begin{equation*}
\underset{\left(x_{1},\cdots ,x_{k},z_{1},\cdots ,z_{L}\right) \in
E^{++}(x_0)}{\lim_{x_{1},\cdots ,x_{k},z_{1},\cdots ,z_{L}\rightarrow x_{0}}}%
Q\left( x_{0},P_{0},x_{1},J_{x_{1}}^{\left( \bar{m}\right) }F,\cdots
,x_{k},J_{x_{k}}^{\left( \bar{m}\right) }F,z_{1},\cdots ,z_{L}\right) =0
\end{equation*}%
if and only if 
\begin{equation*}
\left( P_{0},J_{x_{0}}^{\left( \bar{\bar{m}}\right) }F\right) \in H^{\prime
}\left( x_{0}\right) \text{.}
\end{equation*}%
Together with (\ref{P68}), this tells us the following.

\begin{itemize}
\item[\refstepcounter{equation}\text{(\theequation)}\label{P69}] Suppose (%
\ref{P59}) holds and $\left( x_{0},\cdots ,x_{0}\right) \not\in \underline{E}%
^{+}\left( x_{0}\right) $. Then 
\begin{equation*}
\underset{\left(x_{1},\cdots ,x_{k},z_{1},\cdots ,z_{L}\right) \in
E^{++}(x_0)}{\lim_{x_{1},\cdots ,x_{k},z_{1},\cdots ,z_{L}\rightarrow x_{0}}}%
Q\left( x_{0},P_{0},x_{1},J_{x_{1}}^{\left( \bar{m}\right) }F,\cdots
,x_{k},J_{x_{k}}^{\left( \bar{m}\right) }F,z_{1},\cdots ,z_{L}\right) =0
\end{equation*}%
if and only if 
\begin{equation*}
\left( P_{0},J_{x_{0}}^{\left( \bar{\bar{m}}\right) }F\right) \in H^{\prime
}\left( x_{0}\right) \text{.}
\end{equation*}
\end{itemize}

Next, exploiting (\ref{P1}), we apply our inductive hypothesis (Propositions %
\ref{SAQF1} and \ref{SAQF2} hold for $\dim E^{+}<\Delta $) to 
\begin{equation*}
Q\left( x_{0},P_{0},x_{1},P_{1},\cdots ,x_{k},P_{k},z_{1},\cdots
,z_{L}\right)
\end{equation*}%
restricted to $\left\{ \left( x_{1},\cdots ,x_{k},z_{1},\cdots ,z_{L}\right)
\in E^{++}(x_{0}):\left( x_{1},\cdots ,x_{k}\right) \in \underline{E}%
^{+}(x_{0})\right\} $. From Proposition \ref{SAQF2} applied to this case, we
obtain an integer $\underline{\underline{\hat{m}}}\geq \bar{m}$ and a computable family of vector spaces 
\begin{equation*}
\underline{H}^{\lim }\left( x_{0}\right) \subset \mathcal{P}^{\left(
m\right) }\left( \mathbb{R}^{n},\mathbb{R}^{D}\right) \oplus \mathcal{P}%
^{\left( \underline{\underline{\hat{m}}}\right) }\left( \mathbb{R}^{n},%
\mathbb{R}^{I}\right)
\end{equation*}%
depending semialgebraically on $x_{0}\in E$, such that the following holds.

\begin{itemize}
\item[\refstepcounter{equation}\text{(\theequation)}\label{P70}] Let $%
x_{0}\in E$, $P_{0}\in \mathcal{P}^{\left( m\right) }\left( \mathbb{R}^{n},%
\mathbb{R}^{D}\right) $, $F\in C_{0}^{\infty }\left( \mathbb{R}^{n},\mathbb{R%
}^{I}\right) $. Suppose that $\underline{E}^{+}(x_{0})\not=\emptyset $
(i.e., $x_{0}\in \underline{E}$), and 
\begin{equation*}
\sup \left\{ 
\begin{array}{c}
Q\left( x_{0},P_{0},x_{1},J_{x_{1}}^{\left( \bar{m}\right) }F,\cdots
,x_{k},J_{x_{k}}^{\left( \bar{m}\right) }F,z_{1},\cdots ,z_{L}\right) : \\ 
\left( x_{1},\cdots ,x_{k},z_{1},\cdots ,z_{L}\right) \in E^{++}(x_{0}), \\ 
\left( x_{1},\cdots ,x_{k}\right) \in \underline{E}^{+}(x_{0})%
\end{array}%
\right\} <\infty \text{.}
\end{equation*}%
Then 
\begin{equation*}
\underset{\left( x_{1},\cdots ,x_{k}\right) \in \underline{E}^{+}\left(
x_{0}\right) }{\underset{\left( x_{1},\cdots ,x_{k},z_{1},\cdots
,z_{L}\right) \in E^{++}(x_{0})}{\lim_{x_{1},\cdots ,x_{k},z_{1},\cdots
,z_{L}\rightarrow x_{0}}}}Q\left( x_{0},P_{0},x_{1},J_{x_{1}}^{\left( \bar{m}%
\right) }F,\cdots ,x_{k},J_{x_{k}}^{\left( \bar{m}\right) }F,z_{1},\cdots
,z_{L}\right) =0
\end{equation*}%
if and only if 
\begin{equation*}
\left( P_{0},J_{x_{0}}^{\left( \underline{\underline{\hat{m}}}\right)
}F\right) \in \underline{H}^{\lim }\left( x_{0}\right) \text{.}
\end{equation*}
\end{itemize}

From (\ref{P65}), (\ref{P69}), (\ref{P70}), we obtain the following.

\begin{itemize}
\item[\refstepcounter{equation}\text{(\theequation)}\label{P71}] Let $%
x_{0}\in E$, $P_{0}\in \mathcal{P}^{\left( m\right) }\left( \mathbb{R}^{n},%
\mathbb{R}^{D}\right) $, $F\in C_{0}^{\infty }\left( \mathbb{R}^{n},\mathbb{R%
}^{I}\right) $ be given. Suppose that 
\begin{equation*}
\sup \left\{ 
\begin{array}{c}
Q\left( x_{0},P_{0},x_{1},J_{x_{1}}^{\left( \bar{m}\right) }F,\cdots
,x_{k},J_{x_{k}}^{\left( \bar{m}\right) }F,z_{1},\cdots ,z_{L}\right) : \\ 
\left( x_{1},\cdots ,x_{k},z_{1},\cdots ,z_{L}\right) \in E^{++}(x_{0})%
\end{array}%
\right\} <\infty \text{.}
\end{equation*}%
Then

\begin{itemize}
\item[(a)] If $\left( x_{0},\cdots ,x_{0}\right) \in \underline{E}^{+}\left(
x_{0}\right) $, then 
\begin{equation*}
\underset{\left( x_{1},\cdots ,x_{k},z_{1},\cdots ,z_{L}\right) \in
E^{++}(x_0)}{\lim_{x_{1},\cdots ,x_{k},z_{1},\cdots ,z_{L}\rightarrow x_{0}}}%
Q\left( x_{0},P_{0},x_{1},J_{x_{1}}^{\left( \bar{m}\right) }F,\cdots
,x_{k},J_{x_{k}}^{\left( \bar{m}\right) }F,z_{1},\cdots ,z_{L}\right) =0
\end{equation*}%
if and only if 
\begin{equation*}
\left( P_{0},J_{x_{0}}^{\left( \hat{m}^{+}\right) }F\right) \in \check{H}%
^{\lim }\left( x_{0}\right)
\end{equation*}%
and 
\begin{equation*}
\left( P_{0},J_{x_{0}}^{\left( \underline{\underline{\hat{m}}}\right)
}F\right) \in \underline{H}^{\lim }\left( x_{0}\right) \text{.}
\end{equation*}

\item[(b)] If $\left( x_{0},\cdots ,x_{0}\right) \not\in \underline{E}%
^{+}\left( x_{0}\right) $, then 
\begin{equation*}
\underset{\left(x_{1},\cdots ,x_{k},z_{1},\cdots ,z_{L}\right) \in
E^{++}(x_0)}{\lim_{x_{1},\cdots ,x_{k},z_{1},\cdots ,z_{L}\rightarrow x_{0}}}%
Q\left( x_{0},P_{0},x_{1},J_{x_{1}}^{\left( \bar{m}\right) }F,\cdots
,x_{k},J_{x_{k}}^{\left( \bar{m}\right) }F,z_{1},\cdots ,z_{L}\right) =0
\end{equation*}%
if and only if 
\begin{equation*}
\left( P_{0},J_{x_{0}}^{\left( \bar{\bar{m}}\right) }F\right) \in H^{\prime
}\left( x_{0}\right) \text{.}
\end{equation*}
\end{itemize}
\end{itemize}

From (\ref{P71}), we obtain at once the conclusion of Proposition \ref{SAQF2}
for the data $E$, $E^{+}$, $E^{++}$, $Q$, $A$.

This concludes our inductive proof of Propositions \ref{SAQF1} and \ref%
{SAQF2}.
\end{proof}

\section{Bundles and Glaeser Refinements}

\subsection{Notation, Definitions, Preliminaries\label%
{Notation-Definitions-Prelim}}

Let $Q\in \mathcal{P}^{(m)}\left( \mathbb{R}^{n},\mathbb{R}^{D}\right) ;$
say $Q=\left( Q_{1},\cdots ,Q_{D}\right) $ with each $Q_{i}\in \mathcal{P}%
^{(m)}\left( \mathbb{R}^{n},\mathbb{R}\right) $. Let $P\in \mathcal{P}%
^{(m)}\left( \mathbb{R}^{n},\mathbb{R}\right) $.

For $x\in \mathbb{R}^{n}$, we define $P\odot _{x}Q:=\left( P\odot
_{x}Q_{1},\cdots ,P\odot _{x}Q_{D}\right) $, where $P\odot
_{x}Q_{i}=J_{x}^{\left( m\right) }\left( PQ_{i}\right) $ is the product of $%
P $ and $Q_{i}$ as $m$-jets at $x$. The above multiplication $\odot _{x}$
makes $\mathscr{R}_{x}^{m}:=\mathcal{P}^{(m)}\left( \mathbb{R}^{n}\right) $
into a ring, and it also makes $\mathcal{P}^{(m)}\left( \mathbb{R}^{n},%
\mathbb{R}^{D}\right) $ into an $\mathscr{R}_{x}^{m}$-module.

Let $E\subset \mathbb{R}^{n}$ be compact. Fix $m\geq 0$, $D\geq 1$. A 
\underline{bundle} over $E$ is a family $\mathscr{H}=\left( H_{x}\right)
_{x\in E}$ parameterized by points $x\in E$, where, for each $x$, the 
\underline{fiber} $H_{x}$ is either the empty set or else has the form 
\begin{equation*}
H_{x}=f\left( x\right) +I\left( x\right) ,
\end{equation*}%
where $f\left( x\right) \in \mathcal{P}^{(m)}\left( \mathbb{R}^{n},\mathbb{R}%
^{D}\right) $ and $I\left( x\right) \subset \mathcal{P}^{(m)}\left( \mathbb{R%
}^{n},\mathbb{R}^{D}\right) $ is an $\mathscr{R}_{x}^{m}$-submodule.

We call $\mathscr{H}$ \underline{proper} if each of its fibers is non-empty.
Note that compactness of $E$ is part of the definition of a bundle. Let $%
\mathscr{H}=\left( H_{x}\right) _{x\in E}$ be a bundle, and let $F\in
C^{m}\left( \mathbb{R}^{n},\mathbb{R}^{D}\right) $. We say that $F$ is a 
\underline{section} of $\mathscr{H}$ if $J_{x}^{\left( m\right) }F\in H_{x} $
for all $x\in E$. Clearly, this cannot happen unless $\mathscr{H}$ is proper.

Next, we define the \underline{strong Glaeser refinement} of a bundle $%
\mathscr{H}=\left( H_{x}\right) _{x\in E}$, denoted by $\mathscr{G}\left( %
\mathscr{H}\right) =\left( \tilde{H}_{x}\right) _{x\in E}$. For any $%
x_{0}\in E$, the fiber $\tilde{H}_{x_{0}}$ consists of all $P_{0}=\left(
P_{0,1},\cdots ,P_{0,D}\right) \in H_{x_{0}}$ satisfying the following
conditions, for a large enough $\bar{k}$ determined by $m,n,D$:

\begin{itemize}
\item[(GR1)] For some finite constant $K$, we have 
\begin{equation}
\min \left\{ 
\begin{array}{c}
\underset{x_{k}\not=x_{k^{\prime }}}{\underset{k,k^{\prime }\in \left\{
0,\cdots ,\bar{k}\right\} }{\sum_{i=1,\cdots ,D}}}\sum_{\left\vert \alpha
\right\vert \leq m}\frac{\left\vert \partial ^{\alpha }\left(
P_{k,i}-P_{k^{\prime },i}\right) \left( x_{k}\right) \right\vert ^{2}}{%
\left\vert x_{k}-x_{k^{\prime }}\right\vert ^{2\left( m-\left\vert \alpha
\right\vert \right) }}: \\ 
P_{1}=\left( P_{1,1},\cdots ,P_{1,D}\right) \in H_{x_{1}},\cdots ,P_{\bar{k}%
}=\left( P_{\bar{k},1},\cdots ,P_{\bar{k},D}\right) \in H_{x_{\bar{k}}}%
\end{array}%
\right\} \leq K  \label{NDP1}
\end{equation}%
for all $x_{1},\cdots ,x_{\bar{k}}\in E$.

\item[(GR2)] The left-hand side of $\left( \ref{NDP1}\right) $ tends to zero
as $\left( x_{1},\cdots ,x_{\bar{k}}\right) \rightarrow \left( x_{0},\cdots
,x_{0}\right) $ in $E^{\bar{k}}$.
\end{itemize}

This definition differs from the usual definition of Glaeser refinement in
previous papers \cite{F1,cf-luli-gafa, fl-jets} on Whitney's problem.

The fiber at $x_{0}$ of the \underline{standard Glaeser refinement} of $%
\mathscr{H}$ is defined to consist of all $P_{0}\in H_{x_{0}}$ that satisfy
(GR2); we do not require (GR1). This notion agrees with the definition of
the Glaeser refinement in the previous papers on Whitney's problems.

Note that the strong Glaeser refinement of $\mathscr{H}$ is a subbundle of
the standard Glaeser refinement of $\mathscr{H}$, which in turn is a
subbundle of $\mathscr{H}$.

(We say that $\mathscr{H}=\left( H_{x}\right) _{x\in E}$ is a \underline{%
subbundle} of $\mathscr{H}^{\prime }=\left( H_{x}^{\prime }\right) _{x\in E}$
if $\mathscr{H}$ and $\mathscr{H}^{\prime }$ are bundles and $H_{x}\subseteq
H_{x}^{\prime }$ for all $x\in E$.)

Moreover, Taylor's theorem implies that any section of $\mathscr{H}$ is
already a section of $\mathscr{G}(\mathscr{H})$, the strong Glaeser
refinement of $\mathscr{H}$.

For any bundle $\mathscr{H}$, and for any integer $l\geq 0$, we define the $%
l^{th}$ iterated (strong) Glaeser refinement of $\mathscr{H}$ by the
following induction:

$\mathscr{G}^{\left( 0\right) }\mathscr{H}=\mathscr{H};\mathscr{G}^{\left(
l+1\right) }\mathscr{H}=\mathscr{G}\left( \mathscr{G}^{\left( l\right) }%
\mathscr{H}\right) $. In principle, we can compute any given $\mathscr{G}%
^{(l)}(\mathscr{H})$ from $\mathscr{H}$.

By induction on $l$, we see that the sections of $\mathscr{H}$ are the same
as the sections of $\mathscr{G}^{\left( l\right) }\mathscr{H}$.

The following result is therefore immediate from the corresponding assertion
for the standard Glaeser refinement, proven in the papers \cite{F1,fl-jets}
for $m \geq 1$, and in \cite{Feff-Kollar} for $m=0$.

\begin{theorem}
\label{CK-Theorem}There exists $l_{\ast }$, depending only on $m$, $n$, $D$,
for which the following holds. Let $\mathscr{H}$ be a bundle. Then $%
\mathscr{H}$ has a section if and only if $\mathscr{G}^{\left( l_{\ast
}\right) }\mathscr{H}$ is a proper bundle.
\end{theorem}

Let $\mathscr{H}=\left( f\left( x\right) +I\left( x\right) \right) _{x\in E}$
be a proper bundle. Then $\mathscr{H}=\left( \tilde{f}\left( x\right)
+I\left( x\right) \right) _{x\in E}$ whenever $\tilde{f}\left( x\right)
-f\left( x\right) \in I\left( x\right) $ for all $x\in E$. Thus, $\mathscr{H}
$ uniquely determines $I\left( x\right) $, but it determines $f\left(
x\right) $ only modulo $I\left( x\right) $. If $\varphi \in C^{\infty }_0
\left( \mathbb{R}^{n}\right) $, then we define 
\begin{equation*}
\varphi \odot \mathscr{H}=\left( J_{x}^{\left( m\right) }\varphi \odot
_{x}f\left( x\right) +I\left( x\right) \right) _{x\in E}.
\end{equation*}%
Thus, $\varphi \odot \mathscr{H}$ is again a proper bundle. Note that our
definition of $\varphi \odot \mathscr{H}$ is independent of the choice of $%
f\left( x\right) $, i.e., it is unaffected by changing $f\left( x\right) $
to $\tilde{f}\left( x\right) $ when $f\left( x\right) -\tilde{f}\left(
x\right) \in I\left( x\right) $. If $F$ is a section of $\mathscr{H}$, then $%
\varphi F$ is a section of $\varphi \odot \mathscr{H}$.

The operation $\mathscr{H}\mapsto \varphi \odot \mathscr{H}$ is related to
the (strong) Glaeser refinement as follows:

\begin{lemma}
\label{Lemma-NDP1} Let $\mathscr{H}=(H_{x})_{x\in E}$ be a bundle with
strong Glaeser refinement $\tilde{\mathscr{H}}=(\tilde{H}_{x})_{x\in E}$, and let $%
\varphi \in C^{\infty }_0 (\mathbb{R}^{n})$. If $P_{0}\in \tilde{H}_{x_{0}}$%
, then $J_{x_{0}}^{(m)}\varphi \odot _{x_{0}}P_{0}$ belongs to the fiber of $%
\mathscr{G}(\varphi \odot \mathscr{H})$ at $x_{0}$.
\end{lemma}

\begin{proof}
We write $K_{1},K_{2},\cdots $ to denote constants determined by $%
m,n,D,\varphi ,E,\mathscr{H},$ $P_{0}$.

Let $x_{1},\cdots ,x_{\bar{k}}\in E$. Because (GR1) holds for $x_{0},P_{0},%
\mathscr{H}$, there exist $P_{1}\in H_{x_{1}},\cdots ,P_{\bar{k}}\in H_{x_{%
\bar{k}}}$ such that

\begin{itemize}
\item[\refstepcounter{equation}\text{(\theequation)}\label{X1}] $%
|\partial^\alpha (P_i -P_j)(x_j)| \leq K_1 |x_i -x_j|^{m- |\alpha|}$ for $%
|\alpha| \leq m, i,j = 0, \cdots, \bar{k}$.
\end{itemize}

(We adopt the convention that $0^0= 0$ to deal with the degenerate case $%
|\alpha| = m, x_i = x_j$.)

Taking $j=0$, we see that

\begin{itemize}
\item[\refstepcounter{equation}\text{(\theequation)}\label{X2}] {$%
|\partial^\alpha P_i(x_0)| \leq K_2$ for $|\alpha| \leq m, i =0, \cdots, 
\bar{k}$.}
\end{itemize}

By expanding $P_i(y)$ in powers of $y-x_0$, we deduce from \eqref{X2}

\begin{itemize}
\item[\refstepcounter{equation}\text{(\theequation)}\label{X3}] {$\vert
\partial^{\alpha} P_i(y) \vert \leq K_3$ for $|\alpha|\leq m, |y-x_0| \leq 
\text{diameter} (E), i =0, \cdots, \bar{k}$.}
\end{itemize}

In particular,

\begin{itemize}
\item[\refstepcounter{equation}\text{(\theequation)}\label{X4}] {$\vert
\partial^{\alpha} P_i(x_j) \vert \leq K_3$ for $|\alpha|\leq m, i, j =0,
\cdots, \bar{k}$.}
\end{itemize}

For $k =0, \cdots, \bar{k}$, let

\begin{itemize}
\item[\refstepcounter{equation}\text{(\theequation)}\label{X4a}] {$%
P_k^{\#}=J_{x_k}^{(m)}\varphi \odot_{x_k} P_k$. Thus,}
\end{itemize}

\begin{itemize}
\item[\refstepcounter{equation}\text{(\theequation)}\label{X5}] {$P_k^{\#}$
belongs to the fiber of $\varphi \odot \mathscr{H}$ at $x_k$, for each $k =
0, \cdots, \bar{k}$.}
\end{itemize}

We estimate the derivatives of $P_{k}^{\#}-P_{k^{\prime }}^{\#}$. To do so,
we write 
\begin{eqnarray}
P_{k}^{\#}-P_{k^{\prime }}^{\#} &=&\left[ J_{x_{k}}^{\left( m\right)
}\varphi \odot _{x_{k}}\left( P_{k}-P_{k^{\prime }}\right) \right] +\left[
\left( J_{x_{k}}^{\left( m\right) }\varphi -J_{x_{k^{\prime }}}^{\left(
m\right) }\varphi \right) \odot _{x_{k}}P_{k^{\prime }}\right]  \label{X6} \\
&&+\left[ J_{x_{k^{\prime }}}^{\left( m\right) }\varphi \odot
_{x_{k}}P_{k^{\prime }}-J_{x_{k^{\prime }}}^{\left( m\right) }\varphi \odot
_{x_{k^{\prime }}}\odot _{x_{k^{\prime }}}P_{k^{\prime }}\right]  \notag \\
&\equiv &R_{kk^{\prime }}^{\left( 1\right) }+R_{kk^{\prime }}^{\left(
2\right) }+R_{kk^{\prime }}^{\left( 3\right) }\text{.}  \notag
\end{eqnarray}%
From \eqref{X1}, we have

\begin{itemize}
\item[\refstepcounter{equation}\text{(\theequation)}\label{X7}] {$\left\vert
\partial ^{\alpha }R_{kk^{\prime }}^{\left( 1\right) }\left( x_{k}\right)
\right\vert \leq K_{4}\left\vert x_{k}-x_{k^{\prime }}\right\vert
^{m-\left\vert \alpha \right\vert }$ for $\left\vert \alpha \right\vert \leq
m$.}
\end{itemize}

From (\ref{X4}) and Taylor's theorem for $\varphi $, we have

\begin{itemize}
\item[\refstepcounter{equation}\text{(\theequation)}\label{X8}] {$\left\vert
\partial ^{\alpha }R_{kk^{\prime }}^{\left( 2\right) }\left( x_{k}\right)
\right\vert \leq K_{5}\left\vert x_{k}-x_{k^{\prime }}\right\vert
^{m-\left\vert \alpha \right\vert }$ for $\left\vert \alpha \right\vert \leq
m$.}
\end{itemize}

Moreover,

\begin{itemize}
\item[\refstepcounter{equation}\text{(\theequation)}\label{X9}] {$%
R_{kk^{\prime }}^{\left( 3\right) }=J_{x_{k}}^{\left( m\right) }G_{k^{\prime
}}-J_{x_{k^{\prime }}}^{\left( m\right) }G_{k^{\prime }}$, where $%
G_{k^{\prime }}=\left( J_{x_{k^{\prime }}}^{\left( m\right) }\varphi \right)
\cdot P_{k^{\prime }}$.}
\end{itemize}

Because $J_{x_{k^{\prime }}}^{\left( m\right) }\varphi $ and $P_{k^{\prime
}} $ are polynomials of degree at most $m$, (\ref{X3}) yields the estimate

\begin{itemize}
\item[\refstepcounter{equation}\text{(\theequation)}\label{X9a}] {$%
\left\vert \partial ^{\alpha }G_{k^{\prime }}\left( y\right) \right\vert
\leq K_{6}$ for $\left\vert \alpha \right\vert =m+1$ and $\left\vert
y-x\right\vert \leq $ diameter $\left( E\right) $.}
\end{itemize}

Consequently, (\ref{X9}) and Taylor's theorem tell us that

\begin{itemize}
\item[\refstepcounter{equation}\text{(\theequation)}\label{X10}] {$%
\left\vert \partial ^{\alpha }R_{kk^{\prime }}^{\left( 3\right) }\left(
x\right) \right\vert \leq K_{7}\left\vert x_{k}-x_{k^{\prime }}\right\vert
^{m+1- |\alpha|}\leq K_{8}\left\vert x_{k}-x_{k^{\prime }}\right\vert ^{m-
|\alpha|}$ for $|\alpha|\leq m, k,k^{\prime }=0,\cdots, \bar{k}$.}
\end{itemize}

Putting (\ref{X7}), (\ref{X8}), (\ref{X10}) into (\ref{X6}), we learn that

\begin{itemize}
\item[\refstepcounter{equation}\text{(\theequation)}\label{X11}] {$%
\left\vert \partial ^{\alpha }\left( P_{k}^{\#}-P_{k^{\prime }}^{\#}\right)
\left( x_{k}\right) \right\vert \leq K_{9}\left\vert x_{k}-x_{k^{\prime
}}\right\vert ^{m-\left\vert \alpha \right\vert }$ for $\left\vert \alpha
\right\vert \leq m$, $k,k^{\prime }=0,\cdots ,\bar{k}$.}
\end{itemize}

Together with (\ref{X5}), this tells us that

\begin{itemize}
\item[\refstepcounter{equation}\text{(\theequation)}\label{X12}] {$%
P_{0}^{\#}=J_{x_{0}}^{\left( m\right) }\varphi \odot _{x_{0}}P_{0}$
satisfies (GR1) for the point $x_{0}$ and the bundle $\varphi \odot %
\mathscr{H}$.}
\end{itemize}

Similarly, we establish (GR2) for $P_{0}^{\#}$, $x_{0}$, $\varphi \odot %
\mathscr{H}$.

We sketch the argument. Let $0<\varepsilon <1$ be given. Let $\delta >0$ be
as in (GR2) for $P_{0}$, $x_{0}$, $\mathscr{H}$, and let $x_{1},\cdots ,x_{%
\bar{k}}\in B\left( x_{0},\hat{\delta}\right) \cap E$, where $\hat{\delta}%
\in \left( 0,\delta \right) $ will be picked below.

Then there exist $P_{1}\in H_{x_{1}},\cdots ,P_{\bar{k}}\in H_{x_{\bar{k}}}$
such that

\begin{itemize}
\item[\refstepcounter{equation}\text{(\theequation)}\label{X13}] {$%
\left\vert \partial ^{\alpha }\left( P_{i}-P_{j}\right) \left( x_{j}\right)
\right\vert \leq \varepsilon \left\vert x_{i}-x_{j}\right\vert
^{m-\left\vert \alpha \right\vert }$ for $\left\vert \alpha \right\vert \leq
m$, $i,j=0,\cdots ,\bar{k}$.}
\end{itemize}

Thus, (\ref{X2})$\cdots $(\ref{X4}) hold as before. Defining $P_{k}^{\#}$ $%
\left( k=0,1,\cdots ,\bar{k}\right) $ by (\ref{X4a}), we again have (\ref{X5}%
) and (\ref{X6}).

Estimate (\ref{X13}) implies that

\begin{itemize}
\item[\refstepcounter{equation}\text{(\theequation)}\label{X14}] {$%
\left\vert \partial ^{\alpha }R_{kk^{\prime }}^{\left( 1\right) }\left(
x_{k}\right) \right\vert \leq K_{10}\varepsilon \left\vert
x_{k}-x_{k^{\prime }}\right\vert ^{m-\left\vert \alpha \right\vert }$ for $%
\left\vert \alpha \right\vert \leq m,$ $k,k^{\prime }=0,\cdots ,\bar{k}$.}
\end{itemize}

Estimate (\ref{X4}) and Taylor's theorem for $\varphi $ yield

\begin{itemize}
\item[\refstepcounter{equation}\text{(\theequation)}\label{X15}] {$%
\left\vert \partial ^{\alpha}R_{kk^{\prime }}^{\left( 2\right) }\left(
x_{k}\right) \right\vert \leq K_{11}\left\vert x_{k}-x_{k^{\prime
}}\right\vert ^{\left( m+1\right) -\left\vert \alpha \right\vert }$ for $%
\left\vert \alpha \right\vert \leq m,k,k^{\prime }=0,\cdots ,\bar{k}$.}
\end{itemize}

Moreover, (\ref{X9}) and (\ref{X9a}) hold, from which we have

\begin{itemize}
\item[\refstepcounter{equation}\text{(\theequation)}\label{X16}] {$%
\left\vert \partial ^{\alpha }R_{kk^{\prime }}^{\left( 3\right) }\left(
x_{k}\right) \right\vert \leq K_{12}\left\vert x_{k}-x_{k^{\prime
}}\right\vert ^{\left( m+1\right) -\left\vert \alpha \right\vert }$ for $%
\left\vert \alpha \right\vert \leq m$, $k,k^{\prime }=0,\cdots ,\bar{k}$.}
\end{itemize}

Taking $\hat{\delta}\leq \varepsilon $ so that $\left\vert
x_{k}-x_{k^{\prime }}\right\vert ^{\left( m+1\right) -\left\vert \alpha
\right\vert }\leq \varepsilon \left\vert x_{k}-x_{k^{\prime }}\right\vert
^{m-\left\vert \alpha \right\vert }$, we deduce from (\ref{X14}), (\ref{X15}%
), (\ref{X16}) and (\ref{X6}) that

\begin{itemize}
\item[\refstepcounter{equation}\text{(\theequation)}\label{X17}] {$%
\left\vert \partial ^{\alpha }\left( P_{k}^{\#}-P_{k^{\prime }}^{\#}\right)
\left( x_{k}\right) \right\vert \leq K_{13}\varepsilon \left\vert
x_{k}-x_{k^{\prime }}\right\vert ^{m-\left\vert \alpha \right\vert }$ for $%
\left\vert \alpha \right\vert \leq m$, $k,k^{\prime }=0,\cdots ,\bar{k}$.}
\end{itemize}

From (\ref{X4a}) and (\ref{X17}), we see that

\begin{itemize}
\item[\refstepcounter{equation}\text{(\theequation)}\label{X18}] {$%
P_{0}^{\#}=J_{x_{0}}^{\left( m\right) }\varphi \odot _{x_{0}}P_{0}$
satisfies (GR2) for the point $x_{0}$ and the bundle $\varphi \odot %
\mathscr{H}.$}
\end{itemize}

Thanks to (\ref{X12}) and (\ref{X18}), we know that $J_{x_{0}}^{\left(
m\right) }\varphi \odot _{x_{0}}P_{0}$ belongs to the fiber at $x_{0}$ of $%
\mathscr{G}\left( \varphi \odot \mathscr{H}\right) $.

The proof of Lemma \ref{Lemma-NDP1} is complete.
\end{proof}

A \underline{homogeneous bundle} is a bundle of the form $\mathscr{H}=\left(
I\left( x\right) \right) _{x\in E}$ with each $I\left( x\right) \subset 
\mathcal{P}^{(m)}\left( \mathbb{R}^{n},\mathbb{R}^{D}\right) $ an $%
\mathscr{R}_{x}^{m}$-submodule. (That is, we can take $f\equiv 0$ in $%
\mathscr{H}=\left( f\left( x\right) +I\left( x\right) \right) _{x\in E}$.)

The (strong) Glaeser refinement of a homogenous bundle is again a
homogeneous bundle.

\begin{itemize}
\item[\refstepcounter{equation}\text{(\theequation)}\label{NDP7}] Let $%
\mathscr{H}=\left( f\left( x\right) +I\left( x\right) \right) _{x\in E}$ be
a proper bundle, and let $\left( \tilde{I}\left( x\right) \right) _{x\in E}$
be the (strong) Glaeser refinement of $\left( I\left( x\right) \right)
_{x\in E}$. If for each $x\in E$, $g\left( x\right) $ belongs to the fiber
at $x$ of the strong Glaeser refinement $\mathscr{G}\left( \mathscr{H}%
\right) $, then it follows that $\mathscr{G}\left( \mathscr{H}\right)
=\left( g\left( x\right) +\tilde{I}\left( x\right) \right) _{x\in E}$.
\end{itemize}

Therefore, Lemma \ref{Lemma-NDP1} implies the following.

\begin{corollary}
\label{corollary-NDP1}Let $\mathscr{H}$ be a bundle, and suppose $\mathscr{G}%
\left( \mathscr{H}\right) $ is proper. Then for any $\varphi \in C^{\infty
}_0\left( \mathbb{R}^{n}\right) $, we have $\mathscr{G}\left( \varphi \odot %
\mathscr{H}\right) =\varphi \odot \mathscr{G}\left( \mathscr{H}\right) $.
\end{corollary}

\begin{proof}
For each $x\in E$, pick $f\left( x\right) $ in the fiber of $\mathscr{G}%
\left( \mathscr{H}\right) $ at $x$. In particular, $f\left( x\right) $
belongs to the fiber of $\mathscr{H}$ at $x$, so $\mathscr{H}=\left( f\left(
x\right) +I\left( x\right) \right) _{x\in E}$ for a family of $\mathscr{R}%
_{x}^{m}$-submodules $I\left( x\right) \subset \mathcal{P}^{(m)}\left( 
\mathbb{R}^{n},\mathbb{R}^{D}\right) $. Let $\left( \tilde{I}\left( x\right)
\right) _{x\in E}$ be the strong Glaeser refinement of $\left( I\left(
x\right) \right) _{x\in E}$. We have $\mathscr{G}\left( \mathscr{H}\right)
=\left( f\left( x\right) +\tilde{I}\left( x\right) \right) _{x\in E}$ by %
\eqref{NDP7}, hence $\varphi \odot \mathscr{G}\left( \mathscr{H}\right)
=\left( J_{x}^{(m)}\varphi \odot _{x}f\left( x\right) +\tilde{I}\left(
x\right) \right) _{x\in E};$ also $\varphi \odot \mathscr{H}=\left(
J_{x}^{(m)}\varphi \odot _{x}f\left( x\right) +I\left( x\right) \right)
_{x\in E}$, hence if $\mathscr{G}\left( \varphi \odot \mathscr{H}\right) $
is a proper bundle, then (by (\ref{NDP7})) the fiber at every $x \in E$ of $%
\mathscr{G}(\varphi \odot \mathscr{H})$ has the form 
\begin{equation}
g\left( x\right) +\tilde{I}\left( x\right)  \label{NDP8}
\end{equation}%
for some $g\left( x\right) $. According to the preceding lemma, $\mathscr{G}%
(\varphi \odot \mathscr{H})$ is a proper bundle; and, with $g$ as in %
\eqref{NDP8}, we have $J_{x}^{(m)}\varphi \odot _{x}f\left( x\right) \in
g\left( x\right) +\tilde{I}\left( x\right) $. Therefore, the fiber at $x$ of 
$\mathscr{G}\left( \varphi \odot \mathscr{H}\right) $ is equal to $%
J_{x}^{(m)}\varphi \odot _{x}f\left( x\right) +\tilde{I}\left( x\right) $.
That's also the fiber at $x$ of $\varphi \odot \mathscr{G}\left( \mathscr{H}%
\right) $, as we saw above. Thus, $\mathscr{G}\left( \varphi \odot %
\mathscr{H}\right) =\varphi \odot \mathscr{G}\left( \mathscr{H}\right) $, as
claimed.
\end{proof}

\section{Bundles Determined by Smooth Functions\label{BDSF}}

\textbf{Setup.} Let $E\subset \mathbb{R}^{n}$ be compact, semialgebraic. For 
$x\in E$, let $I\left( x\right) \subset \mathcal{P}^{\left( m\right) }\left( 
\mathbb{R}^{n},\mathbb{R}^{D}\right) $ be an $\mathscr{R}_{x}^{m}$%
-submodule, depending semialgebraically on $x$. For $l\geq 0$, let $\left(
I^{\left( l\right) }\left( x\right) \right) _{x\in E}$ denote the $l^{th}$
(strong) Glaeser refinement of the homogeneous bundle $\left( I\left(
x\right) \right) _{x\in E}$. Thus, $I^{\left( 0\right) }\left( x\right)
=I\left( x\right) $, and $\left( I^{\left( l+1\right) }\left( x\right)
\right) _{x\in E}=\mathscr{G} \left( \left( I^{\left( l\right) }\left(
x\right) \right)_{x\in E}\right) $. Note that $I^{(l)}(x)$ depends
semialgebraically on $x$, by an obvious induction on $l$. Let $T\left(
x\right) :\mathcal{P}^{(\bar{m})}\left( \mathbb{R}^{n},\mathbb{R}^{M}\right)
\rightarrow \mathcal{P}^{(m)}\left( \mathbb{R}^{n},\mathbb{R}^{D}\right) $
be a linear map, depending semialgebraically on $x\in E$. Here, $\bar{m}\geq
m$. For $f \in C^{\infty }_0\left( \mathbb{R}^{n},\mathbb{R}^{M}\right) $,
we write $\mathscr{H}_{f}$ to denote the bundle 
\begin{equation}
\mathscr{H}_{f}=\left( T\left( x\right) J_{x}^{\left( \bar{m}\right)
}f+I\left( x\right) \right) _{x\in E}.  \label{B0}
\end{equation}%
We make the following

\underline{Assumption}:

\begin{equation}
\mathscr{H}_{\varphi f}=\varphi \odot \mathscr{H}_{f}\text{ for any }%
C^{\infty }_0\text{ scalar-valued function }\varphi \text{.}  \label{B1}
\end{equation}%
For $l\geq 0$, let $\mathscr{H}_{f}^{\left( l\right) }$ denote the $l^{th}$
(strong) Glaeser refinement of $\mathscr{H}_{f}$. Thus, $\mathscr{H}%
_{f}^{\left( 0\right) }=\mathscr{H}_{f}$ and $\mathscr{H}_{f}^{\left(
l+1\right) }=\mathscr{G}\left( \mathscr{H}_{f}^{\left( l\right) }\right) $.

Under the above assumptions, we will prove the following result:

\begin{lemma}[Main Lemma on $\mathscr{H}_{f}$]
\label{Main-Lemma-onHf}For each $l\geq 0$, there exist an integer $\bar{\bar{%
m}}\geq \bar{m}$, a finite list of linear differential operators $L_{\nu }$ $%
\left( \nu =1,\cdots ,\nu _{\max }\right) $, and a linear map 
\begin{equation*}
T^{\left( l\right) }\left( x\right) :\mathcal{P}^{(\bar{\bar{m}})}\left( 
\mathbb{R}^{n},\mathbb{R}^{M}\right) \rightarrow \mathcal{P}^{(m)}\left( 
\mathbb{R}^{n},\mathbb{R}^{D}\right) ,
\end{equation*}%
with the following properties.

\begin{itemize}
\item Each $L_{\nu }$ has semialgebraic coefficients, and maps functions in $%
C^{\infty }\left( \mathbb{R}^{n},\mathbb{R}^{M}\right) $ to scalar-valued
functions on $\mathbb{R}^{n}$.

\item Each $L_{\nu }$ has order at most $\bar{\bar{m}}$.

\item $T^{\left( l\right) }\left( x\right) $ depends semialgebraically on $%
x. $

\item Let $f\in C^{\infty }_0\left( \mathbb{R}^{n},\mathbb{R}^{M}\right) $.
Then $\mathscr{H}_{f}^{\left( l\right) }$ is a proper bundle (i.e., all of
its fibers are non-empty) if and only if $L_{\nu }f=0$ on $E$ for each $\nu
=1,\cdots ,\nu _{\max }$.

\item Let $f\in C^{\infty }_0\left( \mathbb{R}^{n},\mathbb{R}^{M}\right) $.
If $\mathscr{H}_{f}^{\left( l\right) }$ is a proper bundle, then $\mathscr{H}%
_{f}^{\left( l\right) }=\left( T^{\left( l\right) }\left( x\right)
J_{x}^{\left( \bar{\bar{m}}\right) }f+I^{\left( l\right) }\left( x\right)
\right) _{x\in E}$.
\end{itemize}
\end{lemma}

\begin{proof}
We use induction on $l$. For $l=0$, we take $\bar{\bar{m}}=\bar{m}$, $%
\left\{ L_{1},\cdots ,L_{\nu _{\max }}\right\} =$ empty list $\left( \nu
_{\max }=0\right) $, $T^{\left( 0\right) }\left( x\right) =T\left( x\right)
. $ The conclusions of Lemma \ref{Main-Lemma-onHf} (for $l=0)$ are immediate
from our assumptions in Setup.

For the induction step, we fix $l\geq 1$ and assume that Lemma \ref%
{Main-Lemma-onHf} on $\mathscr{H}_{f}$ holds with $l$ replaced by $l-1$. We
then prove Lemma \ref{Main-Lemma-onHf} on $\mathscr{H}_{f}$ for the given $l 
$.

By our inductive assumption, there exist $\hat{m}\geq m;$ linear
differential operators 
\begin{equation*}
L_{\nu }^{old} \left( \nu =1,\cdots ,\nu _{\max }^{old}\right)
\end{equation*}
of order at most $\hat{m}$, with semialgebraic coefficients; and a linear
map $T^{old}\left( x\right) $ mapping $\mathcal{P}^{(\hat{m})}\left( \mathbb{%
R}^{n},\mathbb{R}^{M}\right) \rightarrow \mathcal{P}^{(m)}\left( \mathbb{R}%
^{n},\mathbb{R}^{D}\right) $ such that the following hold.

\begin{itemize}
\item[\refstepcounter{equation}\text{(\theequation)}\label{B2}] Let $f\in
C^{\infty }_0\left( \mathbb{R}^{n},\mathbb{R}^{M}\right) $. Then $\mathscr{H}%
_{f}^{\left( l-1\right) }$ is a proper bundle if and only if $L_{\nu
}^{old}f=0$ on $E$ for each $\nu =1,\cdots ,\nu^{old}_{\max}$.

\item[\refstepcounter{equation}\text{(\theequation)}\label{B3}] Let $f\in
C^{\infty }_0\left( \mathbb{R}^{n},\mathbb{R}^{M}\right) $. If $\mathscr{H}%
_{f}^{\left( l-1\right) }$ is a proper bundle, then $\mathscr{H}_{f}^{\left(
l-1\right) }=\left( T^{old}\left( x\right) J_{x}^{\left( \hat{m}\right)
}f+I^{\left( l-1\right) }\left( x\right) \right) _{x\in E}$.
\end{itemize}

Suppose $f\in C^{\infty }_0 \left( \mathbb{R}^{n},\mathbb{R}^{M}\right) $, $%
x_{0}\in E$, $P_{0}\in \mathcal{P}^{(m)}\left( \mathbb{R}^{n},\mathbb{R}%
^{D}\right) $ are given; and assume that $\mathscr{H}_{f}^{\left( l-1\right)
}$ is a proper bundle.

We investigate whether $P_{0}$ belongs to the fiber at $x_{0}$ of $%
\mathscr{H}_{f}^{\left( l\right) }=\mathscr{G}\left( \mathscr{H}_{f}^{\left(
l-1\right) }\right) $. We recall the definition (GR1) and (GR2) of the
(strong) Glaeser refinement in Section \ref{Notation-Definitions-Prelim}.

For $x_{1},\cdots ,x_{\bar{k}}\in E$, and for $P_{1},\cdots ,P_{\bar{k}}\in 
\mathcal{P}^{(\hat{m})}\left( \mathbb{R}^{n},\mathbb{R}^{M}\right) $, we
define $Q\left( x_{0},P_{0},x_{1},P_{1},\cdots ,x_{\bar{k}},P_{\bar{k}%
}\right) $ to be the minimum of the quantity 
\begin{equation*}
\sum_{|\alpha|\leq m}\sum_{i=1}^{D}\sum_{k,k^{\prime }\in \left\{ 0,\cdots ,%
\bar{k}\right\} ,x_{k}\not=x_{k^{\prime }}}\frac{\left\vert \partial
^{\alpha }\left( P_{k,i}^{\#}-P_{k^{\prime },i}^{\#}\right) \left(
x_{k}\right) \right\vert ^{2}}{\left\vert x_{k}-x_{k^{\prime }}\right\vert
^{2\left( m-\left\vert \alpha \right\vert \right) }}
\end{equation*}%
over all $P_{1}^{\#}=\left( P_{1,1}^{\#},\cdots ,P_{1,D}^{\#}\right) ,\cdots
,P_{\bar{k}}^{\#}=\left( P_{\bar{k},1}^{\#},\cdots ,P_{\bar{k}%
,D}^{\#}\right) $ such that 
\begin{equation*}
P_{k}^{\#}-T^{old}\left( x_{k}\right) P_{k}\in I^{\left( l-1\right) }\left(
x_{k}\right) \text{ for each }k \geq 1.
\end{equation*}%
Here, we take $P_0^{\#} = (P_{0,1}^{\#}, \cdots, P_{0,D}^{\#})$ to equal $%
P_0 $.

By linear algebra, $Q\left( x_{0},P_{0},x_{1},P_{1},\cdots ,x_{\bar{k}},P_{%
\bar{k}}\right) $ is a positive semidefinite quadratic form in $P_0,
P_{1},\cdots ,P_{\bar{k}}$, for each fixed $x_{0},x_{1},\cdots ,x_{\bar{k}%
}\in E$. Moreover, since $x\mapsto T^{old}\left( x\right) $ and $x \mapsto
I^{(l-1)}(x)$ are semialgebraic, it follows that 
\begin{equation*}
Q\left( x_{0},P_{0},\cdots ,x_{\bar{k}},P_{\bar{k}}\right)
\end{equation*}%
is a semialgebraic function of $\left( x_{0},P_{0},\cdots ,x_{\bar{k}},P_{%
\bar{k}}\right) $.

If we define $A\left( x_{0},\cdots ,x_{\bar{k}}\right) $ to be the norm of
this quadratic form, i.e., the least $A\left( x_{0},\cdots ,x_{\bar{k}%
}\right) $ such that 
\begin{equation*}
Q\left( x_{0},P_{0},x_{1},P_{1},\cdots ,x_{\bar{k}},P_{\bar{k}}\right) \leq
A\left( x_{0},\cdots ,x_{\bar{k}}\right) \cdot \left[ \sum_{|\alpha |\leq
m}|\partial ^{\alpha }P_{0}(x_{0})|^{2}+\sum_{i=1}^{\bar{k}}\sum_{\left\vert
\alpha \right\vert \leq \hat{m}}\left\vert \partial ^{\alpha }P_{i}\left(
x_{i}\right) \right\vert ^{2}\right] ,
\end{equation*}%
for all $P_{0},P_{1},\cdots ,P_{\bar{k}}$, then $A\left( x_{0},\cdots ,x_{%
\bar{k}}\right) $ is a nonnegative semialgebraic function of $\left(
x_{0},\cdots ,x_{\bar{k}}\right) $.

Therefore, we have all the conditions assumed in the setup in Section \ref%
{quadratic-form-set-up-section}. (The number of $z$'s there is zero.)

From Propositions \ref{SAQF1} and \ref{SAQF2}, we have the following
conclusions.

There exist $\bar{\bar{m}}\geq \hat{m};$ and families of vector subspaces 
\begin{equation*}
H^{\text{bdd}}\left( x_{0},\cdots ,x_{\bar{k}}\right) \mathcal{\subset 
\mathcal{P}}^{(m)}\left( \mathbb{R}^{n},\mathbb{R}^{D}\right) \oplus
\sum_{i=1}^{\bar{k}}\mathcal{P}^{(\bar{\bar{m}})}\left( \mathbb{R}^{n},%
\mathbb{R}^{M}\right)
\end{equation*}%
and 
\begin{equation*}
H^{\lim }\left( x_{0}\right) \mathcal{\subset \mathcal{P}}^{(m)}\left( 
\mathbb{R}^{n},\mathbb{R}^{D}\right) \oplus \mathcal{P}^{(\bar{\bar{m}}%
)}\left( \mathbb{R}^{n},\mathbb{R}^{M}\right) ,
\end{equation*}%
with the following properties.

\begin{itemize}
\item[\refstepcounter{equation}\text{(\theequation)}\label{B4}] $H^{\text{bdd%
}}\left( x_{0},\cdots ,x_{\bar{k}}\right) $ depends semi-algebraically on $%
\left( x_{0},x_{1},\cdots ,x_{\bar{k}}\right) \in E\times \cdots \times E$.

\item[\refstepcounter{equation}\text{(\theequation)}\label{B5}] $H^{\lim
}\left( x_{0}\right) $ depends semi-algebraically on $x_{0}\in E$.

\item[\refstepcounter{equation}\text{(\theequation)}\label{B6}] Let $%
x_{0}\in E$, $P_{0}\in \mathcal{P}^{\left( m\right) }\left( \mathbb{R}^{n},%
\mathbb{R}^{D}\right) $, $f\in C^{\infty }_0\left( \mathbb{R}^{n},\mathbb{R}%
^{M}\right) $. Then 
\begin{equation*}
\left\{ Q\left( x_{0},P_{0},x_{1},J_{x_{1}}^{\left( \hat{m}\right) }f,\cdots
,x_{\bar{k}},J_{x_{\bar{k}}}^{\left( \hat{m}\right) }f\right) :x_{1},\cdots
,x_{\bar{k}}\in E\right\}
\end{equation*}%
is bounded if and only if $\left( P_{0},J_{x_{1}}^{\left( \bar{\bar{m}}%
\right) }f,\cdots ,J_{x_{\bar{k}}}^{\left( \bar{\bar{m}}\right) }f\right)
\in H^{\text{bdd}}\left( x_{0},x_{1},\cdots ,x_{\bar{k}}\right) $ for all $%
x_{1},\cdots ,x_{\bar{k}}\in E$.

\item[\refstepcounter{equation}\text{(\theequation)}\label{B7}] Let $%
x_{0}\in E$, $\mathcal{P}_{0}\in \mathcal{P}^{\left( m\right) }\left( 
\mathbb{R}^{n},\mathbb{R}^{D}\right) $, $f\in C^{\infty }_0\left( \mathbb{R}%
^{n},\mathbb{R}^{M}\right) $. Suppose that 
\begin{equation*}
\left\{ Q\left( x_{0},P_{0},x_{1},J_{x_{1}}^{\left( \hat{m}\right) }f,\cdots
,x_{\bar{k}},J_{x_{\bar{k}}}^{\left( \hat{m}\right) }f\right) :x_{1},\cdots
,x_{\bar{k}}\in E\right\}
\end{equation*}%
is bounded. Then 
\begin{equation*}
\underset{x_{1},\cdots ,x_{\bar{k}}\in E}{\lim_{x_{1},\cdots ,x_{\bar{k}%
}\rightarrow x_{0}}}Q\left( x_{0},P_{0},x_{1},J_{x_{1}}^{\left( \hat{m}%
\right) }f,\cdots ,x_{\bar{k}},J_{x_{\bar{k}}}^{\left( \hat{m}\right)
}f\right) =0
\end{equation*}%
if and only if $\left( P_{0},J_{x_{0}}^{\left( \bar{\bar{m}}\right)
}f\right) \in H^{\lim }\left( x_{0}\right) $.
\end{itemize}

Comparing our definition of $Q(\cdots )$ with (GR1), and (GR2) in the
definition of the (strong) Glaeser refinement, we see the following.

Let $\mathscr{H}^{(l-1)}=\left( T^{old}\left( x\right) J_{x}^{\left( \hat{m}%
\right) }f+I^{\left( l-1\right) }\left( x\right) \right) _{x\in E}$, where $%
f\in C^{\infty }_0\left( \mathbb{R}^{n},\mathbb{R}^{M}\right) $ is given.
Let $x_{0}\in E$, and let $P_{0}$ belong to the fiber of $\mathscr{H}%
^{(l-1)} $ at $x_{0}$. Then (GR1) holds for $\left( x_{0},P_{0}\right) $ and 
$\mathscr{H}^{(l-1)}$, if and only if 
\begin{equation*}
\left\{ Q\left( x_{0},P_{0},x_{1},J_{x_{1}}^{\left( \hat{m}\right) }f,\cdots
,x_{\bar{k}},J_{x_{\bar{k}}}^{\left( \hat{m}\right) }f\right) :x_{1},\cdots
,x_{\bar{k}}\in E\right\}
\end{equation*}%
is bounded. Moreover, (GR2) holds for $\left( x_{0},P_{0}\right) $ and $%
\mathscr{H}^{(l-1)}$, if and only if 
\begin{equation*}
\underset{x_{1},\cdots ,x_{\bar{k}}\in E}{\lim_{x_{1},\cdots ,x_{\bar{k}%
}\rightarrow x_{0}}}Q\left( x_{0},P_{0},x_{1},J_{x_{1}}^{\left( \hat{m}%
\right) }f,\cdots ,x_{\bar{k}},J_{x_{\bar{k}}}^{\left( \hat{m}\right)
}f\right) =0.
\end{equation*}%
Recalling (\ref{B3}), (\ref{B6}), (\ref{B7}), we obtain the following result.

Let $f\in C_{0}^{\infty }\left( \mathbb{R}^{n},\mathbb{R}^{M}\right) $.
Suppose $\mathscr{H}_{f}^{\left( l-1\right) }$ is a proper bundle. Let $%
x_{0}\in E$, and let $P_{0}$ belong to the fiber of $\mathscr{H}_{f}^{\left(
l-1\right) }$ at $x_{0}$. Then $P_{0}$ belongs to the fiber of $\mathscr{H}%
_{f}^{\left( l\right) }=\mathscr{G}\left( \mathscr{H}_{f}^{\left( l-1\right)
}\right) $ at $x_{0}$, if and only if $\left( P_{0},J_{x_{0}}^{\left( \bar{%
\bar{m}}\right) }f\right) \in H^{\lim }\left( x_{0}\right) $ and $\left(
P_{0},J_{x_{1}}^{\left( \bar{\bar{m}}\right) }f,\cdots ,J_{x_{\bar{k}%
}}^{\left( \bar{\bar{m}}\right) }f\right) \in H^{\text{bdd}}\left(
x_{0},x_{1},\cdots ,x_{\bar{k}}\right) $ for all $x_{1},\cdots ,x_{\bar{k}%
}\in E$.

Together with (\ref{B2}) and (\ref{B3}), this immediately implies that there
exists a family of vector spaces

\begin{itemize}
\item[\refstepcounter{equation}\text{(\theequation)}\label{B8}] $%
H^{GR}\left( x_{0},x_{1},\cdots ,x_{\bar{k}}\right) \subset \mathcal{P}%
^{(m)}\left( \mathbb{R}^{n},\mathbb{R}^{D}\right) \oplus \sum_{i=0}^{\bar{k}}%
\mathcal{P}^{(\bar{\bar{m}})}\left( \mathbb{R}^{n},\mathbb{R}^{M}\right) $
depending semialgebraically on

$x_{0},x_{1},\cdots ,x_{\bar{k}}\in E, $
\end{itemize}

for which the following holds.

\begin{itemize}
\item[\refstepcounter{equation}\text{(\theequation)}\label{B9}] Let $f\in
C_{0}^{\infty }\left( \mathbb{R}^{n},\mathbb{R}^{M}\right) $, and suppose $%
\mathscr{H}_{f}^{\left( l-1\right) }$ is a proper bundle. Let $x_{0}\in E$, $%
P_{0}\in \mathcal{P}^{(m)}\left( \mathbb{R}^{n},\mathbb{R}^{D}\right) $.
Then $P_{0}$ belongs to the fiber of $\mathscr{H}_{f}^{\left( l\right) }$ at 
$x_{0}$ if and only if 
\begin{equation*}
\left( P_{0},J_{x_{0}}^{\left( \bar{\bar{m}}\right) }f,J_{x_{1}}^{\left( 
\bar{\bar{m}}\right) }f,\cdots ,J_{x_{\bar{k}}}^{\left( \bar{\bar{m}}\right)
}f\right) \in H^{GR}\left( x_{0},x_{1},\cdots ,x_{\bar{k}}\right)
\end{equation*}%
for all $x_{1},\cdots ,x_{\bar{k}}\in E$.
\end{itemize}

By linear algebra (and basic facts about semialgebraic sets), there exist
finitely many linear functionals 
\begin{equation*}
\left\{ 
\begin{array}{c}
\mu _{\nu }\left( x_{0},\cdots ,x_{\bar{k}}\right) :\mathcal{P}^{(m)}\left( 
\mathbb{R}^{n},\mathbb{R}^{D}\right) \rightarrow \mathbb{R} \\ 
\lambda _{i,\nu }\left( x_{0},\cdots ,x_{\bar{k}}\right) :\mathcal{P}^{(\bar{%
\bar{m}})}\left( \mathbb{R}^{n},\mathbb{R}^{M}\right) \rightarrow \mathbb{R}%
\text{ }(i=0,\cdots ,\bar{k})%
\end{array}%
\right\} _{\nu =1,\cdots ,\nu _{\max }^{\#}}\text{,}
\end{equation*}%
depending semialgebraically on $x_{0},x_{1},\cdots ,x_{\bar{k}}\in E$, such
that the following holds.

Let $x_{0},x_{1},\cdots ,x_{\bar{k}}\in E$. Let $P_{0}\in \mathcal{P}%
^{(m)}\left( \mathbb{R}^{n},\mathbb{R}^{D}\right) $ and let $%
P_{0}^{\#},\cdots ,P_{\bar{k}}^{\#}\in \mathcal{P}^{\left( \bar{\bar{m}}%
\right) }\left( \mathbb{R}^{n},\mathbb{R}^{M}\right) $. Then 
\begin{equation*}
\left( P_{0},P_{0}^{\#},\cdots ,P_{\bar{k}}^{\#}\right) \in H^{GR}\left(
x_{0},x_{1},\cdots ,x_{\bar{k}}\right)
\end{equation*}%
if and only if 
\begin{equation*}
\mu _{\nu }\left( x_{0},\cdots ,x_{\bar{k}}\right) \left[ P_{0}\right]
+\sum_{i=0}^{\bar{k}}\lambda _{i,\nu }\left( x_{0},\cdots ,x_{\bar{k}%
}\right) \left[ P_{i}^{\#}\right] =0\text{ for }\nu =1,\cdots ,\nu _{\max
}^{\#}.
\end{equation*}%
Hence, $\left( \ref{B9}\right) $ yields the following result.

\begin{itemize}
\item[\refstepcounter{equation}\text{(\theequation)}\label{B10}] Let $f\in
C_{0}^{\infty }\left( \mathbb{R}^{n},\mathbb{R}^{M}\right) $, and suppose $%
\mathscr{H}_{f}^{\left( l-1\right) }$ is a proper bundle. Let $x_{0}\in E$, $%
P_{0}\in \mathcal{P}^{(m)}\left( \mathbb{R}^{n},\mathbb{R}^{D}\right) $.
Then $P_{0}$ belongs to the fiber of $\mathscr{H}_{f}^{\left( l\right) }$ at 
$x_{0}$ if and only if 
\begin{equation*}
\mu _{\nu }\left( x_{0},\cdots ,x_{\bar{k}}\right) \left[ P_{0}\right]
+\sum_{i=0}^{\bar{k}}\lambda _{i,\nu }\left( x_{0},\cdots ,x_{\bar{k}%
}\right) \left[ J_{x_{i}}^{\left( \bar{\bar{m}}\right) }f\right] =0
\end{equation*}%
for $x_{1},\cdots ,x_{\bar{k}}\in E$ and $\nu =1,\cdots ,\nu _{\max }^{\#}$.
\end{itemize}

It may happen that $x_{i}=x_{j}$ in $\left( \ref{B10}\right) $ for some $%
i\not=j$. In that case, $\lambda _{i,\nu }$ and $\lambda _{j,\nu }$ play
redundant roles. We may easily rewrite $\left( \ref{B10}\right) $ as follows.

For each integer $k$ $\left( 0\leq k\leq \bar{k}\right) $, there are linear
functionals 
\begin{equation*}
\mu _{\nu }^{k}\left( x_{0},x_{1},\cdots ,x_{{k}}\right) :\mathcal{P}%
^{(m)}\left( \mathbb{R}^{n},\mathbb{R}^{D}\right) \rightarrow \mathbb{R}
\end{equation*}%
and 
\begin{equation*}
\lambda _{i,\nu }^{k}\left( x_{0},x_{1},\cdots ,x_{{k}}\right) :\mathcal{P}%
^{(\bar{\bar{m}})}\left( \mathbb{R}^{n},\mathbb{R}^{M}\right) \rightarrow 
\mathbb{R}\text{,}
\end{equation*}%
$i=0,\cdots ,k$, $\nu =1,\cdots ,\nu _{\max }^{\#}\left( k\right) $,
depending semialgebraically on $\left( x_{0},\cdots ,x_{{k}}\right) \in
E\times \cdots \times E$, for which we have the following.

\begin{itemize}
\item[\refstepcounter{equation}\text{(\theequation)}\label{B11}] Let $f\in
C_{0}^{\infty }\left( \mathbb{R}^{n},\mathbb{R}^{M}\right) $, and suppose $%
\mathscr{H}_{f}^{\left( l-1\right) }$ is a proper bundle. Let $x_{0}\in E$, $%
P_{0}\in \mathcal{P}^{(m)}\left( \mathbb{R}^{n},\mathbb{R}^{D}\right) $.
Then $P_{0}$ belongs to the fiber of $\mathscr{H}_{f}^{\left( l\right) }$ at 
$x_{0}$ if and only if, for each $0\leq k\leq \bar{k}$, for each $%
x_{1},\cdots ,x_{k}\in E\setminus \{x_{0}\}$ distinct, and for each $\nu
=1,\cdots ,\nu _{\max }^{\#}\left( {k}\right) $, we have 
\begin{equation*}
\mu _{\nu }^{k}\left( x_{0},\cdots ,x_{k}\right) \left[ P_{0}\right]
+\sum_{i=0}^{k}\lambda _{i,\nu }^{k}\left( x_{0},\cdots ,x_{k}\right) \left[
J_{x_{i}}^{\left( \bar{\bar{m}}\right) }f\right] =0.
\end{equation*}
\end{itemize}

Now let $f\in C_{0}^{\infty }\left( \mathbb{R}^{n},\mathbb{R}^{M}\right) $,
and suppose $\mathscr{H}_{f}^{\left( l-1\right) }$ is a proper bundle. Let $%
x_{0}\in E$, $P_{0}\in \mathcal{P}^{(m)}\left( \mathbb{R}^{n},\mathbb{R}%
^{D}\right) $, and suppose $P_{0}$ belongs to the fiber at $x_{0}$ of $%
\mathscr{H}_{f}^{\left( l\right) }$.

Let $\varphi $ be a scalar-valued $C_{0}^{\infty }$ function on $\mathbb{R}%
^{n}$. Then $\mathscr{H}_{\varphi f}=\varphi \odot \mathscr{H}_{f}$ (see
assumption (\ref{B1})), hence by Corollary \ref{corollary-NDP1}, we see that 
$\mathscr{H}_{\varphi f}^{\left( l-1\right) }=\varphi \odot \mathscr{H}%
_{f}^{\left( l-1\right) }$. In particular, $\mathscr{H}_{\varphi f}^{\left(
l-1\right) }$ is a proper bundle, so $\left( \ref{B11}\right) $ applies,
with $\varphi f$ in place of $f$, and with $J_{x_{0}}^{\left( m\right)
}\varphi \odot _{x_{0}}P_{0}$ in place of $P_{0}$. Note that $%
J_{x_{0}}^{\left( m\right) }\varphi \odot _{x_{0}}P_{0}$ belongs to the
fiber at $x_{0}$ of $\mathscr{H}_{\varphi f}^{\left( l\right) }=\mathscr{G}%
\left( \mathscr{H}_{\varphi f}^{\left( l-1\right) }\right) =\mathscr{G}%
\left( \varphi \odot \mathscr{H}_{f}^{\left( l-1\right) }\right) $, thanks
to Lemma \ref{Lemma-NDP1}.

Therefore, from (\ref{B11}), we obtain the following result.

Let $f\in C_{0}^{\infty }\left( \mathbb{R}^{n},\mathbb{R}^{M}\right) $, and
suppose $\mathscr{H}_{f}^{\left( l-1\right) }$ is a proper bundle. Let $%
x_{0}\in E$, $P_{0}\in \mathcal{P}^{(m)}\left( \mathbb{R}^{n},\mathbb{R}%
^{D}\right) $.

Suppose $P_{0}$ belongs to the fiber at $x_{0}$ of $\mathscr{H}_{f}^{\left(
l\right) }$.

Then, for any $\varphi \in C_{0}^{\infty }\left( \mathbb{R}^{n}\right) $,
for each $0\leq k\leq \bar{k}$, for each $x_{1},\cdots ,x_{k}\in E\setminus
\{x_{0}\}$ distinct, and for each $\nu =1,\cdots ,\nu _{\max }^{\#}\left(
k\right) $, we have 
\begin{equation}
\mu _{\nu }^{k}\left( x_{0},\cdots ,x_{k}\right) \left[ J_{x_{0}}^{\left(
m\right) }\varphi \odot _{x_{0}}P_{0}\right] +\sum_{i=0}^{k}\lambda _{i,\nu
}^{k}\left( x_{0},\cdots ,x_{k}\right) \left[ J_{x_{i}}^{\left( \bar{\bar{m}}%
\right) }\varphi \odot _{x_{i}}J_{x_{i}}^{\left( \bar{\bar{m}}\right) }f%
\right] =0.  \label{B12}
\end{equation}

Let $0\leq k\leq \bar{k}$, $x_{1},\cdots ,x_{k}\in E$ distinct (also
distinct from $x_{0}$), and $\nu =1,\cdots ,\nu _{\max }^{\#}\left( k\right) 
$ be given. For each $i=0,\cdots ,k$, we may pick $\varphi \in C_{0}^{\infty
}\left( \mathbb{R}^{n}\right) $ so that $\varphi \equiv 1$ in a neighborhood
of $x_{i}$, but $\varphi \equiv 0$ in a neighborhood of each $x_{j}$ $\left(
j\not=i\right) $. (That's why we were careful to arrange that $x_{0},\cdots
,x_{k}$ are distinct.)

If $i=0$, then (\ref{B12}) tells us that 
\begin{equation}
\mu _{\nu }^{k}\left( x_{0},\cdots ,x_{k}\right) \left[ P_{0}\right]
+\lambda _{0,\nu }^{k}\left( x_{0},\cdots ,x_{k}\right) \left[
J_{x_{0}}^{\left( \bar{\bar{m}}\right) }f\right] =0.  \label{B13}
\end{equation}%
If $1\leq i\leq {k}$, then $\left( \ref{B12}\right) $ tells us that 
\begin{equation}
\lambda _{i,\nu }^{k}\left( x_{0},\cdots ,x_{k}\right) \left[
J_{x_{i}}^{\left( \bar{\bar{m}}\right) }f\right] =0.  \label{B14}
\end{equation}

Thus, we have learned the following:

Let $f\in C_{0}^{\infty }\left( \mathbb{R}^{n},\mathbb{R}^{M}\right) $, and
suppose $\mathscr{H}_{f}^{\left( l-1\right) }$ is a proper bundle. Let $%
x_{0}\in E$, $P_{0}\in \mathcal{P}^{(m)}\left( \mathbb{R}^{n},\mathbb{R}%
^{D}\right) $.

\begin{itemize}
\item[\refstepcounter{equation}\text{(\theequation)}\label{B15}] If $P_{0}$
belongs to the fiber of $\mathscr{H}_{f}^{\left( l\right) }$ at $x_{0}$,
then $\left( \ref{B13}\right) $ holds for $0\leq k\leq \bar{k}$, $%
x_{1},\cdots ,x_{{k}}\in E\setminus \left\{ x_{0}\right\} $ distinct, and $%
\nu =1,\cdots ,\nu _{\max }^{\#}\left( k\right) ;$ and also $\left( \ref{B14}%
\right) $ holds for $1\leq k\leq \bar{k}$, $x_{1},\cdots ,x_{{k}}\in
E\setminus \left\{ x_{0}\right\} $ distinct, $i=1,\cdots ,{k}$, and $\nu
=1,\cdots ,\nu _{\max }^{\#}\left( k\right) $.
\end{itemize}

The converse of \eqref{B15} is immediate from \eqref{B11}. Therefore, we
obtain the following result, since the $\mu _{\nu }^{k}\left( x_{0},\cdots
,x_{k}\right) $ and $\lambda _{i,\nu }^{k}\left( x_{0},\cdots ,x_{k}\right) $
are semialgebraic.

There exist families of vector subspaces 
\begin{equation*}
H^{GR,0}\left( x_{0}\right) \subset \mathcal{P}^{(m)}\left( \mathbb{R}^{n},%
\mathbb{R}^{D}\right) \oplus \mathcal{P}^{(\bar{\bar{m}})}\left( \mathbb{R}%
^{n},\mathbb{R}^{M}\right)
\end{equation*}%
and%
\begin{equation*}
H^{GR,1}\left( x_{0},x\right) \subset \mathcal{P}^{(\bar{\bar{m}})}\left( 
\mathbb{R}^{n},\mathbb{R}^{M}\right) ,
\end{equation*}%
with the following properties.

\begin{itemize}
\item[\refstepcounter{equation}\text{(\theequation)}\label{B16}] $%
H^{GR,0}\left( x_{0}\right) $ depends semialgebraically on $x_{0}$; $%
H^{GR,1}\left( x_{0},x\right) $ depends semialgebraically on $x_{0}$, $x$.

\item[\refstepcounter{equation}\text{(\theequation)}\label{B17}] Let $f\in
C_{0}^{\infty }\left( \mathbb{R}^{n},\mathbb{R}^{M}\right) $, and suppose $%
\mathscr{H}_{f}^{\left( l-1\right) }$ is a proper bundle. Let $x_{0}\in E$, $%
P_{0}\in \mathcal{P}^{(m)}\left( \mathbb{R}^{n},\mathbb{R}^{D}\right) $.
Then $P_{0}$ belongs to the fiber of $\mathscr{H}_{f}^{\left( l\right) }$ at 
$x_{0}$ if and only if 
\begin{equation*}
\left( P_{0},J_{x_{0}}^{\left( \bar{\bar{m}}\right) }f\right) \in
H^{GR,0}\left( x_{0}\right)
\end{equation*}%
and 
\begin{equation*}
J_{x}^{\left( \bar{\bar{m}}\right) }f\in H^{GR,1}\left( x_{0},x\right) \text{
for all }x\in E\setminus \left\{ x_{0}\right\} \text{.}
\end{equation*}
\end{itemize}

For each $x_{0}\in E$, there exist a vector subspace $H^{GR,00}\left(
x_{0}\right) \subset \mathcal{P}^{(\bar{\bar{m}})}(\mathbb{R}^{n},\mathbb{R}%
^{M})$ and a linear map 
\begin{equation*}
T^{new}\left( x_{0}\right) :\mathcal{P}^{(\bar{\bar{m}})}\left( \mathbb{R}%
^{n},\mathbb{R}^{M}\right) \rightarrow \mathcal{P}^{(m)}\left( \mathbb{R}%
^{n},\mathbb{R}^{D}\right) ,
\end{equation*}%
both depending semialgebraically on $x_{0}$, such that the following holds:

\begin{itemize}
\item[\refstepcounter{equation}\text{(\theequation)}\label{B18}] Let $%
x_{0}\in E$ and let $P^{\#}\in \mathcal{P}^{\left( \bar{\bar{m}}\right)
}\left( \mathbb{R}^{n},\mathbb{R}^{M}\right) $. Then there exists $P_{0}\in 
\mathcal{P}^{(m)}\left( \mathbb{R}^{n},\mathbb{R}^{D}\right) $ such that $%
\left( P_{0},P^{\#}\right) \in H^{GR,0}\left( x_{0}\right) $, if and only if 
$P^{\#}\in H^{GR,00}\left( x_{0}\right) $. Moreover, if such a $P_{0}$
exists, then $\left( T^{new}\left( x_{0}\right) P^{\#},P^{\#}\right) \in
H^{GR,0}\left( x_{0}\right) $.
\end{itemize}

Indeed, (\ref{B18}) follows from linear algebra and standard properties of
semialgebraic sets and functions.

Now let $f\in C_{0}^{\infty }\left( \mathbb{R}^{n},\mathbb{R}^{M}\right) $,
and suppose $\mathscr{H}_{f}^{\left( l-1\right) }$ is a proper bundle.
According to $\left( \ref{B17}\right) $, $\mathscr{H}_{f}^{\left( l\right) }$
is a proper bundle if and only if

\begin{itemize}
\item $J_{x}^{\left( \bar{\bar{m}}\right) }f\in H^{GR,1}\left(
x_{0},x\right) $ for all $x_{0},x\in E$ with $x_{0}\not=x$, and

\item For each $x_{0}\in E$ there exists $P_{0}\in \mathcal{P}^{(m)}\left( 
\mathbb{R}^{n},\mathbb{R}^{D}\right) $ such that $\left(
P_{0},J_{x_{0}}^{\left( \bar{\bar{m}}\right) }f\right) \in H^{GR,0}\left(
x_{0}\right) $.
\end{itemize}

Therefore, \eqref{B17} and \eqref{B18} tell us the following.

Let $f\in C_{0}^{\infty }\left( \mathbb{R}^{n},\mathbb{R}^{M}\right) $, and
suppose $\mathscr{H}_{f}^{\left( l-1\right) }$ is a proper bundle. Then $%
\mathscr{H}_{f}^{\left( l\right) }$ is a proper bundle if and only if

\begin{itemize}
\item $J_{x}^{\left( \bar{\bar{m}}\right) }f\in H^{GR,1}\left(
x_{0},x\right) $ for all $x_{0},x\in E$ with $x_{0}\not=x$, and

\item $J_{x_{0}}^{\left( \bar{\bar{m}}\right) }f\in H^{GR,00}\left(
x_{0}\right) $ for all $x_{0}\in E$.
\end{itemize}

Moreover, if $\mathscr{H}_{f}^{\left( l\right) }$ is a proper bundle, then,
for each $x_{0}\in E$, $T^{new}\left( x_{0}\right) J_{x_{0}}^{\left( \bar{%
\bar{m}}\right) }f$ belongs to the fiber of $\mathscr{H}_{f}^{\left(
l\right) }$ at $x_{0};$ consequently if $\mathscr{H}_{f}^{\left( l\right) }$
is a proper bundle, then $\mathscr{H}_{f}^{\left( l\right) }=\left(
T^{new}\left( x\right) J_{x}^{\left( \bar{\bar{m}}\right) }f+I^{\left(
l\right) }\left( x\right) \right) _{x\in E}.$ (Recall that $\left( I^{\left(
l\right) }\left( x\right) \right) _{x\in E}$ is the $l^{th}$ (strong)
Glaeser refinement of the homogeneous bundle $\left( I\left( x\right)
\right) _{x\in E}$ introduced in the setup of this section.)

The above remarks imply that there exists a family of vector subspaces $%
H^{new}\left( x\right) \mathcal{\subset \mathcal{P}}^{(\bar{\bar{m}})}\left( 
\mathbb{R}^{n},\mathbb{R}^{M}\right) $, depending semialgebraically on $x\in
E$, such that the following holds.

\begin{itemize}
\item[\refstepcounter{equation}\text{(\theequation)}\label{B19}] Let $f\in
C_{0}^{\infty }\left( \mathbb{R}^{n},\mathbb{R}^{M}\right) $, and suppose $%
\mathscr{H}_{f}^{\left( l-1\right) }$ is a proper bundle. Then $\mathscr{H}%
_{f}^{\left( l\right) }$ is a proper bundle if and only if $J_{x}^{\left( 
\bar{\bar{m}}\right) }f\in H^{new}\left( x\right) $ for all $x\in E$.
Moreover, if $\mathscr{H}_{f}^{\left( l\right) }$ is a proper bundle, then $%
\mathscr{H}_{f}^{\left( l\right) }=\left( T^{new}\left( x\right)
J_{x}^{\left( \bar{\bar{m}}\right) }f+I^{\left( l\right) }\left( x\right)
\right) _{x\in E}$.
\end{itemize}

We can now produce linear functionals $\lambda _{\nu }^{new}\left( x\right) :%
\mathcal{P}^{(\bar{\bar{m}})}\left( \mathbb{R}^{n},\mathbb{R}^{M}\right)
\rightarrow \mathbb{R}$, $\nu =1,\cdots ,\nu _{\max }^{new}$, depending
semialgebraically on $x\in E$ such that 
\begin{equation}
H^{new}\left( x\right) =\left\{ P^{\#}\in \mathcal{P}^{(\bar{\bar{m}}%
)}\left( \mathbb{R}^{n},\mathbb{R}^{M}\right) :\lambda _{\nu }^{new}\left(
x\right) P^{\#}=0\text{ for all }\nu =1,\cdots ,\nu _{\max }^{new}\right\} .
\label{B20}
\end{equation}%
Setting $L_{\nu }^{new}f\left( x\right) :=\lambda _{\nu }^{new}\left(
x\right) \left[ J_{x}^{\left( \bar{\bar{m}}\right) }f\right] $, we obtain
the following from (\ref{B19}), (\ref{B20}):

\begin{itemize}
\item Each $L_{\nu }^{new}$ is a linear differential operator, with
semialgebraic coefficients, carrying functions in $C^{\infty }\left( \mathbb{%
R}^{n},\mathbb{R}^{M}\right) $ to scalar-valued functions.

\item Let $f\in C_{0}^{\infty }\left( \mathbb{R}^{n},\mathbb{R}^{M}\right) $%
. Then $\mathscr{H}_{f}^{\left( l\right) }$ is a proper bundle if and only
if $\mathscr{H}_{f}^{\left( l-1\right) }$ is a proper bundle and $L_{\nu
}^{new}f=0$ for $\nu =1,\cdots ,\nu _{\max }^{new}$.

\item Let $f\in C_{0}^{\infty }\left( \mathbb{R}^{n},\mathbb{R}^{M}\right) $%
. Suppose $\mathscr{H}_{f}^{\left( l\right) }$ is a proper bundle. Then 
\begin{equation*}
\mathscr{H}_{f}^{\left( l\right) }=\left( T^{new}\left( x\right)
J_{x}^{\left( \bar{\bar{m}}\right) }f+I^{\left( l\right) }\left( x\right)
\right) _{x\in E}.
\end{equation*}
\end{itemize}

Recall that $T^{new}\left( x\right) $ depends semialgebraically on $x\in E$.
The above bullet points and (\ref{B2}) together imply the conclusions of
Lemma \ref{Main-Lemma-onHf}.

This completes our induction on $l$, proving Lemma \ref{Main-Lemma-onHf}.
\end{proof}

\section{Proof of Theorem \protect\ref{statement-main-theore}}\label{proofofmain}

\begin{proof} In proving Theorem \ref{statement-main-theore}, we may suppose that $f$ has
compact support, simply because $E$ is compact.

Thus, let $E\subset \mathbb{R}^{n}$ be a compact semialgebraic set, and let $%
M,N\geq 1$. Suppose we are given a matrix $\left( A_{ij}\right) _{1\leq
i\leq N,1\leq j\leq M}$ of semialgebraic functions on $E$. Suppose further
we are given $f_{i}\in C_{0}^{\infty }(\mathbb{R}^{n})$, $i=1,\cdots ,N$.
\bigskip

We want to solve the equations 
\begin{equation}
\sum_{j=1}^{M}A_{ij}\left( x\right) F_{j}\left( x\right) =f_{i}\left(
x\right) \text{ on }E\text{ }\left( i=1,\cdots ,N\right)  \label{S1}
\end{equation}%
for unknown functions $F_{1},\cdots ,F_{M}\in C^{m}\left( \mathbb{R}%
^{n}\right) $.

We write $f=\left( f_{1},\cdots ,f_{N}\right) \in C_{0}^{\infty }\left( 
\mathbb{R}^{n},\mathbb{R}^{N}\right) $.

We reinterpret (\ref{S1}) in terms of bundles. For $x\in E$, let $I\left(
x\right) \subset \mathcal{P}^{(m)}\left( \mathbb{R}^{n},\mathbb{R}%
^{M}\right) $ be the set of all $P=\left( P_{1},\cdots ,P_{M}\right) \in 
\mathcal{P}^{(m)}\left( \mathbb{R}^{n},\mathbb{R}^{M}\right) $ such that 
\begin{equation}
\sum_{j=1}^{M}A_{ij}\left( x\right) P_{j}\left( x\right) =0\text{ for }%
i=1,\cdots ,N.  \label{S2}
\end{equation}%
Note that (\ref{S2}) involves only the values of the $P_{j}$ at $x$, and
thus discards the higher-order information encoded in $P_{j}$.

Note also that $I(x)$ depends semialgebraically on $x$. 

Observe that $I\left( x\right) $ is an $\mathscr{R}_{x}^{m}$-submodule of $%
\mathcal{P}^{(m)}\left( \mathbb{R}^{n},\mathbb{R}^{M}\right) $. Indeed, if $%
P=\left( P_{1},\cdots ,P_{M}\right) \in I\left( x\right) $, and if $Q\in %
\mathscr{R}_{x}^{m}$, then 
\begin{equation*}
Q\odot _{x}P=\left( Q\odot _{x}P_{1},\cdots ,Q\odot _{x}P_{M}\right) ,
\end{equation*}%
and 
\begin{equation*}
\left( Q\odot _{x}P_{j}\right) \left( x\right) =Q\left( x\right) P_{j}\left(
x\right) \text{.}
\end{equation*}%
Therefore, condition (\ref{S2}) for $P$ implies condition (\ref{S2}) for $%
Q\odot _{x}P$.

\begin{itemize}
\item[\refstepcounter{equation}\text{(\theequation)}\label{S3}] For each $%
x\in E$, let $\Pi \left( x\right) $ denote the orthogonal projection from $%
\mathbb{R}^{N}$ onto the range of the matrix $\left( A_{ij}\left( x\right)
\right) $.

\item[\refstepcounter{equation}\text{(\theequation)}\label{S4}] For $\left(
\xi _{1},\cdots ,\xi _{N}\right) \in \mathbb{R}^{N}$, let $T^{o}\left(
x\right) \left[ \xi _{1},\cdots ,\xi _{N}\right] $ denote the vector $\left(
\eta _{1},\cdots ,\eta _{M}\right) \in \mathbb{R}^{M}$ that solves the
equation 
\begin{equation*}
\left( A_{ij}\left( x\right) \right) \left( 
\begin{array}{c}
\eta _{1} \\ 
\vdots \\ 
\eta _{M}%
\end{array}%
\right) =\Pi \left( x\right) \left( 
\begin{array}{c}
\xi _{1} \\ 
\vdots \\ 
\xi _{N}%
\end{array}%
\right) \text{,}
\end{equation*}%
with the (Euclidean) norm of $\left( \eta _{1},\cdots ,\eta _{M}\right) $ as
small as possible.
\end{itemize}

Thus, $\Pi \left( x\right) $ and $T^{o}\left( x\right) $ are matrices that
depend semialgebraically on $x\in E$.

\begin{itemize}
\item[\refstepcounter{equation}\text{(\theequation)}\label{S5}] For any $%
P=\left( P_{1},\cdots ,P_{N}\right) \in \mathcal{P}^{(m)}\left( \mathbb{R}%
^{n},\mathbb{R}^{N}\right) $, and any $x\in E$, let $T\left( x\right) \left[
P\right] $ be the vector of constant polynomials given by $T^{o}\left(
x\right) \left[ P_{1}\left( x\right) ,\cdots ,P_{N}\left( x\right) \right] $.
\end{itemize}

We will prove the following two facts.

\begin{lemma}
\label{Lemma-S1}Let $f=\left( f_{1},\cdots ,f_{N}\right) \in C_{0}^{\infty
}\left( \mathbb{R}^{n},\mathbb{R}^{N}\right) $ be given. A function $%
F=\left( F_{1},\cdots ,F_{M}\right) \in C^{m}\left( \mathbb{R}^{n},\mathbb{R}%
^{M}\right) $ solves equations $\left( \ref{S1}\right) $ if and only if it
is a section of the bundle 
\begin{equation}
\mathscr{H}_{f}=\left( T\left( x\right) J_{x}^{\left( m\right) }f+I\left(
x\right) \right) _{x\in E}  \label{S6}
\end{equation}%
and $\left( I-\Pi \left( x\right) \right) f\left( x\right) =0$ for all $x\in
E$.
\end{lemma}

\begin{lemma}
\label{Lemma-S2}For $f=\left( f_{1},\cdots ,f_{N}\right) \in C_{0}^{\infty
}\left( \mathbb{R}^{n},\mathbb{R}^{N}\right) $, let $\mathscr{H}_{f}$ be
defined by $\left( \ref{S6}\right) $. Then for $\varphi \in C_{0}^{\infty
}\left( \mathbb{R}^{n}\right) $ and $f=\left( f_{1},\cdots ,f_{N}\right) \in
C_{0}^{\infty }\left( \mathbb{R}^{n},\mathbb{R}^{N}\right) $ we have 
\begin{equation*}
\mathscr{H}_{\varphi f}=\varphi \odot \mathscr{H}_{f}\text{.}
\end{equation*}
\end{lemma}

\begin{proof}[Proof of Lemma \protect\ref{Lemma-S1}]
If $\left( I-\Pi \left( x\right) \right) f\left( x\right) \not=0$ for some $%
x\in E$, then $\left( f_{1}\left( x\right) ,\cdots ,f_{N}\left( x\right)
\right) $ doesn't belong to the range of $\left( A_{ij}\left( x\right)
\right) $, so obviously equations (\ref{S1}) have no solution. Hence, we may
assume that 
\begin{equation*}
\left( I-\Pi \left( x\right) \right) f\left( x\right) =0
\end{equation*}
for all $x\in E$.

We then have (by (\ref{S4})): 
\begin{equation*}
\left( A_{ij}\left( x\right) \right) T^{o}\left( x\right) \left( 
\begin{array}{c}
f_{1}\left( x\right) \\ 
\vdots \\ 
f_{N}\left( x\right)%
\end{array}%
\right) =\left( 
\begin{array}{c}
f_{1}\left( x\right) \\ 
\vdots \\ 
f_{N}\left( x\right)%
\end{array}%
\right)
\end{equation*}%
for each $x\in E$, hence 
\begin{equation*}
\left( A_{ij}\left( x\right) \right)\left\{ \left[ T\left( x\right) \left( 
\begin{array}{c}
J_{x}^{\left( m\right) }f_{1} \\ 
\vdots \\ 
J_{x}^{\left( m\right) }f_{N}%
\end{array}%
\right) \right] \left( x\right) \right\} =\left( 
\begin{array}{c}
f_{1}\left( x\right) \\ 
\vdots \\ 
f_{N}\left( x\right)%
\end{array}%
\right) \text{.}
\end{equation*}%
Therefore, for $F=\left( F_{1},\cdots ,F_{M}\right) \in C^{m}\left( \mathbb{R%
}^{n},\mathbb{R}^{M}\right) $, we have 
\begin{equation*}
\left( A_{ij}\left( x\right) \right) \left( F\left( x\right) \right) =\left( 
\begin{array}{c}
f_{1}\left( x\right) \\ 
\vdots \\ 
f_{N}\left( x\right)%
\end{array}%
\right)
\end{equation*}%
if and only if 
\begin{equation*}
\left( A_{ij}\left( x\right) \right) \left[ F\left( x\right) -\left\{T\left(
x\right) \left( 
\begin{array}{c}
J^{(m)}_xf_{1} \\ 
\vdots \\ 
J^{(m)}_xf_{N}%
\end{array}%
\right)\right\}(x)\right] =0\text{, i.e.,}
\end{equation*}%
\begin{equation*}
J_{x}^{\left( m\right) }F-T\left( x\right) J_{x}^{\left( m\right) }f\in
I\left( x\right)
\end{equation*}%
(see the definition (\ref{S2}) of $I\left( x\right) $).

That is, whenever $\left( I-\Pi \left( x\right) \right) f\left( x\right) =0$
for all $x\in E$, we find that $F$ solves $\left( \ref{S1}\right) $ if and
only if 
\begin{equation*}
J_{x}^{\left( m\right) }F\in T\left( x\right) J_{x}^{\left( m\right)
}F+I\left( x\right) \text{ for all }x\in E\text{, i.e.,}
\end{equation*}%
if and only if $F$ is a section of $\mathscr{H}_{f}$; see $\left( \ref{S6}%
\right) $. The proof of Lemma \ref{Lemma-S1} is complete.
\end{proof}

\begin{proof}[Proof of Lemma \protect\ref{Lemma-S2}]
We have 
\begin{equation*}
\mathscr{H}_{f}=\left( T\left( x\right) J_{x}^{\left( m\right) }f+I\left(
x\right) \right) _{x\in E},
\end{equation*}%
so 
\begin{equation*}
\varphi \odot \mathscr{H}_{f}=\left( J_{x}^{\left( m\right) }\varphi \odot
_{x}\left\{ T\left( x\right) J_{x}^{\left( m\right) }f\right\} +I\left(
x\right) \right) _{x\in E}
\end{equation*}%
by definition of $\odot $. Also, by definition, 
\begin{equation*}
\mathscr{H}_{\varphi f}=\left( T\left( x\right) J_{x}^{\left( m\right)
}\left( \varphi F\right) +I\left( x\right) \right) _{x\in E}\text{.}
\end{equation*}

To establish Lemma \ref{Lemma-S2}, we must therefore show that 
\begin{equation*}
J_{x}^{\left( m\right) }\varphi \odot _{x}\left\{ T\left( x\right)
J_{x}^{\left( m\right) }f\right\} -T\left( x\right) J_{x}^{\left( m\right)
}\left( \varphi f\right) \in I\left( x\right) \text{ for all }x\in E\text{.}
\end{equation*}%
By the definition of $I\left( x\right) $, this means that 
\begin{equation}
\sum_{j=1}^{M}A_{ij}\left( x\right) \left[ \varphi \left( x\right) \left\{
T\left( x\right) J_{x}^{\left( m\right) }f\right\} \left( x\right) -\left(
T\left( x\right) J_{x}^{\left( m\right) }\left( \varphi f\right) \right)
\left( x\right) \right] _{j}=0\text{.}  \label{S7}
\end{equation}%
We will check the stronger result that the expression in square
brackets is zero.

In fact, $T\left( x\right) J_{x}^{\left( m\right) }f$ is the vector of
constant polynomials $T^{o}\left( x\right) \left( 
\begin{array}{c}
f_{1}\left( x\right) \\ 
\vdots \\ 
f_{N}\left( x\right)%
\end{array}%
\right) $ by definition $\left( \ref{S5}\right) $.

Hence, $\varphi \left( x\right) \left\{ T\left( x\right) J_{x}^{\left(
m\right) }f\right\} \left( x\right) =\varphi \left( x\right) \cdot
T^{o}\left( x\right) \left( 
\begin{array}{c}
f_{1}\left( x\right) \\ 
\vdots \\ 
f_{N}\left( x\right)%
\end{array}%
\right) $.

Another application of definition $\left( \ref{S5}\right) $ yields 
\begin{equation*}
\left(T\left( x\right) J_{x}^{\left( m\right) }\left( \varphi f\right)
\right) \left( x\right) =T^{o}\left( x\right) \left( 
\begin{array}{c}
\varphi \left( x\right) f_{1}\left( x\right) \\ 
\vdots \\ 
\varphi \left( x\right) f_{N}\left( x\right)%
\end{array}%
\right) \text{.}
\end{equation*}%
Consequently, the expression in square brackets in $\left( \ref{S7}\right) $
is equal to 
\begin{equation*}
\varphi \left( x\right) T^{o}\left( x\right) \left( 
\begin{array}{c}
f_{1}\left( x\right) \\ 
\vdots \\ 
f_{N}\left( x\right)%
\end{array}%
\right) -T^{o}\left( x\right) \left( 
\begin{array}{c}
\varphi \left( x\right) f_{1}\left( x\right) \\ 
\vdots \\ 
\varphi \left( x\right) f_{N}\left( x\right)%
\end{array}%
\right) \text{,}
\end{equation*}%
which equals zero, as promised. The proof of Lemma \ref{Lemma-S2} is
complete.
\end{proof}

Thanks to Lemma \ref{Lemma-S2}, our bundle $\mathscr{H}_{f}$ in $\left( \ref%
{S6}\right) $ satisfies the assumptions made in Section \ref{BDSF}.

Hence, we may apply Lemma \ref{Main-Lemma-onHf}.

We take $l$ in Lemma \ref{Main-Lemma-onHf} to equal the large constant $%
l_{\ast }$ from Theorem \ref{CK-Theorem}.

We can then argue as follows.

Let $f=\left( f_{1},\cdots ,f_{N}\right) \in C_{0}^{\infty }\left( \mathbb{R}%
^{n},\mathbb{R}^{N}\right) $ be given. Then, by Lemma \ref{Lemma-S1}, the
equations (\ref{S1}) admit a $C^{m}$ solution $\left( F_{1},\cdots
,F_{M}\right) $ if and only if 
\begin{equation*}
\left( I-\Pi \left( x\right) \right) f\left( x\right) =0
\end{equation*}%
for all $x\in E$ and $\mathscr{H}_{f}$ has a section.

By Theorem \ref{CK-Theorem}, this holds if and only if $\left( I-\Pi \left(
x\right) \right) f\left( x\right) =0$ for all $x\in E$ and $\mathscr{G}%
^{\left( l_{\ast }\right) }\mathscr{H}_{f}$ is a proper bundle.

By Lemma \ref{Main-Lemma-onHf}, this in turn holds if and only if $\left(
I-\Pi \left( x\right) \right) f\left( x\right) =0$ for all $x\in E$ and $%
L_{\nu }f=0$ on $E$ for $\nu =1,\cdots ,\nu _{\max }$, where each $L_{\nu }$
is a linear partial differential operator with semialgebraic coefficients,
mapping functions in $C^{\infty }\left( \mathbb{R}^{n},\mathbb{R}^{N}\right) 
$ to scalar-valued functions on $\mathbb{R}^{n}$.

Since the equation $\left( I-\Pi \left( x\right) \right) f\left( x\right) =0$
on $E$ is also a system of such linear partial differential equations (of
order 0), the proof of Theorem \ref{statement-main-theore} is complete. \end{proof}

\section{Proof of Theorem \protect\ref{theorem1}}

\label{passtononcompact}

\begin{proof} Let $U$ be a bounded open semialgebraic subset of $\mathbb{R}^{n}$ and let $%
\left( A_{ij}(x)\right) _{1\leq i\leq N,1\leq j\leq M}$ be a matrix of
semialgebraic functions on $\mathbb{R}^{n}$.

According to Theorem \ref{statement-main-theore}, there exist linear partial
differential operators $L_{\nu }$ $(\nu =1,\cdots ,\nu _{\max })$ with
semialgebraic coefficients, such that given $f\in C^{\infty }(\mathbb{R}^{n},%
\mathbb{R}^{N})$ there exists $F\in C^{m}(\mathbb{R}^{n},\mathbb{R}^{M})$
such that $\sum_{j=1}^{M}A_{ij}(x)F_{j}(x)=f_{i}(x)$ $(i=1,\cdots ,N)$, all $%
x\in U^{\text{closure}}$ if and only if $L_{\nu }f=0$ on $U^{\text{closure}}$%
, all $\nu $.

Let $\bar{m}$ be greater than or equal to the order of each $L_{\nu }$.

Now, suppose $F\in C_{loc}^{m}\left(U,\mathbb{R}^{M}\right) $, 
$f\in C^{\infty }\left( U,\mathbb{R}^{N}\right) $, and $\sum_{j=1}^MA_{ij}%
\left( x\right) F_{j}\left( x\right) =f_{i}\left( x\right) $ $\left(
i=1,\cdots ,N\right) $ on $U$.

Fix $x_{0}\in U$, and let $\theta \in C_{0}^{\infty }\left( U\right) $.

Then $\theta F\in C^{m}\left( \mathbb{R}^{n},\mathbb{R}^{M}\right) $, $%
\theta f\in C^{\infty }\left( \mathbb{R}^{n},\mathbb{R}^{N}\right) $, and $%
\sum_{j=1}^MA_{ij}\left( x\right) \left( \theta F_{j}\right) \left( x\right)
=\left( \theta f_{i}\right) \left( x\right) $ $\left( i=1,\cdots ,N\right) $
on $U^{\text{closure}}$.

Consequently, $L_{\nu }\left( \theta f\right) (x_0)=0$ for all $\nu $.

Given any $P\in \mathcal{P}^{\left( \bar{m}\right) }$ there exists $\theta
\in C_{0}^{\infty }\left( U\right) $ such that $J_{x_{0}}^{\left( \bar{m}%
\right) }\theta =P$, hence $L_{\nu }\left( \theta f\right) \left(
x_{0}\right) =L_{\nu }\left( Pf\right) \left( x_{0}\right) $.

Thus, $L_{\nu }\left( Pf\right) =0$ on $U$, all $P\in \mathcal{P}^{\left( 
\bar{m}\right) }$, i.e.,

\begin{itemize}
\item[\refstepcounter{equation}\text{(\theequation)}\label{++1}] {$L_{\nu
}\left( x^{\gamma }f\right) =0$ on $U$ for all $\left\vert \gamma
\right\vert \leq \bar{m}$ and each $\nu $.}
\end{itemize}

Thus, if $f\in C^{\infty }\left( U,\mathbb{R}^{N}\right) $ admits a solution 
$F\in C_{loc}^{m}\left( U,\mathbb{R}^{M}\right) $ of $%
\sum_{j=1}^{M}A_{ij}F_{j}=f_{i}$ ($i=1,\cdots ,N$) on $U$, then (\ref{++1})
holds.

Conversely, suppose (\ref{++1}) holds on $U$. Then $L_{\nu }\left( Pf\right)
\left( x_{0}\right) =0$ for any $x_{0}\in U$, $P\in \mathcal{P}^{\left( \bar{%
m}\right) },$ $\nu =1,\cdots ,\nu _{\max }$.

Let $\theta \in C^{\infty }\left( U\right) $. Then $L_{\nu }\left( \theta
f\right) \left( x_{0}\right) =L_{\nu }\left( \left[ J_{x_{0}}^{\left( \bar{m}%
\right) }\theta \right] f\right) \left( x_{0}\right) =0,$ for any $x_0 \in U$; i.e., $L_{\nu
}\left( \theta f\right) =0$ on $U$.

\begin{lemma}
\label{addendumlemma}Given $f_{1},\cdots ,f_{N}\in C^{\infty }\left(
U\right) $, there exists $\theta \in C^{\infty }\left( U\right) $ such that $%
\theta >0$ on $U$, $\mathbb{I}_{U}\theta f_{i}\in C^{\infty }\left( \mathbb{R%
}^{n}\right) $ for each $i$, and $J_{x}^{\left( \bar{m}\right) }\left( 
\mathbb{I}_{U}\theta f_{i}\right) =0$ for each $x\in \partial U$. Here $%
\mathbb{I}_{U}$ is the indicator function of $U$.
\end{lemma}

\begin{proof}
Fix cutoff functions $\varphi _{\nu }\left( x\right) $ $\left( \nu
=1,2,3,\cdots \right) $ with the following properties.

\begin{itemize}
\item Each $\varphi _{\nu }$ is a nonnegative $C_{0}^{\infty }$ function on $%
\mathbb{R}^{n}$.

\item Supp $\left( \varphi _{\nu }\right) \subset \subset U$ for each $\nu $.

\item For any $x\in U$ we have $\varphi _{\nu }\left( x\right) >0$ for some $%
\nu $.
\end{itemize}

For instance, we may take $\left\{ \varphi _{\nu }\right\} $ to be the
Whitney partition of unity, associated to the decomposition of $U$ into
Whitney cubes. (See \cite{stein-little-book}.)

We then fix a sequence of positive numbers $\tau _{\nu }$ $\left( \nu
=1,2,3,\cdots \right) $ with the following properties.

\begin{itemize}
\item $\tau _{\nu }\cdot \left\vert \partial ^{\alpha }\varphi _{\nu }\left(
x\right) \right\vert \leq 2^{-\nu }$ for $\left\vert \alpha \right\vert \leq
\nu ,$ $x\in \mathbb{R}^{n}$.

\item $\tau _{\nu }\cdot \left\vert \partial ^{\alpha }\left( \varphi _{\nu
}f_{i}\right) \left( x\right) \right\vert \leq 2^{-\nu }$ for $\left\vert
\alpha \right\vert \leq \nu ,$ $x\in \mathbb{R}^{n}$, $i=1,\cdots ,N$.
\end{itemize}

Such $\tau _{\nu }$ exist because $\varphi _{\nu }$, $f_{i}\in C^{\infty
}\left( U\right) $ and supp $\varphi _{\nu }\subset \subset U$.

One checks easily that 
\begin{equation*}
\theta =\sum_{\nu =1}^{\infty }\tau _{\nu }\varphi _{\nu }
\end{equation*}%
has all the properties asserted in the lemma.
\end{proof}

Picking $\theta $ as in Lemma \ref{addendumlemma}, we see that

\begin{itemize}
\item[\refstepcounter{equation}\text{(\theequation)}\label{+2}] {$L_{\nu
}\left( \mathbb{I}_{U}\theta f_{i}\right) =0$ on $U^{\text{closure}}$.}
\end{itemize}

Indeed (\ref{+2}) holds in $U$, and it holds on $\partial U$ because $%
J_{x}^{\left( \bar{m}\right) }\left( \mathbb{I}_{U}\theta f_{i}\right) =0$
for $x\in \partial U$, while $L_{\nu }$ has order $\leq \bar{m}$.

Because $\mathbb{I}_{U}\theta f$ belongs to $C^{\infty }\left( \mathbb{R}%
^{n},\mathbb{R}^{N}\right) $ and is annihilated by the $L_{\nu }$ on $U^{%
\text{closure}}$, there exists $\tilde{F}=\left( \tilde{F}_{1},\cdots ,%
\tilde{F}_{M}\right) \in C^{m}\left( \mathbb{R}^{n},\mathbb{R}^{M}\right) $
such that $\sum_{j=1}^{M}A_{ij}\tilde{F}_{j}=\mathbb{I}_{U}\theta f_{i}$ on $%
U^{\text{closure}}$ ($i=1,\cdots, N$). Setting $F_{j}=\tilde{F}_{j}/\theta $
on $U$ (recall, $\theta >0$ on $U,$ $\theta \in C^{\infty }\left( U\right) $%
), we have $F_{j}\in C_{loc}^{m}\left( U\right) $ and $\sum_{j=1}^M A_{ij}
F_j = f_i $ on $U$ ($i = 1, \cdots, N$).

So we have proven the following corollary of Theorem \ref%
{statement-main-theore}.

\begin{corollary}
\label{cor2}Let $U$ be a bounded open semialgebraic subset of $\mathbb{R}%
^{n} $, and let $\left( A_{ij}\left( x\right) \right) _{1\leq i\leq N,1\leq
j\leq M}$ be a matrix of semialgebraic functions on $U$. Then there exist
linear partial differential operators $L_{1},\cdots ,L_{\nu _{\max }}$ with
semialgebraic coefficients such that

\begin{itemize}
\item Each $L_{\nu }$ maps vectors $f$ of smooth functions to scalar-valued
functions.

\item Let $f=\left( f_{1},\cdots ,f_{N}\right) \in C^{\infty }\left( U,%
\mathbb{R}^{N}\right) $. Then the equations%
\begin{equation*}
\sum_{j=1}^{M}A_{ij}\left( x\right) F_{j}\left( x\right) =f_{i}\left(
x\right) \text{ on }U\left( i=1,\cdots ,N\right)
\end{equation*}%
admit a solution $F_{1},\cdots ,F_{N}\in C_{loc}^{m}\left( U\right) $ if and
only if $L_{\nu }f=0$ on $U$ for $\nu =1,\cdots ,\nu _{\max }$.
\end{itemize}
\end{corollary}

Finally, note that $\mathbb{R}^{n}$ is semialgebraically diffeomorphic to
the open unit cube $U=\left( -1,1\right) ^{n}$ via the map 
\begin{equation*}
U\ni \left( x_{1},\cdots ,x_{n}\right) \mapsto \left( \frac{x_{1}}{%
1-x_{1}^{2}},\cdots ,\frac{x_{n}}{1-x_{n}^{2}}\right) \in \mathbb{R}^{n}%
\text{.}
\end{equation*}

Theorem \ref{theorem1} now follows at once from Corollary \ref{cor2}.
(Recall that our notation has changed; the function space called $C^m(%
\mathbb{R}^n, \mathbb{R}^M)$ in the Introduction is now called $C^m_{loc}(%
\mathbb{R}^n, \mathbb{R}^M)$.) \end{proof}

As promised, all the semialgebraic sets and functions introduced in  
Sections \ref{sec3},...,\ref{passtononcompact} above can be computed using the results and techniques  
presented in Sections \ref{sec2.5} and \ref{sec2.6}. Therefore, in principle, we can  
compute the partial differential operators appearing in Theorem \ref{theorem1}.

\bibliographystyle{plain}
\bibliography{papers}

\end{document}